\documentclass[10pt]{amsart}
\usepackage{amsthm, amsmath, hyperref, cleveref, amsfonts, amssymb, enumerate, textcmds, tensor, enumitem, mathdots}
\usepackage{pst-node}
\usepackage{tikz-cd} 
\usepackage{graphicx,tikz}
\usepackage{enumerate}
\usepackage{pinlabel}
\usepackage[colorinlistoftodos]{todonotes}
\usepackage[normalem]{ulem}
\usepackage{commath}
\usepackage[Bourbaki-arrow]{dynkin-diagrams}
\usepackage{placeins}
\usepackage{mathtools}

\let\svthefootnote\thefootnote
\newcommand\freefootnote[1]{%
  \let\thefootnote\relax%
  \footnotetext{#1}%
  \let\thefootnote\svthefootnote%
}

\newcommand{\R}{\mathbb R}

\newtheorem{thm}{Theorem}[section]
\newtheorem{lem}[thm]{Lemma}
\newtheorem{prop}[thm]{Proposition}
\newtheorem{cor}[thm]{Corollary}
\newtheorem{que}[thm]{Question}

\DeclareMathOperator{\C}{\mathbb{C}}

\allowdisplaybreaks

\theoremstyle{definition}
\newtheorem{defn}[thm]{Definition}
\newtheorem{lemma}[thm]{Lemma}

\newtheorem{remark}[thm]{Remark}

\begin{document}
\title{Holomorphic realizations of pairs of foliations on Riemann surfaces}
\author{Nathaniel Sagman}
\address{Nathaniel Sagman, University of North Carolina at Chapel Hill, NC, United States}
\email{nsagman@unc.edu}
\author{Dragomir \v Sari\' c}
\address{Dragomir \v Sari\' c, Mathematics PhD. Program, Graduate Center of the City University of New York, 365 Fifth Avenue, New York, NY 10016-4309}
\address{Dragomir \v Sari\' c, Department of Mathematics, Queens College of the City University of New York, 65--30 Kissena Blvd., Flushing, NY 11367}
\email{Dragomir.Saric@qc.cuny.edu}

\begin{abstract}
Let $X$ be a hyperbolic Riemann surface and let $\mu$ and $\nu$ be laminations on $X$ homotopic to measured foliations with finite Dirichlet integral. We prove that $\mu$ and $\nu$ are filling if and only if there exists a homeomorphism to another Riemann surface $f:X\to Y$ and an integrable holomorphic quadratic differential $q$ on $Y$, unique up to the natural equivalence, such that the push-forward laminations are homotopic to the horizontal and vertical foliations of $q$ respectively. This extends a classical theorem of Gardiner-Masur from closed surfaces to arbitrary surfaces. As well, the dual $\R$-tree interpretation yields the solution of an asymptotic Plateau problem for minimal surfaces in a product of two $\R$-trees. We construct examples such that $f:X\to Y$ is not homotopic to a quasiconformal map, and we present sufficient conditions that ensure it is.  We deduce applications to main inequalities for locally quasiconformal maps, harmonic maps between surfaces, and big mapping class groups.
\end{abstract}
\maketitle
\tableofcontents

\section{Introduction}

A result of Hubbard-Masur states that on a compact Riemann surface $S$ of genus at least two, for any measured foliation $\mathcal{F}$ there exists a unique holomorphic quadratic differential $q$ on $S$ whose horizontal measured foliation $\mathcal{F}_h(q)$ is homotopic to $\mathcal{F}$ (see \cite{HMmain} and also Kerckhoff \cite{Ker} and Wolf \cite{Wolf} for alternative proofs). When $X$ is an arbitrary Riemann surface with a conformally hyperbolic metric (possibly of infinite area), one considers a generalization of singular measured foliations called partial measured foliations. Every partial measured foliation $\mathcal{F}$ on a Riemann surface has a Dirichlet integral $\mathcal{D}(\mathcal{F})$, which, when the foliation arises from a quadratic differential $q$, is the $L^1$ norm of $q$ and the extremal length (see \cite{Ker}, \cite{SaricShima}). In \cite{Saric24,Saric25}, the second named author extended the Hubbard-Masur theorem to arbitrary surfaces: it was proved that a proper measured foliation $\mathcal{F}$ is homotopic to the horizontal foliation of a unique $L^1$ holomorphic quadratic differential if and only if the Dirichlet integral of $\mathcal{F}$ is finite. 

A result of Gardiner-Masur \cite{GarMas} says that on a compact Riemann surface $S$ of genus at least two, any pair of filling measured foliations $\mathcal{F}_1$ and $\mathcal{F}_2$ can be realized by the horizontal and vertical foliations of a holomorphic quadratic differential $q_1$ on a unique Riemann surface $S_1$ that is quasiconformal to $S$. In other words, the Teichm\"uller space $T(S)$ contains a unique point at which this holomorphic realization $q_1$ of $\mathcal{F}_1$ and $\mathcal{F}_2$ is defined. In fact, they showed that there is a unique Teichm\"uller geodesic on which the product of extremal lengths of $\mathcal{F}_1$ and $\mathcal{F}_2$ is minimized. Wentworth \cite{Wen} gave another proof in the case of a compact surface where the point in $T(S)$ which supports $q_1$ is the unique minimum over $T(S)$ of the sum of the extremal lengths of $\mathcal{F}_1$ and $\mathcal{F}_2$ (also called the energy functional). 

The main result of this paper is the extension of the Gardiner-Masur theorem in the full generality of any Riemann surface with a conformally hyperbolic metric. 
Our result is best stated in terms of laminations; there is a straightening map that takes proper partial measured foliations to measured geodesic laminations, and if a foliation $\mathcal{F}$ straightens to a lamination $\mu$, we say that they are homotopic. Using the straightening, one can rephrase the results of Hubbard-Masur and Gardiner-Masur in terms of laminations.


In the context of Gardiner-Masur, a pair of measured geodesic laminations $\mu$ and $\nu$ is called filling if for any third measured lamination $\xi$, the intersection numbers satisfy $i(\mu,\xi)+i(\nu,\xi)>0$. To properly formulate a realization theorem for infinite area surfaces, one has to first decide upon a definition of filling, and then prove that the horizontal and vertical foliations of an integrable quadratic differential straighten to filling laminations. Our definition is $i(\mu,\gamma)+i(\nu,\gamma)>0$ for every complete simple geodesic $\gamma$.

\begin{thm}\label{thm: main}
    Let $X$ be any Riemann surface and let $\mu$ and $\nu$ be measured geodesic laminations on $X$ homotopic to partial measured foliations with finite Dirichlet integral. 

Then, $\mu$ and $\nu$ are filling if and only if there exists a homeomorphism to another Riemann surface $f:X\to Y$, whose lift to the universal covers extends to a homeomorphism of the ideal boundaries, and an integrable holomorphic quadratic differential $q$ on $Y$ such that the push-forward laminations $f_*\mu$ and $f_*\nu$ are homotopic to the horizontal and vertical foliations of $q$ respectively. The data $(f,Y,q)$ is unique up to isomorphism, and the corresponding pair of measured foliations is energy minimizing. 
\end{thm}
If the conditions of the theorem are met, we say that $(\mu,\nu)$ is \textit{holomorphically realizable}, and the data $(f,Y,q)$ is referred to as the holomorphic realization. In the theorem, two realizations $(f,Y,q)$ and $(g,Z,r)$ being isomorphic means that $h:=g\circ f^{-1}$ is homotopic relative to the ideal boundaries to a biholomorphism that pushes $q$ forward to $r$. The map $f$ need not be homotopic to a quasiconformal map (more on this below). In Theorem \ref{thm: main}, it is interesting that even in this rather wild setting where we lack analytic controls, the Riemann surface $Y$ and the holomorphic differential $q$ are unique.

\textit{Energy minimizing} for the pair of horizontal and vertical foliations $(\mathcal{F}_h(q),\mathcal{F}_v(q))$ means that if $g:X\to Z$ is a homeomorphism whose lift extends to the ideal boundaries, and $\mathcal{F}$ and $\mathcal{G}$ are foliations homotopic to $g_*\mu$ and $g_*\nu$ respectively, then $$\mathcal{D}(\mathcal{F}_h(q))+\mathcal{D}(\mathcal{F}_v(q))\leq \mathcal{D}(\mathcal{F})+\mathcal{D}(\mathcal{G}),$$ with equality if and only if $\mathcal{F}$ and $\mathcal{G}$ come from an isomorphic triple $(Z,g,r)$. Since $g$ is just a homeomorphism, the right-hand side might be infinite. This is in fact equivalent to the line $ (t\mathcal{F}_h(q),t^{-1}\mathcal{F}_v(q)), t>0,$ minimizing the product functional $(\mathcal{F},\mathcal{G})\mapsto \mathcal{D}(\mathcal{F})\mathcal{D}(\mathcal{G})$ (see \S \ref{sec: without R trees}). Hence, we have both minimization results as in \cite{GarMas} and \cite{Wen}. Another viewpoint is that $Y$ and $(\mathcal{F}_h(q),\mathcal{F}_v(q))$ yield a strict lower bound for certain energy functionals over every possible (quasiconformal) Teichm{\"u}ller space $T_{qc}(Z)$ (see below), where $Z$ is homeomorphic to $X$. See also the question posed in \S \ref{sec: a question}.

For showing that filling pairs admit holomorphic realizations, our method of proof is novel in that we do not use a minimization of an energy functional on the Teichm\"uller space (which would not work, see below). Instead, we carry out a hands-on, geometric, and simultaneous construction of the Riemann surface $Y$ and the holomorphic quadratic differential $q$. This also provides a third proof of the Gardiner-Masur theorem for compact surfaces. 

Taking the perspective from \cite{Wolf} and \cite{Wen}, a holomorphic realization of a pair of foliations is equivalent to a harmonic mapping into a product of two $\R$-trees whose image is locally area minimizing. We rigorously develop this point of view in the infinite area setting, and use our framework to establish the uniqueness of the holomorphic realization and the energy minimization result. 

Stepping back from the main result, one of the underlying contributions of this work is the development of new tools to study infinite Riemann surfaces and holomorphic quadratic differentials that lie outside the quasiconformal context.

\subsubsection{Further results}
We prove that for very general surfaces, there exist holomorphic realization maps $f:X\to Y$ such that the Riemann surface $Y$ is not quasiconformal to $X$ (Theorem \ref{thm: non qc}). In certain contexts, we give sufficient conditions that ensure $f$ is homotopic to a quasiconformal map (Theorems \ref{thm: sufficient, first kind} and \ref{thm: sufficient, disk}). It would be interesting to understand if there is a precise criterion on $\mu$ and $\nu$ that guarantees quasiconformality.

As a complement to the realization theorem, we study the asymptotic behaviour of harmonic maps to $\R$-trees, which leads naturally to an asymptotic Plateau problem for a product of $\R$-trees. Minimal surfaces in products of $\R$-trees have attracted interest recently due to their connection with a now settled conjecture of Labourie (see \cite{MarkovicSagmanSmillie}, \cite{SagmanSmillie}). In a very general context, Theorem \ref{thm: main} implies existence and uniqueness for this asymptotic Plateau problem (Theorem \ref{thm: plateau_disk}). 

The energy minimization aspect of Theorem \ref{thm: main} has applications. We use it to extend the recently introduced {\it new main inequality} (see \cite{MarkovicMinimalDiffeomorphisms}) to locally quasiconformal mappings (Theorem \ref{thm: new main}). As a consequence, we recover Markovi{\'c}-Mateljevi{\'c}'s most general version of the Reich-Strebel inequality on the unit disk (see \cite{MarkovicMateljevic}), and extend it to arbitrary Riemann surfaces (Corollary \ref{cor: generalized main}). Inequalities of this type have been used to prove uniqueness theorems for harmonic maps and minimal surfaces; we use our new inequalities to prove new uniqueness theorems (Corollaries \ref{cor: uniqueness_harmonic} and \ref{cor: minimal}).

Finally, the realization theorem implies that a big mapping class with a global pseudo-Anosov structure has a unique axis (which is the Teichm\" uller geodesic corresponding to a quadratic differential)
in a quasiconformal Teichm\" uller space (Corollary \ref{cor:big-affine-realization}).

All of these results are discussed in more detail in the sections below.

\subsubsection{Overview of the proof of Theorem \ref{thm: main}}\label{subsubsec: overview}
The proof of Theorem \ref{thm: main} is broken up as follows: Theorems \ref{thm:necessary_realization} and \ref{thm:necessary_realization_first_kind} (filling is necessary for realization), Theorems \ref{thm:realization_sufficient} and \ref{thm:realization_first_kind} (filling is sufficient for realization), Theorem \ref{thm: uniquness} (uniqueness, see also \S \ref{sec: without R trees}), and Theorem \ref{thm: energy minimizing} (energy minimizing). In between proving the ``if and only if" portion in \S \ref{sec: filling, necessary}  and \S \ref{sec: filling, sufficient} and proving uniqueness and energy minimization in \S \ref{sec: trees}, we address the quasiconformal realization problem in \S \ref{sec: qc realizations}.

To prove that foliations of $L^1$ quadratic differentials are filling, we first treat the unit disk, and then the general case. On the disk, a key input is a result of Strebel \cite[Lemma 19.6]{Strebel}, which says that if we have an $L^1$ quadratic differential $q$ on $\mathbb{D}$ with singular flat metric $|q|$, then the $|q|$-distance between any two distinct points on $S^1$ is positive (this uses Hardy spaces on the disk). This is a linchpin in our work; it is unknown whether this property holds for quadratic differentials in other function spaces (for instance, the Bers $L^\infty$ space). For each complete geodesic $\gamma$, we then construct a curve with the same endpoints on $S^1$ that nearly realizes the sum of the intersection numbers; in general, one cannot just replace it with a $|q|$-geodesic, so we go by hand. For a general surface $X$, we study the intersections of the two foliations with different types of geodesics (geodesics that stay in the convex core, cross-cuts outside the convex core, etc.). For some cases we modify the argument from the case $X=\mathbb{D}$.

Our proof of sufficiency is very direct. After lifting to the universal cover and modifying the (lifted) laminations near their atomic geodesics, we superimpose the supports of the two laminations and collapse each complementary component, together with the maximal subarcs of one lamination that do not meet the other. We prove using Moore's theorem that the quotient is (abstractly) homeomorphic to the original surface. The two laminations descend to a pair of transverse measured foliations on the quotient, which we then divide into rectangular flow boxes for the foliations. Replacing every flow box by a Euclidean rectangle whose side lengths are prescribed by the transverse measures, and gluing by translations, we produce a holomorphic disk equipped with a quadratic differential. The construction respects the deck group action on $\Tilde{X}$, and the quotient of all this data gives the holomorphic realization. 

For uniqueness and energy we use $\R$-trees. Any measured lamination (or foliation) has a dual $\R$-tree, and the Hopf differential of an equivariant harmonic map to the $\R$-tree is the Hubbard-Masur differential (this is Wolf's proof of Hubbard-Masur \cite{Wolf}). For a holomorphic realization, the two component maps have opposite Hopf differentials, so their product is a conformal harmonic map (this is the starting point in \cite{Wen}). For uniqueness, we prove that if we have two holomorphic realizations, then the two conformal harmonic maps $f_1$ and $f_2$ are embeddings with the same image. The biholomorphism $f_2\circ f_1^{-1}$ then yields an isomorphism of realizations. For energy minimization, we prove that the image of any equivariant harmonic map contains the image of the unique conformal harmonic map, and then we use that energy dominates area. It is worth noting that the images of the conformal harmonic maps are not in general convex; we are really exploiting the $2$-dimensional nature of the problem.

We also found alternative proofs of uniqueness and energy minimization without using $\R$-trees. We sketch them in \S \ref{sec: without R trees}. We went with $\R$-trees for the full proofs because this viewpoint is geometrically revealing and leads into the asymptotic Plateau problem (see \S \ref{subsec: intro asymptotic plateau}), and it might be of interest to the $\R$-tree/harmonic map community to see it all fleshed out.

\begin{remark}
      For a closed surface, the image of the conformal harmonic map is exactly the Guirardel core (see \cite{GuirardelCore}). For arbitrary surfaces, the hands-on construction of a realization of a filling pair can be regarded as an explicit construction of a generalized Guirardel core (see also Remark \ref{rem: equivalence relation and trees} and Definition \ref{def: core}). 
\end{remark}

\subsection{Quasiconformal realizations}
The (quasiconformal) Teichm\" uller space $T_{qc}(X)$ of an arbitrary Riemann surface $X$ consists of equivalence classes of quasiconformal maps from $X$ to variable Riemann surfaces $Y$. Two quasiconformal maps $f_1:X\to Y_1$ and $f_2:X\to Y_2$ are {\it equivalent} if there exists a biholomorphic map $c:Y_1\to Y_2$ such that $f_2^{-1}\circ c\circ f_1:X\to X$ is homotopic to the identity modulo points on the ideal boundary, if any. For $[f:X\to Y], [g:X\to Z]\in T_{qc}(X)$, the Teichm\" uller distance is
$$
d_{T_{qc}(X)} ([f],[g])=\inf_{h\in [g^{-1}\circ f]}\frac{1}{2}\log K(h),
$$
where $K(h)$ is the quasiconformal constant of $h$.

In the original Gardiner-Masur theorem starting on a compact Riemann surface $X$, any homeomorphism $f:X\to Y$ is homotopic to a quasiconformal map $g:X\to Y$. Thus, the holomorphic quadratic differential $q$ coming from a filling pair defines 
a Teichm\"uller geodesic in $T_{qc}(X)$ (i.e., a geodesic for $d_{T_{qc}(X)}$) through $[f:X\to Y]$, and there is an injective correspondence between Teichm\" uller geodesics in $T_{qc}(X)$ and pairs of filling measured laminations on $X$.

On an arbitrary Riemann surface, a homeomorphism is, in general, not homotopic to a quasiconformal map. Therefore, $f:X\to Y$ might not represent a point in $T_{qc}(X)$, i.e., the pair is realized in $T_{qc}(Y)\neq T_{qc}(X)$. 
Motivated by the above result for compact Riemann surfaces, we study the question of when a pair of filling laminations is realized in the quasiconformal class of the initial Riemann surface $X$, i.e., when $T_{qc}(X)=T_{qc}(Y)$. 
\begin{defn}
    Let $\mu$ and $\nu$ be measured laminations on a Riemann surface $X$ realizable by tightening partial measured foliations with finite Dirichlet integral. We say that the pair $(\mu,\nu)$ is \textit{quasiconformally holomorphically realizable} if the homeomorphism $f:X\to Y$ produced by Theorem \ref{thm: main} is homotopic to a quasiconformal map.
\end{defn}
For convenience, we will abbreviate to \textit{quasiconformally realizable}.

Theorem \ref{thm: non qc} below shows that quasiconformal realizations should not be expected in general, and Theorems \ref{thm: sufficient, first kind} and \ref{thm: sufficient, disk} give sufficient criteria for certain types of laminations to admit quasiconformal realization. We focus on Riemann surfaces arising from Fuchsian groups of the first kind and elementary Fuchsian groups (of the second kind). We expect that analogous results hold for Fuchsian groups of the second kind, but the missing ingredient is a robust criterion for a single lamination on a general surface to arise from tightening an integrable partial measured foliation. Pursuing such a result would take us too far afield from the themes of the current paper.

\begin{thm}\label{thm: non qc}
    Let $X=\mathbb{D}/\Gamma$, where either
    \begin{enumerate}
        \item $\Gamma$ is elementary (i.e., $X=\mathbb{D}$ or $X$ is an annulus), or
        \item $\Gamma$ is infinitely generated and of the first kind.
    \end{enumerate}
    Then, there exists a pair of filling laminations $\mu$ and $\nu$, arising from tightening integrable partial measured foliations, that are not quasiconformally realizable.
\end{thm}

Our first sufficiency result concerns filling multicurves on Riemann surfaces of the first kind. Given a curve $\alpha$ on a surface, let $\delta_\alpha$ be the Dirac lamination.

\begin{thm}\label{thm: sufficient, first kind}
    Let $X=\mathbb{D}/\Gamma$, where $\Gamma$ is of the first kind. Let $\{\alpha_n\}_{n=1}^\infty$ and $\{\beta_k\}_{k=1}^\infty$ be a locally finite and filling collection of multicurves, where the $\alpha_n$'s form the cuffs of a pants decomposition. Assume that for some $M>0$ and all $n,k\in \mathbb{N}$, $$\frac{1}{M}\leq \ell_X(\alpha_n), \ell_X(\beta_k)\leq M.$$ Let
    $(m(\alpha_n))_{n=1}^\infty$ and $(m(\beta_k))_{k=1}^\infty$ be square-summable sequences of positive numbers, and define measured laminations on $X$ by
    $$\mu=\sum_{n=1}^\infty m(\alpha_n)\delta_{\alpha_n}, \qquad \nu=\sum_{k=1}^\infty m(\beta_k)\delta_{\beta_k}.$$ If there exists $L>0$ such that for all $n,k\in \mathbb{N}$ with $i(\alpha_n,\beta_k)\neq 0$ we have $$\frac{1}{L}\leq \frac{m(\alpha_n)}{m(\beta_k)}\leq L,$$ then $\mu$ and $\nu$ are quasiconformally realizable.  
\end{thm}
The definition and existence of locally finite filling pairs are discussed in \S \ref{sec: qc, infinite}. The square-summability condition guarantees that $\mu$ and $\nu$ arise from integrable partial measured foliations (see Theorem \ref{thm: realization criteria modulus}).

Our second sufficiency result is about a certain family of laminations on the unit disk, which we call \textit{dual regular polygonal laminations}. The construction and definition are intuitive, but we defer everything to \S \ref{subsubsec: qc realizations} so as to not take too much space in the introduction. All we will say here is that they are specified by partitions of $S^1$ into $n\geq 3$ arcs of equal length, say $I_1,\dots, I_n$, and $J_1,\dots, J_n$, and $L^2$ functions $h_j: I_j\to [0,\infty)$, $h_k': J_k\to [0,\infty)$ symmetric about the midpoints that determine measures on each $I_j$ and $J_k$, which we call endpoint measures.

\begin{thm}\label{thm: sufficient, disk}
   Let $\mu$ and $\nu$ be dual regular polygonal laminations on $\mathbb{D}$. If the endpoint measures are uniformly comparable to the Lebesgue measure, 
then $(\mu,\nu)$ admits a quasiconformal realization.
\end{thm}

\subsection{Minimal surfaces in products of trees}\label{subsec: intro asymptotic plateau}

After completing the proof of Theorem \ref{thm: main}, we consider the dual minimal surface, i.e., the image of the conformal harmonic map in the product of $\R$-trees.
One motivation comes from minimal surfaces in non-positively curved locally symmetric spaces, which have recently become objects of intense study, especially in relation with higher Teichm{\"u}ller theory (see \cite{LabourieCyclicSurfaces}). Large area minimal surfaces in locally symmetric spaces approximate minimal surfaces in affine buildings (see the discussion in \cite{SagmanSmillie}). An important case is when the locally symmetric space is a product of $n$ hyperbolic spaces; the affine building is then a product of $n$ $\R$-trees. One of the main takeaways from the paper \cite{MarkovicSagmanSmillie} is that stability of minimal surfaces in products of closed hyperbolic surfaces is completely determined by minimal surfaces in products of $\R$-trees (this paper led to \cite{SagmanSmillie} about general symmetric spaces and Labourie's conjecture).

The minimal surfaces arising from Theorem \ref{thm: main} are best understood in the context of (mixed) asymptotic Plateau problems. The asymptotic Plateau problem in a Riemannian manifold $M$ with a suitable ``boundary at infinity" $\partial_\infty M$ asks: given a Jordan curve $\gamma$ in $\partial_\infty M$, what can we say about the existence, regularity, and multiplicity of minimal surfaces asymptotic to $\gamma$? A mixed asymptotic Plateau problem allows for $\gamma$ to be in $M\cup \partial_\infty M$. The basic example is $M=\mathbb{H}^3$, so that $\partial_\infty M=\mathbb{CP}^1$, which was first studied in detail by Anderson in \cite{AndersonCompleteMinimalHypersurfaces}; existence holds and uniqueness fails. Another asymptotic Plateau problem concerns $M=\mathbb{H}^2\times \mathbb{H}^2$, where $\partial_\infty M$ is the Furstenberg boundary $S^1\times S^1$. In this context, results from a series of papers together show that for any homeomorphism $f:S^1\to S^1$, the graph asymptotically bounds a unique minimal surface (see \cite{BonsanteSchlenker}, \cite{SeppiSmithToulisse}, \cite{TrebeschiCMCHypersurfaces}). 

Given the connection between minimal surfaces in $\mathbb{H}^2\times \mathbb{H}^2$ and in products of two $\R$-trees, the results on $\mathbb{H}^2\times \mathbb{H}^2$ suggest (along with the proof from \cite{Wen}) that some corresponding asymptotic Plateau problem in some product of $\R$-trees admits a unique solution. In \S \ref{sec: plateau}, we take up the task of properly formulating an asymptotic Plateau problem, and Theorem \ref{thm: main} is used to prove the sought existence and uniqueness theorem.

To keep things transparent in the introduction, we state our result (Theorems \ref{thm: asymptotic plateau existence} and \ref{thm: asymptotic plateau uniqueness}, about general Riemann surfaces) here just in the case of the unit disk. Let $(\mu,\nu)$ be a pair of laminations on $X=\mathbb{D}$ homotopic to foliations with finite Dirichlet integral, and let $P$ be the product of $\R$-trees. By taking the Gromov boundaries of the individual $\R$-trees, $P$ has a natural boundary $\partial_\infty P$. The pair $(\mu,\nu)$ induces a canonical boundary map $$\beta: S^1\to P\cup \partial_\infty P$$ (Proposition \ref{prop: canonical boundary map}). Minimal surfaces do not come with parametrizations of their boundary data, so the asymptotic Plateau problem should concern conformal harmonic maps asymptotic to a reparameterization of $\beta$ (see also Remark \ref{rem: non-tangential limits}).
\begin{thm}\label{thm: plateau_disk}
   Let $\mu$ and $\nu$ be measured geodesic laminations on the unit disk $\mathbb{D}$ homotopic to partial measured foliations with finite Dirichlet integral, and let $P$ be the corresponding product of $\R$-trees. The following are equivalent.
    \begin{enumerate}
        \item $\mu$ and $\nu$ are filling.
        \item The boundary map $\beta: S^1 \to P\cup \partial_\infty P$ is injective.
        \item There exists a conformal harmonic map $\Pi: \mathbb{D}\to P$ with Korevaar-Schoen trace map represented by $\beta\circ g^{-1}$ for some homeomorphism $g$ of $S^1$.
    \end{enumerate}
    When the condition holds, the map $\Pi$ is unique up to pre-composition with M{\"o}bius transformations, and energy minimizing among all $W^{1,2}$ maps from $\mathbb{D}$ to $P$ with Korevaar-Schoen trace maps of the form $\beta\circ h^{-1}$ with $h:S^1\to S^1$ a homeomorphism.
\end{thm}

We note that a finite Plateau problem in a product of two $\R$-trees was considered in \cite[\S 3.2 and \S 3.3]{Sagman} and used to study the asymptotic Plateau problem for a product of disks with variable curvature. The proofs in the current paper are different from those of \cite[\S 3.2 and \S 3.3]{Sagman} (in fact, trying to do the proof from \cite{Sagman} in our context seems difficult), and Theorem \ref{thm: plateau_disk} implies the main results from \cite[\S 3.2 and \S 3.3]{Sagman}.

A different but similar way of looking at Theorem \ref{thm: plateau_disk} is as a uniformization theorem for a region bounded by $\beta(S^1)$. Recently, there have been developments on uniformization problems for metric spaces homeomorphic to a surface (see, for instance, \cite{MeierWenger}). Most results in this direction are for boundary curves with $C^0$ or $W^{1,2}$ parametrizations, while our parametrizations are in some sense just $L^2$. We hope our results here can give insight in this direction.

\subsection{Further applications}

\subsubsection{Main inequalities}

 The Reich-Strebel inequality for quasiconformal maps, also called the main inequality, is a generalization of the classical Groetzch argument. The result is a key ingredient in the proof of Teichm{\"u}ller’s uniqueness theorem for Teichm{\"u}ller geodesics. The Reich-Strebel inequality is valid for quasiconformal maps between Riemann surfaces, and in \cite{MarkovicMateljevic}, Markovi{\'c}-Mateljevi{\'c} extended it to locally quasiconformal homeomorphisms. 
 
 In \cite{MarkovicMinimalDiffeomorphisms}, the so-called \textit{new main inequality}, which implies the ordinary Reich-Strebel inequality, is introduced as a tool to study minimal surfaces in products of closed hyperbolic surfaces. In \cite{MarkovicSagmanMainInequality}, a direct relation is established between the new main inequality and minimal surfaces in products of $\R$-trees. In \cite[Theorem B]{Sagman}, the new main inequality is proved for pairs of quasiconformal maps on the unit disk (using $\R$-trees). The theorem below, which extends \cite[Theorem B]{Sagman} to locally quasiconformal maps on arbitrary surfaces, is a direct consequence of the energy minimizing property from Theorem \ref{thm: main}.
\begin{thm}\label{thm: new main}
    Let $X$ be a conformally hyperbolic Riemann surface and let $f_1,f_2:X\to Y$ be locally quasiconformal homeomorphisms with Beltrami forms $\mu_1$ and $\mu_2$ and whose lifts to the universal covers have a common boundary extension $\partial_\infty \Tilde{X}\to \partial_\infty \Tilde{Y}$. For every $L^1$ holomorphic quadratic differential $q$ on $X$, the new main inequality holds:
\begin{equation}\label{eq: new main}
    \textrm{Re}\int_X q\cdot \Big (\frac{ \mu_1}{1-|\mu_1|^2} - \frac{ \mu_2}{1-|\mu_2|^2}\Big )\leq \int_X |q|\cdot \Big (\frac{|\mu_1|^2}{1-|\mu_1|^2}+\frac{|\mu_2|^2}{1-|\mu_2|^2}\Big).
\end{equation}
\end{thm}

In (\ref{eq: new main}), we allow both sides to be infinite, in which case the inequality is trivial. Taking $f_2:X\to X$ to be the identity, and $f_1$ to be any locally quasiconformal homeomorphism that extends to the identity on $\partial_\infty X$, we obtain that the Reich-Strebel inequality is valid for locally quasiconformal maps; when $X=\mathbb{D}$, this recovers Markovi{\'c}-Mateljevi{\'c}'s result from \cite{MarkovicMateljevic}.
\begin{cor}\label{cor: generalized main}
    Let $X$ be a conformally hyperbolic Riemann surface and let $f:X\to X$ be any locally quasiconformal homeomorphism with Beltrami form $\mu$ and whose lift to the universal cover extends to the identity map $\partial_\infty \Tilde{X}\to \partial_\infty \Tilde{X}$. For every $L^1$ holomorphic quadratic differential $q$ on $X$, 
\begin{equation}\label{eq: generalized main}
    \textrm{Re}\int_X q\cdot \frac{ \mu}{1-|\mu|^2} \leq \int_X |q|\cdot \frac{|\mu|^2}{1-|\mu|^2}.
\end{equation}
\end{cor}

\begin{remark}\label{rem: main other form}
    By a well-known calculation, (\ref{eq: generalized main}) is equivalent to $\int_X |q|\leq \int_X |q|T_\mu q ,$ where $T_\mu q = \frac{|1-\mu\frac{q}{|q|}|^2}{1-|\mu|^2}$. This is the inequality stated in \cite{MarkovicMateljevic}. 
\end{remark}

\subsubsection{Harmonic maps and minimal surfaces}

In \cite[Theorems 2 and 3]{MarkovicMateljevic}, (\ref{eq: generalized main}) is used to prove new uniqueness theorems for harmonic diffeomorphisms of the unit disk. Similarly, in \cite[Theorem A]{Sagman}, (\ref{eq: new main}) (for quasiconformal maps) is used to prove a uniqueness result for an asymptotic Plateau problem in a product of disks (extending \cite{MarkovicMinimalDiffeomorphisms}, which concerns closed surfaces). The theorems apply to harmonic diffeomorphisms and minimal immersions with $L^1$ Hopf differentials, which is a restriction, but the trade-off is that there is no assumption on the curvature or completeness of the targets. The flexibility with the curvature is especially notable for minimal surface problems because all results prior to \cite{Sagman} on the asymptotic Plateau problem in $\mathbb{H}^2\times \mathbb{H}^2$ (see \cite{BonsanteSchlenker}, \cite{SeppiSmithToulisse}, \cite{TrebeschiCMCHypersurfaces}) made strong use of the fact that the curvature of both factors is constant: they all rely on a duality between $\mathbb{H}^2\times \mathbb{H}^2$ and $3d$ anti-de Sitter space. In particular, if we rescale the metric on one of the factors but not the other, the anti-de Sitter proofs have no chance of going through. In fact, in variable negative curvature and without the $L^1$ assumption, existence and uniqueness are totally open.

Using (\ref{eq: generalized main}) and (\ref{eq: new main}) and following the proofs from \cite{MarkovicMateljevic} and \cite{Sagman}, we prove uniqueness results on general surfaces. Note that if the surfaces are of the first kind, the ideal boundary conditions are vacuous.
\begin{cor}\label{cor: uniqueness_harmonic}
    Let $X$ be a conformally hyperbolic Riemann surface, let $(Y,\sigma)$ be a Riemannian surface, and let $f,g:X\to (Y,\sigma)$ be harmonic diffeomorphisms extending to the same boundary homeomorphism $\partial_\infty f=\partial_\infty g$. If the Hopf differentials of $f$ and $g$ are $L^1$, then $f=g$.
\end{cor}

\begin{remark}
When $X$ and $Y$ are of the first kind and $\sigma$ has pinched negative curvature, uniqueness is already known for quasiconformal harmonic diffeomorphisms.
On arbitrary Riemann surfaces, harmonic maps with $L^1$ quadratic differentials are not always quasiconformal (see \cite[Theorem 13]{WanCMC}).
\end{remark}
 For minimal surface uniqueness, Theorem \ref{thm: new main} does not appear to be enough to fully remove a quasiconformal condition (but we can weaken it slightly, see Remark \ref{rem: weaken quasiconformal}).
\begin{cor}\label{cor: minimal}
    Let $(Y,\sigma_1)$ and $(Z,\sigma_2)$ be conformally hyperbolic Riemannian surfaces and let $m:\partial_\infty \Tilde{Y}\to \partial_{\infty} \Tilde{Z}$ be an equivariant homeomorphism. Then, up to Teichm{\"u}ller equivalence, there exists at most one Riemann surface $X$ with minimal immersion $$F=(f,g):X\to (Y\times Z,\sigma_1\oplus \sigma_2)$$ such that $f$ and $g$ are quasiconformal (harmonic) diffeomorphisms with $L^1$ Hopf differentials and $\partial_\infty g\circ (\partial_\infty f)^{-1}=m$.
\end{cor}

\subsubsection{Pseudo-Anosov elements in big mapping class groups}
Let \(S\) be an oriented topological surface. Its
orientation-preserving mapping class group is
\[
   \operatorname{MCG}(S)
   =
   \operatorname{Homeo}^{+}(S)/
   \operatorname{Homeo}_{0}(S),
\]
where \(\operatorname{Homeo}_{0}(S)\) denotes the subgroup of homeomorphisms
isotopic to the identity. When \(S\) has infinite topological type, this group
is commonly called a \emph{big mapping class group}. In contrast with the
finite-type case, big mapping class groups generally do not admit a global
Nielsen--Thurston trichotomy: a single mapping class may exhibit different
dynamical behaviour on different subsurfaces or near different ends.

Nevertheless, some elements of big mapping class groups possess
a pseudo-Anosov-like structure (with applications to hyperbolic $3$-manifolds). See, for instance, \cite{CantwellConlonFenley}, \cite{HooperGridGraphs}, \cite{LandryMinskyTaylorEndperiodic}.
Theorem \ref{thm: main} suggests a new class of elements and yields an affine realization criterion. 
\begin{cor}
\label{cor:big-affine-realization}
Let $S$ be an oriented surface and let
$f:S\to S$ be an orientation preserving homeomorphism.
Suppose that there is a hyperbolic Riemann surface $X$ homeomorphic to $S$ and carrying measured geodesic
laminations \(\mu\) and \(\nu\) such that             
\begin{enumerate}
   \item for some \(\lambda>1\),
        $f_{*}\mu=\lambda\mu$
           and $f_{*}\nu=\lambda^{-1}\nu$,
   \item $\mu$ and $\nu$ are homotopic to partial measured foliations of finite
         Dirichlet integral, and
    \item  \(\mu\) and \(\nu\) are filling.
\end{enumerate}
Let $h:X\to Y$ be the homeomorphism and $q$ the integrable quadratic differential for $\mu$ and $\nu$ obtained through Theorem \ref{thm: main}. Then, the conjugate
\[
   F=h\circ f\circ h^{-1}:Y\to Y
\]
is homotopic to a map that is affine in the natural parameters of \(q\), expanding the horizontal direction by \(\lambda\) and contracting the
vertical direction by \(\lambda^{-1}\). 
\end{cor}
Informally, we start with a map $f$ on a topological surface $S$ that represents a big mapping class with global pseudo-Anosov behaviour (in the corollary, we had to have a hyperbolic structure to talk about \textit{geodesic} laminations and the filling property, but (1) and (3) are morally topological properties). The pair \((\mu,\nu)\) plays the role of
the stable and unstable measured laminations, while \(\lambda\) is the
stretch factor. Corollary \ref{cor:big-affine-realization} says that if we can find a Riemann surface structure on which $\mu$ and $\nu$ satisfy an integrability condition (i.e., condition (2) holds), then changing the marking provides a unique Teichm\" uller space $T_{qc}(Y)$ in which a Teichm\"uller geodesic through the basepoint is the translation axis of the obtained element $[F]\in \mathrm{MCG}(Y)$ (which has translation length $\log \lambda$).

The justification for corollary \ref{cor:big-affine-realization} is simple: let $t\mapsto (Y_t,q_t)$ be the Teichm{\"u}ller geodesic ray associated with $q$. Since the intersection numbers of $\lambda \mu$ and $\lambda^{-1} \nu$ agree with that of $\mathcal{F}_h(q_\lambda)$ and $\mathcal{F}_v(q_\lambda)$ respectively, by injectivity of the heights map (see \cite{Saric22}) and uniqueness in Theorem \ref{thm: main}, the holomorphic realization $(\lambda \mu, \lambda^{-1}\nu)$ on $X$ is $(h,Y_\lambda, q_\lambda)$. The claim for $F=h\circ f \circ h^{-1}$ follows.

\begin{remark}
   In Corollary \ref{cor:big-affine-realization}, the mapping
class \([f]\) has infinite order, and no positive power of
\([f]\) fixes the isotopy class of an essential simple closed curve. Indeed, the first claim follows from $\lambda>1$, and the second from invariance of geometric
intersection number and Definition \ref{def:filling} (we leave the simple argument for the reader). In particular, $[f]$ is neither periodic nor reducible in the sense of ordinary mapping class groups. 
\end{remark}

\subsection{Acknowledgments}
Both authors thank Kasra Rafi for numerous helpful discussions. The second named author thanks Camillo De Lellis for  insightful discussions.  

The first named author gratefully acknowledges that most of this research was completed while he was a visiting scholar at the Fields Institute for Research in Mathematical Sciences. The second named author was partially supported by National Science Foundation award DMS-2521870, Simons Collaboration Grant award 00012869, and 
 PSC-CUNY grant 67366-00 55. Most of the results in the paper were obtained while the second named author
was a member of the Institute for Advanced Study during the academic year 2025-2026. The second named author  
 is grateful to the Institute for Advanced Study for their hospitality and AMIAS
Member Fund for their support.

\section{Preliminaries}
\label{sec:prel}

\subsection{Quadratic differentials and foliations} 
Throughout the paper if $X$ is a Riemann surface, we denote the universal cover by $\Tilde{X}$. We write $\partial_\infty X$ and $\partial_\infty \Tilde{X}$ for ideal boundaries. When an identification between $\Tilde{X}$ and the unit disk $\mathbb{D}$ is fixed, we just write $S^1$ for $\partial_\infty \Tilde{X}$. Under an identification, $\partial_\infty X$ is the quotient by $\Gamma$ of the complement of the limit set of $\Gamma$.

From now on in this section, $X=\mathbb{D}/\Gamma$. We recall that the Fuchsian group $\Gamma$ is of the first kind if the limit set of $\Gamma$ is equal to $S^1$, and otherwise $\Gamma$ is of the second kind. In general, $X$ is the union of its convex core with some number of hyperbolic funnels and hyperbolic half-planes; if $X$ has at least one funnel or half-plane, then $\Gamma$ is of the second kind. 
The geodesic half-planes in $X$ appear when the limit set of $\Gamma$ on $S^1$ has a complementary interval component whose stabilizer is trivial (see \cite{BS}). Equivalently, a disjoint sequence of simple closed geodesics accumulates on an open geodesic in $X$ and this open geodesic is the boundary of the attached geodesic half-plane.

A holomorphic quadratic differential $q$ on $X$ is a holomorphic section of the square of the canonical (cotangent) bundle of $X$. Locally, $q$ is written $q=\varphi(z)dz^2$, where $\varphi$ is holomorphic. The quadratic differential $q$ is integrable (equivalently, finite-area, or $L^1$), if $$\int_X |q|<\infty,$$ where $|q|=|\varphi(z)||dz|^2$. 

The horizontal foliation $\mathcal{F}_h(q)$ of $q$ is the singular measured foliation on $X$ whose non-singular leaves are the integral curves of the line field on $X\backslash q^{-1}(0)$ on which $q$ returns a positive real number. The singular points are the zeros of $q$; a zero of order $k>0$ corresponds to a standard $(k+2)$-prong in $\mathcal{F}_h(q)$. The foliation comes with a density $|\textrm{Im}\sqrt{q}|$ that determines a transverse measure on the horizontal foliation $\mathcal{F}_h(q)$. Around a point $p$ where $q(p)\neq 0$, one can choose a natural parameter $$w(z)=u+iv =\int_{z_0}^z \sqrt{q},$$ in which $q$ is represented as $q=dw^2$. Any other such coordinate $w_1$ can be written $w_1=\pm w+ c,$ for some $c\in \C$. The horizontal coordinate lines $\{v=\textrm{ constant}\}$ realize leaves of the horizontal foliation, and the transverse measure of a nearby transverse curve is the vertical coordinate distance between the endpoints.

Arcs in the horizontal foliation are called horizontal arcs, and an in-extendable subset of the foliation parametrized by a continuous map $\R\to X$ is called a horizontal trajectory. A horizontal trajectory can either be a regular trajectory, which means it is a maximal horizontal arc, or it can be a generalized trajectory, which means it passes through zeros of $q$ (see also \cite[\S 3.1]{Saric22}).

The vertical foliation $\mathcal{F}_v(q)$ is the horizontal foliation of $-q$. Similar to above, its non-singular leaves are integral curves of the line field on $X\backslash q^{-1}(0)$ on which $q$ gives a negative number, the density is $|\textrm{Re}\sqrt{q}|$, and in a natural coordinate patch for $q$, its leaves are vertical lines.

We recall the notion of a partial measured foliation from \cite{Saric24}, which generalizes the notion of an ordinary singular measured foliation, and whose definition is motivated by the paper of Gardiner and Lakic \cite{GLPartial}.
\begin{defn}
    A partial measured foliation $\mathcal{F}$ on $X$ is a collection $\{(U_i,v_i)\}_{i\in I}$, where $U_i\subset X$ is a closed Jordan domain with piecewise $C^1$ boundary, and $v_i:U_i\to \R$ are $C^1$ maps with surjective derivative at each point in the interior of $U_i$ with the property that for any $i,j$ such that $U_i\cap U_j\neq \emptyset$, $v_i=\pm v_j + c$ for some $c\in \R$.
\end{defn}
By the implicit function theorem, the interior of every $U_i$ is foliated by $C^1$ arcs $v_i^{-1}(c)$, $c\in v_i(\textrm{int}(U_i))\subset \R$. A horizontal arc for $\mathcal{F}$ is a connected union of arcs $v_i^{-1}(c_i)$, $c_i\in \R$, for some finite or infinite number of choices of $i$ and $c_i$. A horizontal trajectory is a maximal horizontal arc. A horizontal trajectory that limits to a point in $X$ is called a singular trajectory. The densities $|dv_i|$ define a transverse measure on $X$, supported in $\cup_i U_i$. Of course, the foliation associated with a holomorphic quadratic differential gives a partial measured foliation such that the transverse measures coincide: one takes a covering of $X\backslash q^{-1}(0)$ by natural coordinate disks $U_i$, and the $v_i$'s are the natural coordinate imaginary part projections. 
\begin{defn}
    A partial measured foliation $\mathcal{F}$ on $X$ is proper if, except for countably many singular trajectories, the lift to the universal cover $\mathbb{D}$ of each horizontal trajectory accumulates to two distinct ideal endpoints on $S^1$.
\end{defn}
We always assume that $\{(U_i)\}_{i\in I}$ is locally finite. That is, every compact subset of $X$ intersects at most finitely many $U_i$'s. Given $(U_i,v_i)$, the Dirichlet integral of $v_i$ is $$\int_{U_i} |\nabla v_i|^2=\int_{U_i} \Big|\frac{\partial v_i}{\partial x}\Big|^2+\Big|\frac{\partial v_i}{\partial y}\Big|^2 dxdy.$$  Note that for any $U_j$ intersecting $U_i$, $|\nabla v_j|^2=|\nabla v_i|^2.$ Using a partition of unity on $X$, the Dirichlet integral $\mathcal{D}(\mathcal{F})$ is defined by patching together the Dirichlet integrals of the $v_i$'s.
\begin{defn}
  A partial measured foliation $\mathcal{F}$ on $X$ is integrable if $\mathcal{D}(\mathcal{F})<\infty$.
\end{defn}
When $\mathcal{F}$ arises from the horizontal foliation of a quadratic differential $q$, it is easily seen that the density $|\nabla v_i|^2$ is just $|q|$, and that $\mathcal{D}(\mathcal{F})$ is the $L^1$-norm $\int_X |q|$.

\subsection{Realizable laminations}\label{sec: realizable laminations}
Let $X=\mathbb{D}/\Gamma$ be a Riemann surface as above.  We recall that a geodesic current on $X$ is a $\Gamma$-invariant Radon measure on the set of unoriented unparametrized complete geodesics $\mathcal{G}(\mathbb{D})$ of $\Tilde{X}$. $\mathcal{G}(\mathbb{D})$ is parametrized and topologized by the quotient of $(S^1)^2\backslash \{(x,y)\in (S^1)^2: x=y\}$ by the swapping map $(x,y)\mapsto (y,x)$. In this paper, we view a lamination on $X$ as both (1) a closed set of geodesics equipped with a transverse measure and (2) as a geodesic current.

Let $\mathcal{F}$ be a proper partial measured foliation on $X$ and let $\tilde{\mathcal{F}}$ be the $\Gamma$-invariant lift to the universal cover. By definition, every non-singular horizontal trajectory has exactly two ideal endpoints on $S^1$; such a trajectory can then be pulled tight to a complete geodesic on $\mathbb{D}$ with the same ideal endpoints. We denote by $\mathcal{G}(\Tilde{\mathcal{F}})$ the set of geodesics obtained in this fashion. As explained in 
\cite{Saric24}, $\mathcal{G}(\tilde{\mathcal{F}})$ is a geodesic lamination.

In the case of a quadratic differential, we also have tightening for non-singular trajectories: by Marden and Strebel \cite{MardenStrebel1}, it follows that every generalized trajectory tightens to a unique geodesic. As well, if $g$ in the support of the lamination is not the lift of a simple closed geodesic, then there is a unique leaf tightening to $g$ \cite[Theorem 2(c)]{MardenStrebel1}. If $g$ is the lift of a simple closed geodesic, then we have the lift of a ring domain tightening to $g$.  The two boundary trajectories of this set are generalized trajectories and lifts of the boundary components of the downstairs ring domain.

By repeating the arguments in \cite{Saric24} (which is specialized to quadratic differentials, but goes through in general), the transverse measure for $\mathcal{F}$ induces a transverse measure on $\tilde{\mathcal{F}}$ which in turn induces a transverse measure on $\mathcal{G}(\Tilde{\mathcal{F}})$. We denote by $\mu_{\Tilde{\mathcal{F}}}$ the induced measured lamination in $\mathbb{D}$, which is invariant under the action of $\Gamma$. By $\Gamma$-invariance, $\mu_{\Tilde{\mathcal{F}}}$ descends to a measured lamination $\mu_{\mathcal{F}}$ on $X$ that is called the straightening of $\mathcal{F}$. 

Straightening to the same lamination can be seen as an equivalence relation on proper partial measured foliations. In \cite{Saric24,Saric25}, it is proved that, under the integrability hypothesis, every equivalence class is uniquely realized. This generalizes a classical theorem of Hubbard and Masur \cite{HMmain}.
\begin{defn}
A measured lamination $\mu$ is realizable by an integrable partial measured foliation if there exists an integrable proper partial measured foliation $\mathcal{F}$ such that $\mu=\mu_{\mathcal{F}}.$
\end{defn}
\begin{thm}[\cite{Saric24,Saric25}]\label{thm: HMS}
    Let $X$ be any Riemann surface. For any measured lamination $\mu$ realizable by an integrable partial measured foliation, there exists a unique integrable holomorphic quadratic differential $q$ on $X$ such that $\mu_{\mathcal{F}_h(q)}=\mu.$
\end{thm}

It will be useful to have the relations between $\mathcal{F}$ and $\mu_{\mathcal{F}}=:\mu$ made explicit. The transverse distance between $x,y\in\mathbb{D}$ for the lifted foliation $\Tilde{\mathcal{F}}$ is the infimal transverse measure over all paths connecting $x$ and $y$. Let $x$ and $y$ be points in $\mathbb{D}$ lying on leaves $\ell_x$ and $\ell_y$, which tighten to geodesics $g_x$ and $g_y$ respectively. Then, viewing $\mu$ as a current, with lift $\Tilde{\mu}$ if $g_x$ and $g_y$ are not atomic, the transverse distance between $x$ and $y$ is equal to the $\Tilde{\mu}$-measure of the set of all geodesics separating $g_x$ and $g_y$. If one of the geodesics, say $g_x$, is atomic, we have to be a bit more careful. First note that the transverse measure of the set of all leaves $L_{g_x}$ homotopic to $g_x$ is $\Tilde{\mu}(\{g_x\})$. If a transverse path from $x$ to $y$ does not intersect $L_{g_x}$, then we have the same equality as above. If a transverse path passing through all of $L_{g_x}$, then the transverse distance is equal to the sum of $\Tilde{\mu}(\{g_x\})$ and the $\mu$-measure of the set of all geodesics separating $g_x$ and $g_y$. The analogous results hold if $g_y$ is also atomic.

These characterizations lead to equalities of intersection numbers. Viewing $\mu_{\mathcal{F}}$ as a set of geodesics with a transverse measure, for any compact arc $\gamma$ in $X$, the intersection number $i(\mu,\gamma)$ is defined to be the infimum of the transverse measures of arcs with the same endpoints as $\gamma$. If $\gamma$ is a closed loop, $i(\mu,\gamma)$ is the infimum of the transverse measures of homotopic loops. We recall that a {\it geodesic cross-cut} is a complete geodesic that leaves compact subsets of $X$ at both ends. For a simple cross-cut $\gamma$, $i(\mu,\gamma)$ is the infimum of transverse measures of cross-cuts with the same ideal endpoints and the infimum is achieved by the unique geodesic cross-cut homotopic to $\gamma$. For a general complete simple geodesic $\gamma$, $i(\mu,\gamma)$ is defined by lifting to the universal cover and taking the transverse measure of the corresponding cross-cut; if $\gamma$ was already a cross-cut, this is the same definition because its stabilizer in $\Gamma$ is trivial. All of these definitions are made in the same way for $\mathcal{F}$. By definition, if $\gamma$ is a compact arc, $$i(\mathcal{F},\gamma)=i(\mu,\gamma).$$ It follows from \cite[Lemma 3.3 and Lemma 3.4]{Saric22} and Theorem \ref{thm: HMS} that if $\gamma$ is a simple closed curve or a complete simple geodesic, then $$i(\mathcal{F},\gamma)=i(\mu,\gamma).$$
The intersection number can be defined in even a greater generality. Assume we are given two proper measured foliations $\mathcal{F}$ and $\mathcal{G}$ of $X$. The intersection of one leaf of $\mathcal{F}$ with the measured foliation $\mathcal{G}$ is the infimum of all intersections of curves that share the endpoints with the leaf and the foliation $\mathcal{G}$. The quantity $i(\mathcal{F},\mathcal{G})$ is obtained by integrating these intersection numbers over the set of leaves of $\mathcal{F}$ with respect to the transverse measure of $\mathcal{F}$. Equivalently, we can replace $\mathcal{F}$ and $\mathcal{G}$ with the homotopic measured laminations $\mu$ and $\nu$, and then we define $i(\mathcal{F},\mathcal{G})$ to be equal to $i(\mu ,\nu )$, which is defined in the standard fashion.

\subsubsection{Atoms}
We will need a description of the atoms of a realizable lamination. Let $q$ be an integrable holomorphic quadratic differential on $X$ with horizontal foliation $\mathcal{F}_h(q)$ and realizing the lamination $\mu$. For $X=\mathbb{D}$, i.e, for $\Gamma$ trivial, it is proved in \cite[Proposition 4.4]{HakSar} that the transverse measure of $\mathcal{F}_h(q)$ on the set of geodesics $\{ a\}\times [c,d]\subset G(\mathbb{D})$ is zero for all $a,c,d\in S^1$ in the counterclockwise order. In particular, the measured lamination corresponding to $\mathcal{F}_h(q)$ has no atoms. On an arbitrary surface, the measured (geodesic) lamination $\mu$ obtained by straightening $\mathcal{F}_h(q)$ can have atoms because the foliation $\mathcal{F}_h(q )$ may have cylinders foliated by closed leaves (see \cite[Theorem 13.1]{Strebel}).
We show that if we remove the cylinders, then \cite[Proposition 4.4]{HakSar} goes through in general.

\begin{lem}\label{lem:asymptotic}
Let $\mathcal{F}$ be a proper integrable partial measured foliation on a Riemann surface $X=\mathbb{D}/\Gamma$, tightening to geodesic lamination $\mu$, and let $\tilde{\mu}$ be the lift of $\mu$ to $\mathbb{D}$. For $p\in S^1$, let $E_p$ be the set of geodesics with one endpoint at $p$. Then, either:
\begin{enumerate}
    \item $\tilde{\mu}(E_p)=0$, or 
    \item $p$ is the endpoint of a geodesic $g$ in the support of $\tilde{\mu}$ that descends to a simple closed curve on $X$. Moreover, in this case, $\tilde{\mu}(E_p)=\tilde{\mu}(\{g\})>0$.
\end{enumerate}
\end{lem}
\begin{proof}

We replace the partial measured foliation $\mathcal{F}$ with the homotopic horizontal foliation $\mathcal{F}_h(q)$ of a unique finite-area holomorphic quadratic differential $q$ on $X$ (see \cite[Theorem 1.1]{Saric25}). Then $\mathcal{F}_h(q)$ is homotopic to $\mu$ and it suffices to prove the lemma for $\mathcal{F}_h(q)$. Let $\widetilde{\mathcal{F}_h(q)}$ be the lift of $\mathcal{F}_h(q)$ to $\mathbb{D}$, equivalently the horizontal foliation of the lift $\tilde{q}$ of the quadratic differential ${q}$.

We first assume that $p$ is not a fixed point of an element in $\Gamma\setminus\{ id\}$, and we prove $\tilde{\mu}(E_p)=0$. Let $\beta$ be a short closed vertical arc of $\tilde{q}$ (containing no zeros of $\Tilde{q}$) such that the two horizontal leaves $\ell_p^1$ and $\ell_p^2$ of $\widetilde{\mathcal{F}_h(q)}$ through $\partial \beta$ have a common ideal endpoint $p\in S^1$. (If no such curve exists, then $\Tilde{\mu}(E_p)=0$ and there's nothing to prove.) The union of $\beta$, the point $p$, and the two rays of $\ell_p^1$ and $\ell_p^2$ starting at $\beta$ and ending at $p$ forms the boundary of a Jordan domain $\Omega$. Any horizontal ray of $\tilde{q}$ that starts at the interior of $\beta$ and enters $\Omega$ accumulates to $p\in S^1$ since it cannot intersect the finite boundary of $\Omega$ by the Teichm\"uller lemma \cite[\S 14.1]{Strebel}. By slightly perturbing the original $\ell_p^1$ and $\ell_p^2$, we can assume that they form regular trajectories, i.e., they never pass through zeros of $\Tilde{q}$.

If $z\in\Omega$ is a zero of $\tilde{q}$ then there are at least three horizontal rays of $\tilde{q}$ issued from $z$. At most one of the rays can intersect $\beta$. Indeed, if two rays from $z$ intersect $\beta$ we have a contradiction by the Teichm\"uller lemma. Therefore, at least two rays end at $S^1$ and their limit points on $S^1$ are distinct by \cite[Theorem 1]{MardenStrebel1}. This is a contradiction since $\Omega$ has only one point of $S^1$ on its boundary. Therefore, the quadratic differential $\tilde{q}$ has no zeros in the closure of $\Omega$, and $\Omega$ is foliated by horizontal rays starting at the interior points of $\beta$ and limiting to $p\in S^1$.

By \cite[Theorem 13.1]{Strebel}, almost every leaf of $\mathcal{F}_h(q)$ is either closed, a spiral (i.e., the limit set of any ray along the trajectory contains more than one point), or a cross-cut. The leaves through $\beta$ are not lifts of closed leaves of $\mathcal{F}_h(q)$ since their unique limit point $p$ is not fixed by a hyperbolic element of $\Gamma$. If a set of trajectories through $\beta$ of positive area are cross-cuts, then almost every trajectory of this set has a finite length. We map $\Omega$ into the natural parameter of $\tilde{q}$ and note that the map is conformal since $\tilde{q}$ has no zeros. Each trajectory of finite length corresponds to two prime ends for the image of $\Omega$, one prime end corresponding to a point of $\beta$ and the other prime end corresponding to the point $p$. The prime ends corresponding to two different horizontal trajectories are different prime ends of the image of $\Omega$. We obtain an uncountable set of prime ends corresponding to a single point $p$ on the boundary of $\Omega$. This is a contradiction since the Carath\' eodory theorem establishes a homeomorphism between the spaces of prime ends.

Thus, almost all trajectories of $\tilde{q}$ through $\beta$ are lifts of spiral trajectories on $X$.
By decreasing $\beta$ if necessary, we can assume that the (universal covering) projection $\pi :\beta\to \pi (\beta )$ is a homeomorphism. We can further assume that the trajectory through one endpoint $\pi (z_1)$ of $\pi (\beta )$ is a spiral. Let $\pi (r_1)$ be the ray from $\pi (z_1)$ such that its lift $r_1$ from $z_1\in\partial\beta$ is part of the boundary of $\Omega$ (that is, after decreasing $\beta$ and up to relabeling, $r_1$ is one of the trajectories through $\ell_p^1$ and $\ell_p^2$). The ray $\pi (r_1)$ intersects $\pi (\beta )$ in countably many points. Consider the first point $\pi (z_1')$ of intersection $\pi (r_1)\cap\pi (\beta )$ after $\pi (z_1)$. We decrease $\beta$ even further such that $\pi (\beta )$ has boundary points $\pi (z_1)$ and $\pi (z_1')$. There are two possibilities for the ray $\pi (r_1)$: either the germ of $\pi(r_1)$ approaching $z_1'$ lies on the same side of $\pi(\beta)$ as the germ of $\pi(r_1)$ leaving $z_1$, or this germ lies on the opposite side.

\begin{figure}[htb]
    \centering
    \includegraphics[width=12 cm]{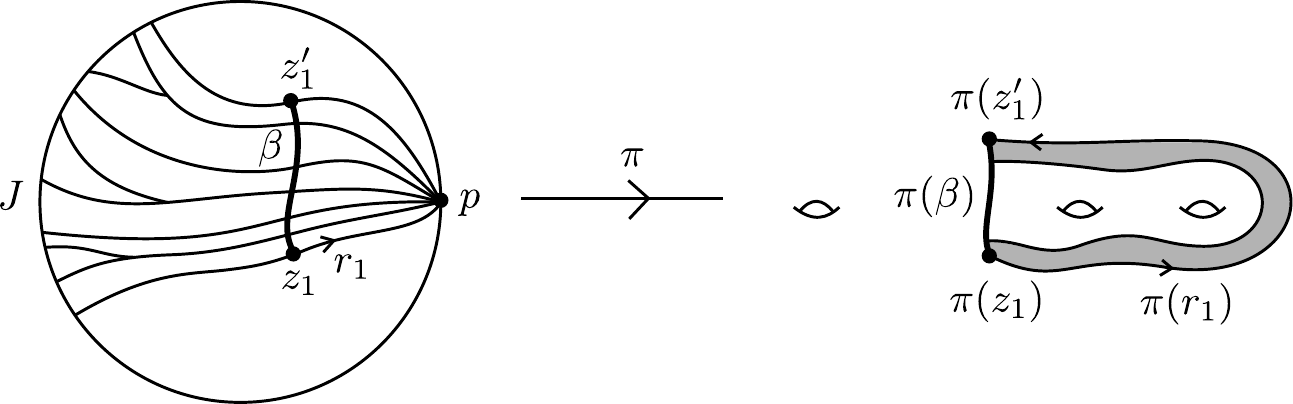}
    \caption{No unexpected atoms in the lift.}
    \label{fig:E_p}
\end{figure}

Assume that we are in the case where the germ of $\pi (r_1)$ approaching $\pi (\beta )$ at the point $\pi (z_1')$ is on the same side as the germ of $\pi (r_1)$ leaving $\pi(z_1)$ (see Figure \ref{fig:E_p}). Since $\Omega$ contains no zeros of $\tilde{q}$, there is a rectangle in the natural coordinate of $q$ whose one horizontal side lies on the ray $\pi (r_1)$ and is to the left of $\pi (r_1)$, and the two vertical sides lie on $\pi (\beta )$. The vertical sides on $\pi (\beta )$ of the rectangle are of the same length and they lie on the same side of $\pi (\beta )$. Therefore the vertical sides are disjoint. The horizontal side $h$ of the rectangle that does not lie on $\pi (r_1)$ has the 
same properties as the side on $\pi (r_1)$. Thus, there exists a rectangle on the left of $h$ with two vertical sides on $\pi (\beta )$. The new rectangle extends the original rectangles and this extension is possible since $\Omega$ has no zeros of $\tilde{q}$. The union of all such extensions is a half-open rectangle with one horizontal side on $\pi (r_1)$ and two vertical sides equal to two half-open arcs of $\pi (\beta )$ with endpoints $z_1$ and $z_1'$, and the open end at the midpoint of $\pi (\beta )$. The horizontal sides of these extension of the original rectangle are converging to a horizontal arc of $q$. Since the endpoints of the horizontal sides lie on $\pi (\beta )$ and they have equal non-zero $|q|$-length, it follows that the limiting horizontal arc is parametrized by that length with start and endpoint at the midpoint of $\pi (\beta )$. By our assumption, no horizontal trajectory through $\pi (\beta )$ is closed. This forces the horizontal trajectory to pass through a zero of $q$. However, this trajectory is the image under the covering map $\pi$ of a horizontal ray in $\Omega$ and $\Omega$ contains no zeros of $\tilde{q}$. This is a contradiction since the zeros of $q$ are precisely the image of zeros of $\tilde{q}$ under the covering map $\pi$.

Assume now that $\pi (r_1)$ meets $\pi (\beta )$ at $z_1'$ from the opposite side compared to where it started at $z_1$. The next intersection $z_1''$ of $\pi (r_1)$ with $\pi (\beta )$ intersects $\pi (\beta )$ from the same side as the intersection at $z_1$ or at $z_1'$. We decrease $\pi (\beta )$ so that it's endpoints at the two intersections are on the same side and note that the horizontal arcs in the given direction are covered under the covering map $\pi$ with the horizontal rays in $\Omega$. Then the same argument as in the previous case gives a contradiction. We conclude that $E_p$ cannot consist of lifts of spirals, which completes the proof that $\Tilde{\mu}(E_p)=0$ in the case that $p$ is not a fixed point of an element of $\Gamma\setminus\{ id\}$. 

Now assume that $p$ is the fixed point of a parabolic element in $\Gamma\setminus\{ id\}$. Since $q$ is a finite-area holomorphic quadratic differential, it extends to the corresponding puncture in $X$ either holomorphically or with a pole of order $1$. If $q$ is holomorphic, then it has a zero of order $k\geq 0$, and there are $k+2$ horizontal and $k+2$ vertical trajectories ending at the puncture. If $q$ has a pole of order $1$, then there is a single horizontal and a single vertical trajectory ending at the puncture. Since the straightening of trajectories fixes the puncture and any simple geodesic entering a neighborhood of the puncture must limit to the puncture (thus it is isolated), we conclude that $\mu$ cannot have geodesics ending at any puncture of $X$ (because the transverse measure of the straightening of a single infinite leaf is zero). Therefore, there are no geodesics upstairs limiting to $p$, and we deduce $\Tilde{\mu}(E_p)=0.$

Finally, assume that $p$ is the fixed point of a hyperbolic element $\gamma \in \Gamma\setminus\{ id\}$. After possibly exchanging $\gamma$ with its inverse, we can assume that $p$ is the attracting fixed point. Let $g$ be the axis of $\gamma$. We prove that $\tilde{\mu }(E_p\backslash\{g\})=0$.  Toward this, suppose for contradictions sake that $\Tilde{\mu}(E_p\backslash \{g\})>0$. Pick a compact subset $K\subset E_p\backslash \{g\}$ with positive $\Tilde{\mu}$-measure. For $n$ sufficiently large, $\gamma^{-n} K$ is disjoint from $K$, and moreover for all unequal $j,k>0$, $\gamma^{-jn}K\cap \gamma^{-kn} K = \emptyset.$ Since the iterates $\gamma^{-n} K$ converge to $\{g\}$ in the Hausdorff sense, the disjoint union $$C=(\cup_{j\geq 0} \gamma^{-jn}K)\cup \{g\}$$ is compact in the space of geodesics $\mathcal{G}(\mathbb{D})$. Since every element in the above union is disjoint and $\Tilde{\mu}$ is $\Gamma$-invariant, $C$ has infinite measure, which contradicts the Radon property of $\Tilde{\mu}$. Therefore, we indeed have $\tilde{\mu }(E_p\backslash\{g\})=0$.

Since $\tilde{\mu }(E_p\backslash\{g\})=0$, either (1) $\Tilde{\mu}(E_p)=0$ if $g$ is not in the support of $\Tilde{\mu}$, or (2)  $\Tilde{\mu}(E_p)=\Tilde{\mu}(\{g\})>0$ if it is.
\end{proof}

\subsection{Energy and harmonic maps}\label{sec: intro harmonic maps}
The content of this subsection will not come into play until \S \ref{sec: trees}. Throughout, let $X$ be a Riemann surface equipped with a conformal hyperbolic metric, which we use to measure distances.

A segment in a metric space is the isometric image of a real interval. A metric space is called geodesic if any two points can be connected by a segment. Korevaar-Schoen developed a theory of $W_{\textrm{loc}}^{1,2}$ maps from a Riemannian manifold to complete and geodesic metric spaces \cite{KorevaarSchoen}. We don't need the precise definitions in this paper, and we content ourselves with mentioning that locally Lipschitz maps are $W_{\textrm{loc}}^{1,2}$

For any complete and non-positively curved geodesic metric space $(M,d)$ and $W_{\textrm{loc}}^{1,2}$ map $h:X\to (M,d)$, by \cite[Theorem 2.3.2]{KorevaarSchoen}, we can associate a locally $L^1$ measurable metric $g(h)$, defined on pairs of Lipschitz vector fields. If $h$ is a $C^1$ map to a smooth manifold $M$ and the distance $d$ is induced by a Riemannian metric $\sigma$, then $g(h)$ is represented by the pullback metric $h^*\sigma$. The energy density form is the locally $L^1$ form, defined with respect to a holomorphic coordinate $z=x+iy$ on $X$ by
\begin{equation}\label{en}
    e(h)=\frac{1}{2}(\textrm{trace}( g(h)))dx dy,
\end{equation}
where the trace is just the ordinary trace on symmetric $2$-tensors on $\R^2$. The total energy is $$E(h) = \int_X e(h),$$ which we allow to be infinite.

\begin{defn}\label{harmdef}
$h$ is harmonic if for all relatively compact subsets $U\subset X$, $h|_U$ is a critical point of $f\mapsto \mathcal{E}(U,f)$ relative to its boundary values.
\end{defn}
See \cite{KorevaarSchoen} for the precise notion of boundary values. In the coordinate $z=x+iy$, let $g_{xx}(h), g_{yy}(h)$, and $g_{xy}(h)$ be the components of $g(h)$. 
\begin{defn}
    The Hopf differential of $h$ is the measurable quadratic differential given by
\begin{equation}\label{mhopf}
q(h)(z)=\frac{1}{4}(g_{xx}(h)(z)-g_{yy}(h)(z)-2ig_{xy}(h)(z))dz^2.
\end{equation}
\end{defn}
When $h$ is harmonic, the Hopf differential is represented by a holomorphic quadratic differential. In the Riemannian setting, $h$ is weakly conformal if and only if $q(h)=0.$ Since a harmonic immersion to a Riemannian manifold is conformal if and only if it is minimal, we refer to a harmonic map $h$ to a metric space with $q(h)=0$ as a minimal map.

Directly from the definitions (\ref{en}) and (\ref{mhopf}), the energy density and Hopf differential of a map into a product of spaces is the sum of the energies and Hopf differentials of the component maps. Consequently, a map into a product of spaces is harmonic if and only if every component is harmonic.

Finally, we will often be concerned with maps from the universal cover $\Tilde{X}$ to $(M,d)$ that are equivariant for the deck group action of $\pi_1(X)$ on $\Tilde{X}$ and some action by isometries $\rho: \pi_1(X)\to \textrm{Isom}(M,d)$. In this case, it is easily verified that $g(h)$, and consequently $e(h)$ and $q(h)$, descend to objects on $X$. The energy and area are defined by integration over $X$ instead of $\Tilde{X}$.

\subsubsection{Maps to $\R$-trees}
In \S \ref{sec: trees} and \S \ref{sec: plateau}, the target metric space will always be an $\R$-tree or a product of $\R$-trees. In the context of  (well-behaved) maps to $\R$-trees, $e(h)$ and $q(h)$ can be defined without going through the whole theory of \cite{KorevaarSchoen}. We recall one of the many equivalent definitions of an $\R$-tree.

\begin{defn}
    An $\R$-tree is a geodesic metric space with the property that if two segments intersect in a single point, which is an endpoint of both, then their union is a segment.
\end{defn}
The completion of an $\R$-tree is also an $\R$-tree (see \cite[Theorem II.1.9]{MorganShalenDegenerationsII}), so we always implicitly complete our trees in order to apply the Korevaar-Schoen theory.

Let $\pi:\Tilde{X}\to (T,d)$ be a map to an $\R$-tree. Let $B$ denote the subset of points $p\in\Tilde{X}$ such that every neighbourhood of $p$ maps into multiple segments of $T$. Assume that $B$ has measure $0$ (this is the case when $\pi$ is harmonic). Then, around every point in $\Tilde{X}\backslash B$, we can represent $\pi$ by a map into a real interval. The map $\pi$ is $W_{\textrm{loc}}^{1,2}$ if and only if these local representations are $W_{\textrm{loc}}^{1,2}$. Moreover, $g(\pi)$ is obtained by taking the pullback metric $g(f)$ of any local representation $f$ of $\pi$ as a map to an interval; since any two identifications of a segment with an interval differ at most by a reflection and translation, this definition of $g(\pi)$ is independent of choices.
\begin{remark}
    When $\pi$ is harmonic, the set $B$ is just the zero set of the Hopf differential.
\end{remark}
For a map to a product of $n$ $\R$-trees with the same local property as above, almost everywhere we have a local representation as a map to $\R^n$. In the same fashion, one can define the measurable metric by pulling back the Euclidean metric through local representatives.

\section{The filling condition and its necessity}\label{sec: filling, necessary}

In this section, we state our filling condition and prove that straightening the horizontal and vertical measured foliations of an $L^1$ holomorphic quadratic differential on a conformally hyperbolic Riemann surface $X=\mathbb{D}/\Gamma$ gives a filling pair. We then prove some general properties of filling pairs (without assuming holomorphicity) in \S \ref{sec: properties}, which we use in \S \ref{sec: filling, sufficient}.
\begin{defn}
\label{def:filling}
Let $X=\mathbb{D}/\Gamma$ be a Riemann surface. Two measured laminations $\mu$ and $\nu$ on $X$ are {\it filling} if for every complete simple geodesic $\eta$, $$
i(\mu ,\eta )+i(\nu ,\eta )>0.
$$
\end{defn}

\subsection{The unit disk}
To make the ideas easier to follow, we first do the case $X=\mathbb{D}$ separately. For general surfaces, at some points we will re-use and modify the argument for the disk. 
\begin{thm}
\label{thm:necessary_realization}
Let $q$ be a non-trivial finite-area holomorphic quadratic differential on the unit disk $\mathbb{D}$ and let $\mu$ and $\nu$ be the measured laminations realized by the horizontal and vertical foliations respectively. Then $\mu$ and $\nu$ are filling.
\end{thm}
Throughout, we write $q=\varphi(z)dz^2$. If $\gamma$ is a curve in $\mathbb{D}$, define the $q$-length of $\gamma$ by
$$
|\gamma |_q=\int_{\gamma}|\varphi (z)|^{1/2}|dz|.
$$
The $q$-distance between two points $w_1,w_2\in\mathbb{D}$ is given by
$$
d_q(w_1,w_2)=\inf_{\gamma} \int_{\gamma}|\varphi (z) |^{1/2}|dz| =\inf_{\gamma} |\gamma |_q,
$$
where the infimum is over all rectifiable curves $\gamma$ connecting $w_1$ and $w_2$ in $\mathbb{D}$ (see \cite{Strebel}). The $q$-distance extends to the unit circle. For $\xi_1,\xi_2\in S^1$, define
$$
d_q(\xi_1,\xi_2)=\liminf_{n\to\infty} d(z_1^n,z_2^n),
$$
where the infimum is over all sequences $z_i^n\in\mathbb{D}$ with $\lim_{n\to\infty}z_i^n=\xi_i^n$ for $i=1,2$. 

\begin{lemma}[Strebel \cite{Strebel}, Lemma 19.6]
\label{lem:positive-dist}
Any two distinct points on $S^1$ have positive $q$-distance. 
\end{lemma}

Using the notion of the distance between points on $S^1$, we prove that the intersection between an arbitrary hyperbolic geodesic $\eta$ and either $\mathcal{F}_h(q)$ or $\mathcal{F}_v(q)$ is positive. Since $$
i(\mathcal{F}_h(q),\eta)+i(\mathcal{F}_v(q),\eta)= i(\mu ,\eta)+i(\nu ,\eta)>0 ,
$$ 
Theorem \ref{thm:necessary_realization} follows immediately from the proposition below.

\begin{prop}
\label{prop:inter-positive}
Let $q=\varphi (z)dz^2\not\equiv 0$ be a finite-area holomorphic quadratic differential and let $\eta$ be an arbitrary hyperbolic geodesic in $\mathbb{D}$ with endpoints $\xi_1,\xi_2\in S^1$. Then
$$
0<d_q(\xi_1,\xi_2)\leq i(\eta,\mathcal{F}_v(q))+i(\eta,\mathcal{F}_h(q)).
$$
\end{prop}

\begin{proof}
Let $d_q(\xi_1,\xi_2)=d >0$. It is not immediate that $d_q>0$ implies that at least one of the two intersections $i(\eta,\mathcal{F}_h(q))$, $i(\eta,\mathcal{F}_v(q))$ is non-zero because the two intersection numbers may be approximated over two different sequences of curves homotopic to $\eta$. 

Let $\epsilon>0$ be small, and let $1>r_1,r_2>0$ be two radii such that the $q$-lengths of 
$\{ z:|z-\xi_1|=r_1\}\cap\mathbb{D}$ and $\{ z:|z-\xi_2|=r_2\}\cap\mathbb{D}$ are both less than $\epsilon$. Such choice of radii exist because $\int_{\mathbb{D}}|\varphi (z)|dxdy<\infty$.  
Let $\gamma$ be an arc connecting $\{ z\in\mathbb{D}:|z-\xi_1|=r_1\}$ and $\{ z\in\mathbb{D}:|z-\xi_2|=r_2\}$ in $\mathbb{D}\setminus (\{ z\in\mathbb{D}:|z-\xi_1|\leq r_1\}\cup \{ z\in\mathbb{D}:|z-\xi_2|\leq r_2\})$ that is a concatenation of finitely many $\theta$-arcs of $q$, i.e., $\gamma$ is piecewise straight. 

Consider the set of all horizontal arcs that have both endpoints on $\gamma$. If $\delta$ is one such arc, then $\delta$ together with a subarc $\gamma (\delta )$ of $\gamma$ is a closed Jordan curve which bounds a Jordan domain $\Omega_{\delta}$ (see Figure \ref{fig:partial-order}). Then any horizontal ray that starts on $\gamma (\delta )$ and enters $\Omega_{\delta}$ meets $\gamma (\delta )$ in another point (see Figure \ref{fig:partial-order}). There is a partial order on these horizontal arcs. For a horizontal arc $\delta_1$ with both endpoints on $\gamma$, we say that $\delta \prec \delta_1$ if $\Omega_{\delta}\subset\Omega_{\delta_1}$. Each $\delta$ belongs to a maximal family $\mathcal{P}_{\delta}$ of horizontal arcs with both endpoints on $\gamma$ such that for any $\delta_1,\delta_2\in\mathcal{P}_{\delta}$, either $\delta_1\prec\delta_2$ or $\delta_2\prec\delta_1$ (this follows since two horizontal arcs cannot intersect if they are different).  

\begin{figure}[htb]
    \centering
    \includegraphics[width=7 cm]{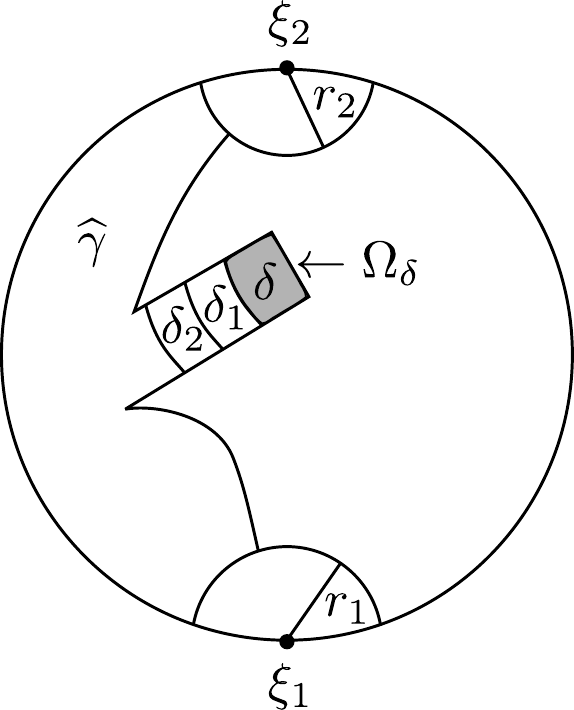}
    \caption{The partial order on $\mathcal{P}_{\delta}$.}
    \label{fig:partial-order}
\end{figure}

Since each $\delta$ cuts out an arc of $\gamma$ with interior, there are at most countably many ordered families $\{\mathcal{P}_i\}_{i=1}^{\infty}$. If a family $\mathcal{P}_i$ has a maximal element, then we denote it by $\delta^i$ and replace the arc $\gamma (\delta^i)$ with $\delta^i$. Note that this process shortens the length of the curve. It can happen that $\mathcal{P}_i$ has no maximal element. Then we choose $\delta_i\in\mathcal{P}_i$ such that the total transverse measure (the vertical measure $\int|Im(\sqrt{\varphi (z)}dz|$) of all $\delta\in\mathcal{P}_i$ for which $\delta^i\prec\delta$ is less than $\epsilon /2^i$. We replace $\gamma (\delta^i)$ with $\delta^i$. Therefore, we replaced countably many disjoint subarcs of $\gamma$ with horizontal arcs and obtained a curve $\widehat{\gamma}$. By the construction, the total transverse measure of all horizontal arcs with both endpoints on $\widehat{\gamma}$ is at most $\epsilon$. We emphasize here that a set of horizontal trajectories may intersect $\widehat{\gamma}$ more than twice, which will result in the transverse measure of a single horizontal trajectory being taken more than once. We proved that this total transverse measure with the repetitions is less than $\epsilon$.

Consider a vertical arc $\omega$ with both endpoints on $\widehat{\gamma}$. If both endpoints are not on the added horizontal arcs (i.e., they are on $\gamma\setminus\cup_i\delta^i$), then $\omega$ is also a vertical arc with both endpoints on $\gamma$. Suppose at least one endpoint of $\omega$ is on $\delta^i$. In that case, $\omega$ can be extended to a vertical arc $\omega'$ by entering $\Omega_{\delta^i}$ until it hits the original curve $\gamma$ (which must happen because the Teichm\"uller lemma \cite[Theorem 14.1, page 71]{Strebel} prevents it from intersecting the horizontal arc $\delta^i$ twice). If needed, we extend $\omega$ at the other endpoint. Therefore, for each $\omega$ there is an extension $\omega'$ with both endpoints on $\gamma$. Thus, the total transverse measure of all vertical arcs with both endpoints on $\widehat{\gamma}$ is the same as that of $\gamma$, while the transverse measure of horizontal arcs connecting $\widehat{\gamma}$ to itself is less than $\epsilon$. 

Next, we perform the analogous process on $\widehat{\gamma}$ for the vertical arcs connecting it to itself. The total transverse measure of these vertical arcs is less than $\epsilon$, as well as the total transverse measure of the horizontal arcs. For simplicity of notation, denote the new curve $\widehat{\gamma}$.

Let arcs $I_1$ and $I_2$ be the two components of $S^1\setminus ( \{ z:|z-\xi_1|\leq r_1\}\cup\{ z:|z-\xi_2|\leq r_2\})$.
Consider all horizontal trajectories that intersect $\widehat{\gamma}$. The set of horizontal trajectories that intersect $\widehat{\gamma}$ more than once has transverse measure less than $\epsilon$. Any horizontal trajectory that intersects $\widehat{\gamma}$ only once can either intersect $\{ z:|z-\xi_1|= r_1\}\cap\mathbb{D}$; or it can intersect $\{ z:|z-\xi_2|= r_2\}\cap\mathbb{D}$; or it has one endpoint in $I_1$ and the other endpoint in $I_2$. The total transverse measure of the set of all trajectories in the above cases, except the last case, is at most $2\epsilon$.

Let $|\gamma |_{vert}=\int_{\gamma} |Im(\sqrt{\varphi (z)}dz)|$ and $|\gamma |_{hor}=\int_{\gamma} |Re(\sqrt{\varphi (z)}dz)|$. The above gives
\begin{equation}
\label{eq:intersection-ineq-hor}
|\widehat{\gamma}|_{vert}\leq \mathcal{F}_h(q)(I_1\times I_2)+3\epsilon\leq i(\eta,\mathcal{F}_h(q))+3\epsilon,
\end{equation}
where $\mathcal{F}_h(q)(I_1\times I_2)$ is the total transverse measure of the leaves of the foliation $\mathcal{F}_h(q)$ with one endpoint in $I_1$ and the other in $I_2$.  

An analogous argument for the same curve $\widehat{\gamma}$ gives
\begin{equation}
\label{eq:intersection-ineq-vert}
|\widehat{\gamma}|_{hor}\leq \mathcal{F}_v(q)(I_1\times I_2)+3\epsilon\leq i(\eta,\mathcal{F}_v(q))+3\epsilon.
\end{equation}

A crucial point in this proof is that the equations (\ref{eq:intersection-ineq-hor}) and (\ref{eq:intersection-ineq-vert}) hold for the same curve $\widehat{\gamma}$ constructed above. By integrating (along $\widehat{\gamma}$) the inequality 
$$
|\sqrt{\varphi (z)}dz|\leq |Re(\sqrt{\varphi (z)}dz)|+|Im(\sqrt{\varphi (z)}dz)|
$$ 
and using (\ref{eq:intersection-ineq-hor}) and (\ref{eq:intersection-ineq-vert}), we conclude that
$$
|\widehat{\gamma}|_{q}\leq i(\eta,\mathcal{F}_v(q))+i(\eta,\mathcal{F}_h(q))+6\epsilon .
$$
Since $\epsilon>0$ was arbitrary, we can carry out the process above over any positive sequence $(\epsilon_n)_{n=1}^\infty$ tending to zero. By integrability of $q$ there exist radii $(r_1^n)_{n=1}^\infty$, $(r_2^n)_{n=1}^\infty$ tending to zero simultaneously with $\epsilon_n$, and such that the limiting $q$-length of the corresponding curves $\widehat{\gamma}^n$ is bounded below by $d_q(\xi_1,\xi_2)$. Taking $n\to\infty$, we have
$$
d_q(\xi_1,\xi_2)\leq i(\eta,\mathcal{F}_v(q))+i(\eta,\mathcal{F}_h(q)).
$$
\end{proof}

\subsection{General surfaces}
Here we prove necessity of the filling condition in Theorem \ref{thm: main}.

\begin{thm}
\label{thm:necessary_realization_first_kind}
Let $q$ be a non-trivial finite-area holomorphic quadratic differential on a Riemann surface $X=\mathbb{D}/\Gamma$. Let $\mu$ and $\nu$ be measured laminations obtained by straightening the horizontal and vertical foliations of $q$. Then, $\mu$ and $\nu$ are filling.
\end{thm}
\begin{proof}
    We divide the proof into four cases based on the properties of a complete simple geodesic $\eta$: either $\eta$ does not intersect the convex core of $X$ and it is not asymptotic to the boundary of the convex core, or $\eta$ does not intersect the convex core and one of its rays is asymptotic to the boundary of the convex core, or $\eta$ intersects the interior of the convex core, or $\eta$ is on the boundary of the convex core.

\vskip .2 cm
    
    \noindent {\bf Case 1.} The first case is when $\eta$ does not intersect the convex core of $X$ and it is not asymptotic to the boundary of the convex core. Then both endpoints of $\eta$, denoted by $\xi_1$ and $\xi_2$ belong to a single ideal boundary component of $X$, which is either a closed arc in the case of a funnel or an open arc in the case of a hyperbolic half-plane. Let $I$ be the closed interval on the ideal boundary $\partial_{\infty}X$ of $X$ that faces $\eta$ with endpoints $\xi_1$ and $\xi_2$. Let $I_1$ be another closed interval on the ideal boundary $\partial_{\infty}X$ that contains $I$ in its interior. 

We take a simply connected domain $\Omega$ in $X$ which contains $\eta$ and the interval $I_1\subset \partial_{\infty} X$ on its boundary. We restrict $q$ to $\Omega$ and note that the restriction $q_{\Omega}$ is of finite-area on $\Omega$. Since $\Omega$ is conformal to $\mathbb{D}$, Proposition \ref{prop:inter-positive} holds for the restriction $q_\Omega$ of $q$ to $\Omega$.   
Thus we have
\begin{equation}
\label{eq:ineq_Omega}
0<d_{q_\Omega}(\xi_1 ,\xi_2 )\leq i(\eta  ,\mathcal{F}_v(q_{\Omega}))+i(\eta ,\mathcal{F}_h(q_{\Omega})), 
\end{equation}
where $i(\eta  ,\mathcal{F}_v(q_{\Omega}))$ is the intersection number computed with respect to the homotopy class of $\eta$ in $\Omega$ and the measured foliation $\mathcal{F}_v(q_{\Omega})$ in $\Omega$. 
However, the quantity $i(\eta ,\mathcal{F}_v(q_{\Omega}))$ (which is computed in $\Omega$) could be strictly larger than the intersection number $i(\eta ,\mathcal{F}_v(q))$ on $X$, since the latter might be approximated only by homotopic curves that exit $\Omega$. 

\begin{figure}[htb]
    \centering
    \includegraphics[width=8 cm]{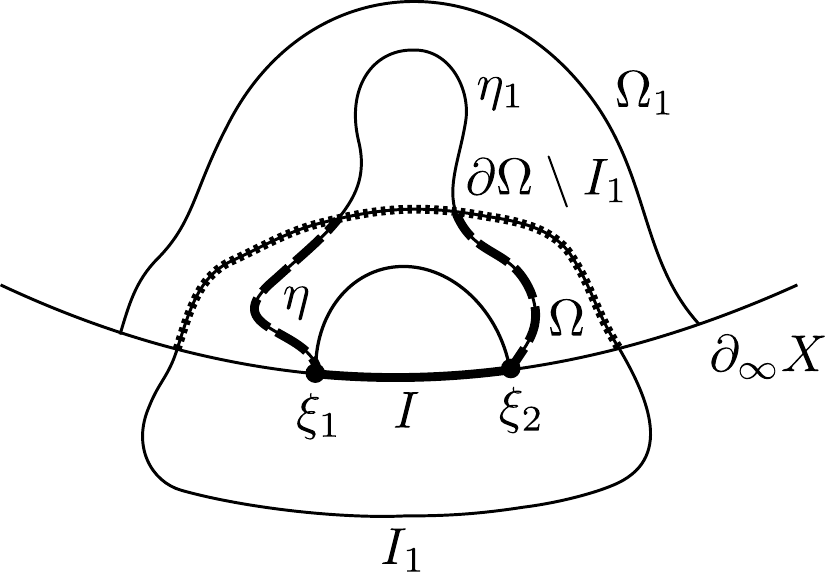}
    \caption{The domain $\Omega$ and curve $\eta_1$ outside $\Omega$ and inside $\Omega_1$.}
    \label{fig:omega}
\end{figure}

Let $\eta_1$ be an arbitrary curve homotopic to $\eta$. When $\eta_1$ is in $\Omega$, the inequality (\ref{eq:ineq_Omega}) holds and the left-hand side is positive and independent of the choice of $\eta_1$. It remains to prove that for all $\eta_1$ that leave $\Omega$, there is a single positive lower bound on the sum of the intersection numbers.

Assume $\eta_1$ is not contained in $\Omega$. Then we form a larger simply connected domain $\Omega_1\subset X$ which contains $\Omega$ and $\eta_1$ (and $\Omega_1$ will necessarily contain $I_1$, $\xi_1$, and $\xi_2$ on its boundary, see Figure \ref{fig:omega}). Using 
the construction in the proof of Proposition \ref{prop:inter-positive}, we replace $\eta_1$ with a curve $\hat{\eta}_1$ in $\Omega_1$ made by concatenation of $\theta$-arcs of $q$ such that the total transverse measure of the set of points in $\hat{\eta}_1$ that intersect a single horizontal arc more than once is less than $\epsilon$ and the same for the multiple intersections with the vertical arcs, where $\epsilon >0$ is arbitrary. 

Let $c_1$ and $c_2$ be two subarcs of $\hat{\eta}_1$ that connect $\xi_1$ and $\xi_2$ to the boundary of $\Omega$ while staying in $\Omega$. Let $z_1$ and $z_2$ be the other two endpoints of $c_1$ and $c_2$. By the proof of Proposition \ref{prop:inter-positive}, we have
\begin{equation*}
d_{q_{\Omega}}(\xi_j,z_j)\leq|c_j|_{vert}+|c_j|_{hor}\leq i(\eta_1  ,\mathcal{F}_h(q_{\Omega_1}))+i(\eta_1 ,\mathcal{F}_v(q_{\Omega_1}))+6\epsilon
\end{equation*}
for $j=1,2$. 
The points $z_1$ and $z_2$ are on $\partial\Omega\setminus I_1$. Since $d_{q}|_\Omega$ is lower semi-continuous, $\xi_j\in I$, $z_j\in \partial\Omega\setminus I_1$ and $I\cap \overline{\partial\Omega\setminus I_1}=\emptyset$,  it follows that 
$$
d_{q_{\Omega}}(\xi_j,z_j)\geq m_j:=\inf_{z\in \partial\Omega\setminus I_1}d_{q_{\Omega}}(\xi_j,z)>0$$
for $j=1,2$. 

By using the above two inequalities and letting $\epsilon \to 0$, we obtain
$$
0<m_1+m_2\leq i(\eta_1  ,\mathcal{F}_h(q_{\Omega_1}))+i(\eta_1 ,\mathcal{F}_v(q_{\Omega_1})).
$$
Since the left-hand side is independent of $\Omega_1$, $\inf_{\Omega_1\supset\Omega} i(\eta_1  ,\mathcal{F}_h(q_{\Omega_1})) =i(\eta , \mathcal{F}_h(q))$, and $\inf_{\Omega_1\supset\Omega} i(\eta_1  ,\mathcal{F}_v(q_{\Omega_1})) =i(\eta , \mathcal{F}_v(q))$ we conclude
$$
0<m_1+m_2\leq i(\eta , \mathcal{F}_h(q))+ i(\eta , \mathcal{F}_v(q)).$$

To recap, in this case, we did not find a single curve that almost realizes the intersection numbers but rather we used Proposition \ref{prop:inter-positive} to estimate the sum of the intersection numbers by the distance in the $q$-metric of the endpoints $\xi_1,\xi_2$ over larger and larger domains in $X$. For a fixed domain, the distance is positive and if a curve (between $\xi_1$ and $\xi_2$) goes deep into $X$ it still has to have subarcs connecting $\xi_1$ and $\xi_2$ to a compact arc on the boundary of the fixed domain $\Omega$, forcing a positive lower bound.

\vskip .2 cm
    
    \noindent {\bf Case 2.}
For the second case, the geodesic $\eta$ does not intersect the convex core of $X$ and one of its rays is asymptotic to a boundary component of the convex core. When this boundary component is an open geodesic, the proof follows the steps from the proof when $\eta$ is the open geodesic on the boundary of the convex core that we give later (see Case 4). When the ray of $\eta$ is asymptotic to a closed geodesic $\gamma$ on the boundary of the convex core, it eventually enters the half-collar of $\gamma$ (coming from the funnel side). Choosing two equidistant curves to $\gamma$ contained in this half-collar, $\eta$ has a subarc properly crossing the compact annular region between them. Every curve homotopic to $\eta$ also crosses this annulus. The argument that we give below in Case 3, applied to curves crossing this annulus, gives a positive lower bound on the sum of the intersection numbers.

\vskip .2 cm

\noindent {\bf Case 3.} The third case is when a simple geodesic $\eta$ intersects the convex core of $X$. This includes geodesics that have an endpoint at a puncture or on an ideal boundary, or are completely contained in the convex core.
The geodesic $\eta$ essentially intersects at least one simple closed geodesic, or it has at least one end going into a puncture (example: a pair of pants), or both things occur. 

We first treat the case where it crosses a simple closed geodesic $\gamma$. Then, it also crosses the standard collar neighborhood $C_\gamma$ of $\gamma$ and any curve ${\eta}_1$ homotopic to $\eta$ modulo its ideal endpoints also crosses the standard collar $C_\gamma$ because the intersection cannot disappear in the homotopy. 
Let $\eta_1^{\gamma}$
be a closed subarc of $\eta_1$ with endpoints on the two boundaries of the collar of $\gamma$ that otherwise stays inside of the collar. We replace $\eta_1^{\gamma}$ with a homotopic (modulo endpoints) curve that is a concatenation of $\theta$-arcs in $C_{\gamma}$ and that does not increase the horizontal and the vertical lengths. Then we choose a simply connected domain in $C_\gamma$ which contains $\eta_1^\gamma$ and two arcs on the boundary of $C_{\gamma}$ around the endpoints of 
$\eta_1^\gamma$. Using the construction in the proof of Proposition \ref{prop:inter-positive} we replace  $\eta_1^\gamma$ with a curve which is within $\epsilon >0$ of the intersection numbers with $\mathcal{F}_h(q)$ and $\mathcal{F}_v(q)$ of the curve $\eta_1^\gamma$. Then the sum of these intersection numbers is bounded below by the distance of its endpoints. There is a positive lower bound on the $q$-distance between the two boundary components of $C_{\gamma}$ and we conclude that the intersection of any $\eta_1$ with $\mathcal{F}_h(q)$ and $\mathcal{F}_v(q)$ inside $C_\gamma$ (hence, in the whole surface) is bounded below by a positive constant. 

In the case where the geodesic $\gamma$ does not cross a simple closed curve but has an end going into a puncture, we can choose two equidistant curves contained in a horoball neighbourhood of the puncture, bounding an annulus, and then every curve homotopic to $\gamma$ must cross the annulus. The argument from above shows that we have a positive intersection inside the annulus, hence in the whole surface.

\vskip .2 cm
    
    \noindent {\bf Case 4.} For this case, assume first that $\eta$ is the open geodesic boundary of a geodesic half-plane in $X$ with ideal endpoints $\xi_1$ and $\xi_2$. We denote by $I$ the ideal boundary arc facing $\eta$ which is also the ideal boundary of the half-plane bounded by $\eta$. Let $\Omega$ be a simply connected region in $X$ that has $I$ on its boundary and contains $\eta$ in its interior. The restriction $q_{\Omega}$ of $q$ to $\Omega$ is of finite-area. As in the proof of Proposition \ref{prop:inter-positive} we replace $\eta$ with a piecewise $\theta$-arc $\widehat{\eta}$ such that for some $\epsilon >0$ we have
$$
0<d_{q_\Omega }(\xi_1,\xi_2)\leq |\widehat{\eta}|_{vert} +|\widehat{\eta}|_{hor}\leq  i(\eta_\Omega ,\mathcal{F}_v(q_{Omega}))+i(\eta_\Omega ,\mathcal{F}_h(q_{Omega}))+6\epsilon .
$$
Recall that the property of the arc $\widehat{\eta}$ is that it almost realizes both intersection numbers in the domain $\Omega$. 

Let $\eta_1$ be homotopic to $\eta$ modulo endpoints. If $\eta_1$ is in $\Omega$ then the above inequality holds and there is a positive lower bound on the sum of the intersection numbers.

Assume that $\eta_1$ is not completely contained in $\Omega$ and let $\Omega_1$ be a simply connected domain in $X$ that contains $\Omega$ and $\eta_1$. Then by the construction in the proof of Proposition \ref{prop:inter-positive} we can replace $\eta_1$ with $\widehat{\eta}_1$ which satisfies the analogous inequality as above only applied to the domain $\Omega_1$.

If a horizontal trajectory of $q$ intersects $\widehat{\eta}_1$ in only one point, then it has to intersect $\widehat{\eta}$ or end in $\{\xi_1,\xi_2\}$ (the later has zero transverse measure since $\{\xi_1,\xi_2\}$ are not endpoints of a closed geodesic, see \cite[Theorem 2(c)]{MardenStrebel1}). The total measure of such horizontal trajectories that intersect $\widehat{\eta}$ in only one point is at most $|\widehat{\eta}_1|_{vert}-\epsilon$ since by Lemma \ref{lem:asymptotic} a set of trajectories of positive transverse measure cannot end at an endpoint of $\eta$ and the transverse measure of horizontal trajectories intersecting $\widehat{\eta}$ in more than one point is at most $\epsilon$. A completely analogous statement holds for the vertical trajectories that intersect $\widehat{\eta}_1$ in one point.
Thus,
\begin{equation*}
\begin{split}
i(\eta_{\Omega_1} ,\mathcal{F}_v(q_{\Omega_1}))+i(\eta_{\Omega_1} ,\mathcal{F}_h(q_{\Omega_1}))\geq |\widehat{\eta}_1|_{vert} +|\widehat{\eta}_1|_{hor}-6\epsilon\\ \geq |\widehat{\eta}|_{vert} +|\widehat{\eta}|_{hor}-8\epsilon\geq d_\Omega (\xi_1,\xi_2)-8\epsilon >0.
\end{split}
\end{equation*}
By letting $\epsilon\to 0$ we obtain a positive bound $d_\Omega (\xi_1,\xi_2)$ over all $\eta_1$ and $\Omega_1$. Therefore,
$$
i(\mathcal{F}_h(q),\eta )+i(\mathcal{F}_v(q),\eta )>0.
$$
In the same case, we now assume that $\eta$ is a closed geodesic on the boundary of the convex core. In this case, we draw a strip bounded by two geodesics that intersect $\gamma$ and have both endpoints at the ideal closed curve facing $\gamma$ (the strip can be drawn in the union of the funnel and a single geodesic pair of pants that has $\gamma$ on its boundary). Then any curve homotopic to $\gamma$ must have two subarcs in the strip  connecting the two boundary geodesics of the strip. We apply the argument from Case 3 to these subarcs. Note that the $q$ distance between the geodesics is bounded below by a positive constant because the distance between the endpoints is positive. The theorem follows.
\end{proof}

\subsection{Properties of filling pairs}\label{sec: properties}
\begin{prop}
\label{prop:component_closures}
Let $X=\mathbb{D}/\Gamma$ be a Riemann surface, let $\mu$ and $\nu$ be measured laminations on $X$ that are realized by integrable partial measured foliations, and let $\Tilde{\mu}$ and $\Tilde{\nu}$ be the lifted laminations on  $\mathbb{D}$. Assume that $\mu$ and $\nu$ fill. Let $K$ be a component of $X\backslash(|\mu|\cup |\nu|)$ and let $C$ be a lift to $\mathbb{D}$. Then,
    \begin{enumerate}
\item	if $C$ has compact closure $\overline{C}$, then $\overline{C}$ is a finite-sided geodesic polygon, with boundary components alternating between $\Tilde{\mu}$ and $\Tilde{\nu}$, and
\item if $C$ has non-compact closure $\overline{C}$, then $\overline{C}\cap S^1$ is a single point.
    \end{enumerate}
\end{prop}

\begin{proof}
We first prove that $C$ is convex. Indeed, for every geodesic $g$ in  $|\Tilde{\mu} |\cup |\Tilde{\nu}|$, $C$ lies in one of the half-spaces bounded by $g$, say $H_g^C$. Then, $$C=\bigcap_{g\in |\Tilde{\mu}|\cup |\Tilde{\nu}|} H_g^C=:B.$$ To see this: the inclusion $C\subset B$ is obvious, and for the other inclusion, if we fix $x\in C$ and take any $y\in B$, then the segment $[x,y]$ lies in $B$ and meets no geodesic in $|\mu|\cup |\nu|$, and hence $y$ is in the same connected component as $x$, i.e., $y \in C$.

If $\overline{C}$ is compact then it is bounded by finitely many arcs on $|\tilde{\mu} |\cup|\tilde{\nu} |$ and the boundary geodesics have to alternate, since $\tilde{\mu}$ and $\tilde{\nu}$ cannot intersect themselves. Item (1) in the proposition is now immediate.

To prove (2), assume that $C$ has at least two points on $S^1$. Let $\tilde{\eta}$ be the geodesic in $\mathbb{D}$ connecting the two points of the ideal boundary of $C$. Then the quotient geodesic $\pi(\Tilde{\eta})$ in $K\subset X$ is disjoint from $|\mu |\cup |\nu |$. By removing loops, we obtain a simple curve $\eta$ in $K$ that is disjoint from $|\mu |\cup |\nu |$ which has the same asymptotic behaviour (either ideal points, or recurrent behaviour). It is homotopic to a geodesic, and since geodesics realize the geometric intersection number, it does not intersect $|\mu|$ or $|\nu|$. This contradicts the filling property of $\mu$ and $\nu$ and hence proves (2). 
\end{proof}

\begin{prop}
\label{prop:endpoints_S^1}
    Let $X=\mathbb{D}/\Gamma$ be a Riemann surface, and let $\mu$ and $\nu$ be measured laminations on $X$ that are realized by integrable partial measured foliations. If $\mu$ and $\nu$ fill, then the set of endpoints of the lifts $\Tilde{\mu}$ and $\Tilde{\nu}$ to $\mathbb{D}$ are dense in $S^1$.
\end{prop}
\begin{proof}
    If the set of endpoints omits an interval $I\subset S^1$, then any geodesic with endpoints in $I$ determines a complementary region of $|\Tilde{\mu}|\cup |\Tilde{\nu}|$ with non-compact closure. This contradicts item (2) in the proposition above.
\end{proof}

\section{Sufficiency of the filling condition}\label{sec: filling, sufficient}

In this section we prove that filling is sufficient for holomorphic realization. As in \S \ref{sec: filling, necessary}, we first treat the case of the unit disk, and then we do the general case. We could merge into one proof, but we keep things separate (and refer to previous arguments where needed) so that the proofs are more digestible.

\subsection{Partitions on $\hat{\mathbb{C}}$}
Preparing for the main proofs, we review upper semi-continuous partitions.

\begin{defn}
    A partition $\mathcal{G}$ of $\widehat{\mathbb{C}}$ is {\it upper semi-continuous} if for any $G\in \mathcal{G}$ and any open set $U\supset G$ there is an open set $V$, which is a union of elements of $G$, such that $U\supset V\supset G$. 
\end{defn}
The usefulness of the notion comes from Moore's theorem \cite{moore}, which says that the quotient of $\widehat{\mathbb{C}}$ obtained by collapsing elements of (an upper semicontinuous decomposition by non-separating continua) $\mathcal{G}$ to points is homeomorphic to $\widehat{\mathbb{C}}$.

For $\mathcal{G}$ any partition of $\widehat{\mathbb{C}}$ and $z\in\widehat{\mathbb{C}}$, let $G(z)\in\mathcal{G}$ be the element of the partition which contains $z$. 
\begin{prop}
\label{prop:semi-c}
Let $\mathcal{G}$ be a partition of $\widehat{\mathbb{C}}$ by closed sets. 
     Then $\mathcal{G}$ is {\it upper semi-continuous} if and only if
     for any two sequences $z_n,w_n\in \widehat{\mathbb{C}}$ such that $G(z_n)=G(w_n)$, if $z_n\to z$ and $w_n\to w$, then $G(z)=G(w)$. 
\end{prop}

\begin{proof}
Assume first that $\mathcal G$ is upper semi-continuous. If $w\notin G(z)$, let $D_r(w)$ be an open disk in the spherical metric of some positive radius $r$ such that $G(z)\cap D_r(w)=\emptyset$. The choice of such $r>0$ is possible since $G(z)$ is closed by assumption and $w\notin G(z)$. Let $U$ be the complement of the closed disk $\overline{D_{r/2}(w)}$. Then
\[
G(z)\subset U,\qquad w\notin \overline{U}.
\]
By upper semi-continuity, there is an open set $V$ with $G(z)\subset V\subset U$ such that every element of $\mathcal G$ meeting $V$ is contained in $V$. For all sufficiently large $n$, we have $z_n\in V$, hence $G(z_n)\subset V$ and therefore $w_n\in V$. Thus $w\in\overline V\subset \overline U$, a contradiction.

Conversely, assume the sequential condition.
We use the decomposition $\mathcal{G}$ to define a subset of $\mathbb{C}^2$ by
\[
\mathcal{R}=\{(z,w):G(z)=G(w)\}.
\]
The set $\mathcal{R}$
is closed in $\mathbb{C}^2$. Indeed, if $(z_n,w_n)\in \mathcal{R}$ and
$(z_n,w_n)\to (z,w)$, then the sequential condition implies $(z,w)\in \mathcal{R}$.

Let $G_0\in\mathcal G$ and let $U$ be an open set containing $G_0$. Put
$F=\widehat{\mathbb C}\setminus U$ and let
\[
\operatorname{Sat}(F)
 =\{z\in\widehat{\mathbb C}:G(z)\cap F\neq\emptyset \}.
\]
We have
\[
\operatorname{Sat}(F)
 =\pi_1\bigl(\mathcal{R}\cap(\widehat{\mathbb C}\times F)\bigr),
\]
where $\pi_1$ is projection onto the first factor. Since $\mathcal{R}$ and $F$ are
closed and $\widehat{\mathbb C}^{\,2}$ is compact, $\operatorname{Sat}(F)$
is compact, hence closed. Therefore
\[
V=\widehat{\mathbb C}\setminus\operatorname{Sat}(F)
\]
is open and $G_0\subset V$. The set $V$ is the union of elements of $\mathcal{G}$ because if $G\cap V\neq\emptyset$ then $G\not\subset \operatorname{Sat}(F)$ which implies $G\cap F=\emptyset$ and, by definition of $\operatorname{Sat}(F)$, also $G\cap \operatorname{Sat}(F)=\emptyset$. Thus every open
set containing $G_0$ contains an open subset containing $G_0$ which is union of elements of $\mathcal{G}$, so
$\mathcal G$ is upper semi-continuous.
\end{proof}

\subsection{Realization on $\mathbb{D}$}
The goal of this subsection is to prove the following.
\begin{thm}
\label{thm:realization_sufficient}
   Let $\mu$ and $\nu$ be a pair of filling measured laminations on $\mathbb{D}$ that are realizable by integrable (partial) measured foliations on $\mathbb{D}$.
   
Then there exists a homeomorphism $h:S^1 \to S^1$ and a finite-area holomorphic quadratic differential $q$ such that the push-forward measured laminations $h_*(\mu )$ and $h_*(\nu )$ are homotopic to the horizontal and vertical foliations of $q$.
\end{thm}

Let $\mu$ and $\nu$ be as in the statement. We define a partition $\mathcal{G}_{\mu ,\nu}$ of the Riemann sphere $\widehat{\mathbb{C}}$ as follows. The first kind of elements of the partition are the closures of the components of the complement of $|\mu |\cup |\nu |$ in $\mathbb{D}\cup S^1$ (see Figure \ref{fig:partition1}). 
Next we consider the components of $|\mu |\setminus |\nu |$ and $|\nu |\setminus |\mu |$. Each component is either an arc on a geodesic (of either $|\mu |$ or $|\nu |$) bounded by two points of $|\mu |\cap |\nu |$, or a ray (on a geodesic of either $|\mu |$ or $|\nu |$) with initial point in $|\mu |\cap |\nu |$ and an ideal endpoint on $S^1$. The second kind of elements of the partition are the closures in $\mathbb{D}\cup S^1$ of these arcs and rays (see Figure \ref{fig:partition2}), excluding those contained in closures of components of the complement of $|\mu |\cup |\nu |$. The third kind of elements are the point sets of the complements of the union of the first two kinds of elements (see Figure \ref{fig:partition3}). In particular, all points of $\widehat{\mathbb{C}}\setminus (\mathbb{D}\cup S^1)$ are elements of the partition. 
Also, this kind of element includes all points of $|\mu |\cap |\nu |$ that are accumulated from both sides by geodesics of $|\mu |$ and by geodesics of $|\nu |$. Equivalently, $z=g_1\cap g_2$ for $g_1\in |\mu |$ and $g_2\in |\nu |$ is not an element of the partition if and only if either $g_1\setminus |\nu |$ has an open arc or ray with endpoint $z$, or $g_2\setminus |\mu |$ has an open arc or ray with endpoint $z$. Note that elements of $\mathcal{G}_{\mu ,\nu}$ are non-separating closed sets in $\widehat{\mathbb{C}}$.  

To see that $\mathcal{G}_{\mu,\nu}$ is a genuine partition (i.e., the elements do not intersect), we just need to show that elements of the first or second kind don't intersect other elements of the first or second kind. It is clear from the definitions that such elements don't intersect at points in $\mathbb{D}$. If we have two such components $A$ and $B$ intersecting at a point $p\in S^1$, first note that the complement contains an open subset of the disk accumulating at $p$. Indeed, two complementary components cannot meet along a single geodesic ending at $p$, since the laminations have no atoms, and for other types of components this is obvious. The components $A$ and $B$ are therefore accumulated by leaves of $|\mu|$ or $|\nu|$ all ending at $p$, and this contradicts Lemma \ref{lem:asymptotic}.

\begin{lem}
\label{lem:UpperC}
    The partition 
    $\mathcal{G}_{\mu ,\nu}$ of $\widehat{\mathbb{C}}$ is upper semi-continuous partition by non-separating continua.
\end{lem} 

\begin{proof}
    We first note that all the components of $\mathcal{G}_{\mu ,\nu}$ are non-separating, connected and closed subsets of $\widehat{\mathbb{C}}$ by the definition. Therefore, we can apply the sequential condition from Proposition \ref{prop:semi-c}. The proof is elementary but it is split into several cases. Lemma \ref{lem:asymptotic} implies that $\mu$ and $\nu$ have no atoms and each geodesic in their support $|\mu |$ and $|\nu |$ is accumulated from at least one side by other geodesics of the support which do not have common endpoints.

    {\bf Case 1. Either $G(z)=\{ z\}\subset \mathbb{D}$ or $G(w)=\{ w\}\subset \mathbb{D}$.} Without loss of generality, assume that $G(z)=\{ z\}\subset \mathbb{D}$. Let $g_z\in |\mu |$ be the geodesic of $|\mu |$ containing $z$ and let $h_z\in |\nu |$ be the geodesic of $|\nu |$ containing $z$. We fix an orientation of $g_z$ and $h_z$. Then there exist sequences $g_k^l\in |\mu |$ and $g_k^r\in |\mu |$ which are on the left and the right side of $g_z$ and converge to $g_z$ as $k\to\infty$. Similar for sequences $h_k^l\in |\nu |$ and $h_k^r\in |\nu |$, and the geodesic $h_z$. For $k$ large enough, the geodesics $(g_k^l,g_k^r,h_k^l,h_k^r)$ determine a quadrilateral $R_k$ that contains $z$ in its interior. There exists $n(k)$ such that $z_n\in R_k$ for all $n\geq n(k)$. Then $G(z_n)\subset R_k$ since the boundary of $R_k$ is the union of the components of $\mathcal{G}_{\mu ,\nu}$. Since $G(z_n)=G(w_n)$, it follows that $w_n\in R_k$. Therefore $w_n\to z$, $z=w$, and $G(z)=G(w)$.

    {\bf Case 2. Either $G(z)$ or $G(w) $ is a closed ray (including the ideal endpoint on $S^1$) or a closed arc on a geodesic of $|\mu |$ or of $|\nu |$.} Without loss of generality, assume that is the case for $G(z)$ and $|\mu |$. When $G(z)$ is a closed arc on a geodesic $g_z\in |\mu |$, let $h^l\in |\nu |$ and $h^r\in |\nu |$ be the two geodesics through the endpoints of the closed arc $G(z)\subset g_z$. 
    Let $h_k^l\in |\nu |$ and $h_k^r\in |\nu |$ be the geodesics accumulating to $h_l$ and $h_r$. 
    Let $g_k^l$ and $g_k^r$ be the two sequences converging to $g_z$ from opposite sides. Then we consider the quadrilateral $R_k$ determined by $(g_k^l,g_k^r,h_k^l,h_k^r)$. Then, for $n$ large enough, we have that $z_n\in R_k$ and thus $G(z_n)\subset R_k$. Since $R_k$ accumulates to the arc $G(z)$ and $w_n\in G(z_n)$, we conclude that $w\in G(z)$ and $G(z)=G(w)$. When $G(z)$ is a ray on $g_z\in |\mu |$, then $R_k$ will be bounded by an arc on $h_k^l\in |\nu |$, and two rays on $g_k^l\in |\mu |$ and $g_k^r\in |\mu |$. The same argument gives $w\in G(z)$ and $G(z)=G(w)$.

    {\bf Case 3. Either $G(z)$ or $G(w) $ is an element of the first kind.} In this case, since $z$ and $w$ are limits of $z_n$ and $w_n$, they are on the boundary of the components unless they belong to the same component $ G(z)=G(w)$ when the components are constant in $n$. In the latter situation there is nothing to prove. In the former situation, the argument is the same as in Case 2. 

    {\bf Case 4. Both $z$ and $w$ are on $S^1$.} If $z_n$ or $w_n$ are outside $\mathbb{D}\cup S^1$ then $G(z_n)=G(w_n)=\{ z_n\} =\{ w_n\}$ and $z=w$. If say $z$ is on a component $G(z)$ which is not a point, then this is covered by Cases 2 and 3. The same is true for $w$. Therefore, the sequence $G(z_n)$ must accumulate to $z$. Since $w_n\in G(z_n)$, it follows that $w_n$ accumulates to $z$. Hence $z=w$ and thus $G(z)=G(w)$.

    {\bf Case 5. At least one of $z$ or $w$ is in the complement of $\mathbb{D}\cup S^1$.} 
    Then $z=w$ because $z_n=w_n$ for $n$ large enough. 
\end{proof}

\begin{figure}[htb]
    \centering
    \includegraphics[width=6 cm]{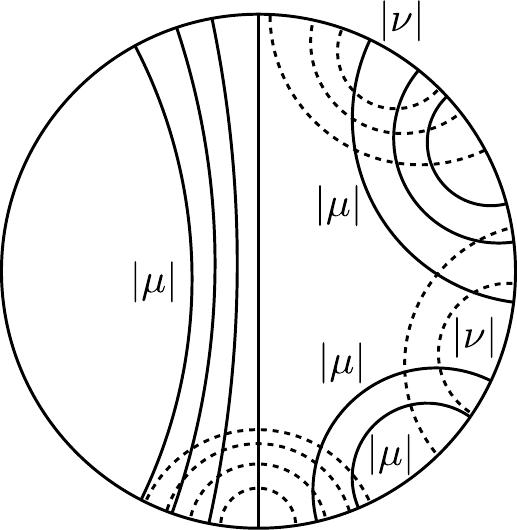}
    \caption{First kind elements of partition; $|\mu |$ full geodesics, $|\nu |$ dotted geodesics.}
    \label{fig:partition1}
\end{figure}

\begin{figure}[htb]
    \centering
    \includegraphics[width=6 cm]{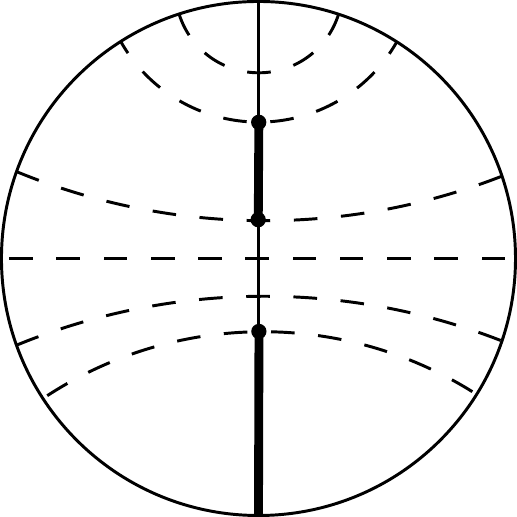}
    \caption{Second kind elements are in bold.}
    \label{fig:partition2}
\end{figure}

\begin{figure}[htb]
    \centering
    \includegraphics[width=6 cm]{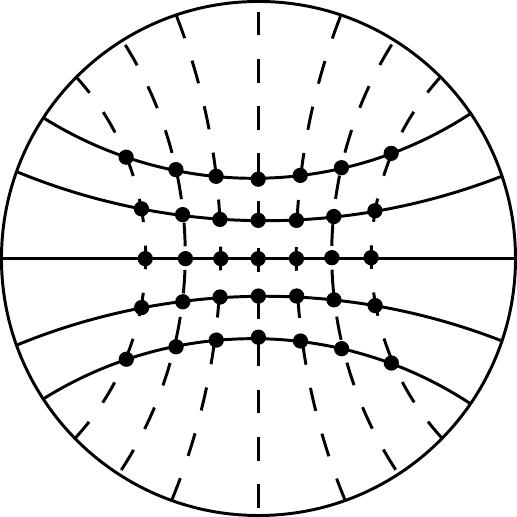}
    \caption{Third kind elements are bold points.}
    \label{fig:partition3}
\end{figure}

\begin{proof}[Proof of Theorem \ref{thm:realization_sufficient}]
Let $\sim$ be the equivalence relation on $\widehat{\mathbb{C}}$ whose equivalence classes are given by the above partition by non-separating continua. By Lemma  \ref{lem:UpperC} it is upper semi-continuous. Therefore, by Moore's theorem \cite{moore}, the quotient $\widehat{\mathbb{C}}/\sim$ is homeomorphic to $\widehat{\mathbb{C}}$.
By Proposition \ref{prop:component_closures}, no equivalence class contains two points of $S^1$, hence the restricted quotient map $S^1\to \widehat{\mathbb{C}}/\sim$ is injective. Since $S^1$ is compact and the range $\widehat{\mathbb{C}}/\sim$ is Hausdorff, the quotient map is a homeomorphism of $S^1$ onto its image (thus, the image is a Jordan curve). Since the partition in the complement of $\mathbb{D}\cup S^1$ consists of singletons and has a topological circle boundary, it follows that $(\mathbb{D}\cup S^1)/\sim$ is homeomorphic to $\mathbb{D}\cup S^1$.

We explain how $\mu$ and $\nu$ determine transverse measured foliations $\mu /\sim $ and $\nu / \sim$ on $(\mathbb{D}/\sim )\setminus (S^1 /\sim )$. 
Each geodesic in $|\mu |$ and in $|\nu |$ is collapsed to an arc in $(\mathbb{D}/\sim )\setminus (S^1 /\sim )$ with two endpoints in $S^1/\sim$ (because each geodesic in $|\mu |$ is intersected by a geodesic in $|\nu |$ so it does not collapse to a point). The endpoints of the quotient of any geodesic on $S^1/\sim$ are distinct because each geodesic is broken into more than one component of the decomposition $\mathcal{G}_{\mu ,\nu}$. When collapsing a compact component of the complement of $|\mu |\cup |\nu |$ in $\mathbb{D}$, we join several boundary geodesics at a single point, which creates a singularity. If the compact component has $N$ boundary components in $|\mu |$, and necessarily $N$ boundary components in $|\nu |$, then the collapsing forms an $N$-pronged topological singularity for the foliation obtained by quotienting $|\mu |$, and also an $N$-pronged topological singularity for the quotient of $|\nu |$ (see Figure \ref{fig:collapse_compact}). The foliations do not have any singularities at the elements of the third kind. By its construction, the quotient map $S^1\to S^1/\sim$ sends the endpoints of geodesics in $|\mu |$ and $|\nu |$ onto the endpoints of the corresponding geodesics in $|\mu /\sim |$ and $|\nu /\sim |$.

\begin{figure}[htb]
    \centering
    \includegraphics[width=9 cm]{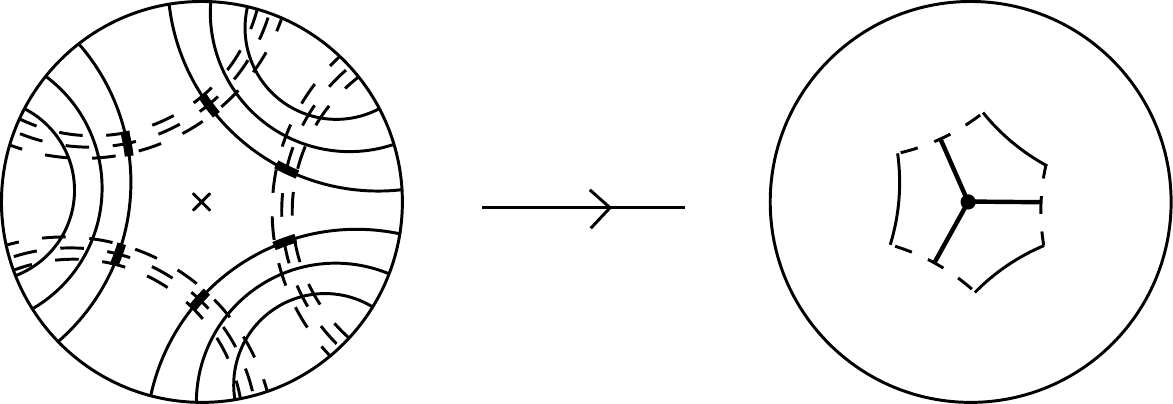}
    \caption{Collapsing compact complementary regions.}
    \label{fig:collapse_compact}
\end{figure}

Notice that 
$$
i(\mu ,\nu) =i(\mu /\sim ,\nu /\sim ) 
$$
by the definition. By the realizability of integrable partial foliations, i.e., \cite[Theorem 1.6]{Saric24} and \cite[Theorem 1.1]{Saric25}, there exist finite-area holomorphic quadratic differentials $q_{\mu}$ and $q_{\nu}$ on $\mathbb{D}$ whose horizontal foliations realize $\mu$ and $\nu$. Minsky's inequality (see \cite{SaricShima}) gives
$$
i(\mu ,\nu)^2\leq \int_{\mathbb{D}}|q_{\mu}|\int_{\mathbb{D}}|q_{\nu}|<\infty .
$$
Therefore
$$
i(\mu /\sim ,\nu /\sim ) <\infty .
$$

We partition $(\mathbb{D}/\sim )\setminus (S^1 /\sim )$ into flow rectangles for the filling foliations  $\mu /\sim$ and $\nu /\sim$. The rectangles are glued along their boundaries, and $N$ rectangles share a vertex at an $N$-pronged singularity. This picture is topological, with the lengths of the sides of the rectangles measured by the transverse measure of the foliation that crosses them. Then the intersection number $i(\mu /\sim ,\nu /\sim )$ is the sum of the areas of the rectangles.

We form a finite-area holomorphic quadratic differential from this data. Namely, we replace each rectangle of the above partition with a Euclidean rectangle of the same size, and we glue rectangles according to the gluing given by the topological picture. The gluings respect the transverse measures; therefore, they are Euclidean translations. The obtained Riemann surface $X$ has a natural holomorphic quadratic differential $q$ which is given by $dz^2$ on each rectangle. Since the sizes of the Euclidean rectangles are chosen to be equal to the sizes of the topological rectangles, it follows that $q$ is of finite-area on $X$. Note that $(\mathbb{D}/\sim )\setminus (S^1/\sim )$ is homeomorphic to $X$ by the construction, thus $X$ is homeomorphic to $\mathbb{D}$. The Riemann surface $X$ is biholomorphic to either the unit disk or to the complex plane. 
Since $q$ is a non-trivial, finite-area holomorphic quadratic differential on $X$, it follows that $X$ is biholomorphic to $\mathbb{D}$.

We now construct a homeomorphism from $S^1/\sim$ to $\partial_\infty X$ that maps the endpoints of leaves of the foliations $\mu /\sim$ and $\nu /\sim$ to the endpoints of leaves of the horizontal and the vertical foliations of $q$, respectively.
Note that there is a one-to-one correspondence between the above families of leaves. This is implied by the choice of the correspondence of the rectangles on $(\mathbb{D}/\sim )\setminus (S^1/\sim )$ and $X$, where the leaves are determined by the distances between their intersections with sides to the vertices of the rectangles.

 Each horizontal leaf of the differential $q$ limits at two distinct endpoints on the ideal boundary $\partial_{\infty}X$ (see \cite{Strebel}). Thus, the above correspondence of leaves gives a map from the pairs of endpoints of $\mu /\sim$ to the pairs of endpoints of the horizontal leaves of $q$. There is then a unique map
 from endpoints of $\mu /\sim$ to endpoints of the horizontal leaves that respects the separation properties for the triple of points on $S^1$ and $S^1/\sim$ for the natural orientations as boundaries of $\mathbb{D}$ and $(\mathbb{D}/\sim )\setminus (S^1/\sim )$; this map is necessarily monotone for the given orientations. The closure of the set of endpoints of $\mu /\sim$ may have open intervals in its complement in $S^1/\sim$ which are images of open intervals of $S^1$ in the complement of the endpoints of $|\mu |$ (and also correspond to the endpoints of the horizontal leaves of $q$). 
 By Proposition \ref{prop:endpoints_S^1}, the union of the sets of endpoints of $|\mu |$ and $|\nu |$ is dense in $S^1$, and hence the same is true for $\mu /\sim$ and $\nu /\sim$ in $S^1/\sim$. Thus, each leaf of $\nu /\sim$ with an endpoint in these complementary intervals is part of a rectangle with one side on this interval. Then we can consider the corresponding rectangle in $X$. We extend the map in the intervals by assigning the endpoints of the leaves of $\nu /\sim$ in the intervals to the endpoints of the vertical leaves of $q$ in $\partial_{\infty}X$. Therefore, the map extends to a dense subset of $S^1/\sim$, it is monotone, and the image is dense in $\partial_{\infty}X$ (density follows because $q$ is of finite-area in $X$ and hence we can apply Proposition \ref{prop:endpoints_S^1}). By density, we can extend it to a map $h_0:S^1/\sim\to\partial_{\infty}X$, which must be a homeomorphism. 
 
 By design, $h_0$ maps endpoints of trajectories of $\mu /\sim$ onto endpoints of horizontal trajectories of $q$, and similar for the endpoints of $\nu/\sim$ that lie in a complementary interval for $\mu/\sim$. The remaining trajectories of $\nu /\sim$ are determined by which trajectories of $\mu /\sim$ they cross, and these intersections are preserved under the homeomorphism $h_0$ restricted to the set of endpoints of $\mu /\sim$. The homeomorphism $h_0$ therefore has the desired properties.
 
 By pre-composing $h_0$ with the boundary homeomorphism of the homeomorphism from $\mathbb{D}\cup S^1$ onto $(\mathbb{D}\cup S^1)/\sim$ (from Moore's theorem) and post-composing with the boundary homeomorphism of a uniformization $X\to\mathbb{D}$, we obtain the homeomorphism and the finite-area holomorphic quadratic differential with the desired properties.
\end{proof}

\subsection{Realization for $X=\mathbb{D}/\Gamma$}\label{subsec: realization, general}

In this section, we prove sufficiency for general Riemann surfaces.

\begin{thm}
\label{thm:realization_first_kind}
Let $X=\mathbb{D}/\Gamma$ be a Riemann surface.
Let $\mu$ and $\nu$ be a pair of filling measured laminations on $X$ that are realizable by integrable (partial) measured foliations.

Then there exists a homeomorphism $f:X\to Y=\mathbb{D}/\Gamma_1$, whose lift to the universal covers extends to a homeomorphism of the ideal boundaries, and a finite-area holomorphic quadratic differential $q$ on $Y$ such that the push-forward measured laminations $f_*(\mu )$ and $f_*(\nu )$ are homotopic to $\mathcal{F}_h(q)$ and $\mathcal{F}_v(q)$.
\end{thm}
The boundary homeomorphism $S^1\to S^1$ takes endpoints of leaves of the lifted laminations $\Tilde{\mu}$ and $\Tilde{\nu}$ to endpoints of horizontal and vertical leaves respectively for the lifted differential $\Tilde{q}$. Since $\Tilde{\mu}$ and $\Tilde{\nu}$ are $\Gamma$-invariant and $\Tilde{q}$ is $\Gamma_1$-invariant, the boundary homeomorphism necessarily conjugates $\Gamma$ to $\Gamma_1$. It follows that $\Gamma$ is of the first (resp. second) kind if and only if $\Gamma_1$ is, and $f$ takes funnels in $X$ to funnels in $Y$ and geodesic half-planes in $X$ to geodesic half-planes in $Y$, if any.

Before beginning the proof, we introduce the blown-up disk.
Lift $\mu$ and $\nu$ to $\Gamma$-invariant laminations $\Tilde{\mu}$ and $\Tilde{\nu}$ on $\mathbb{D}$. Given any $\Gamma$-orbit of atomic geodesics of $|\Tilde{\mu} |$ (corresponding to a simple closed curve in the support of $\mu$), choose a representative $g$ and cut $\overline{\mathbb{D}}$ along $g$ (including the endpoints of $g$), and denote the resulting boundary components of the new disconnected surface by $g_-$ and $g_+$. Glue in a strip $$C_g = g\times \Big [-\frac{\mu(\{g\})}{2},\frac{\mu(\{g\})}{2}\Big ],$$ attaching the boundary components by identifying $g_-$ with $g\times \{-\frac{\mu(\{g\})}{2}\}$ and $g_+$ with $g\times \{\frac{\mu(\{g\})}{2}\}$. The resulting object is homeomorphic to a closed disk; on the boundary circle, it has two arcs of length $\mu(\{g\})$ corresponding to the endpoints of $g$. We foliate $C_g$ by the lines $g\times \{y\}$, $y\in [-\frac{\mu(\{g\})}{2},\frac{\mu(\{g\})}{2}]$ and assign the vertical transverse measure according to length in $[-\frac{\mu(\{g\})}{2},\frac{\mu(\{g\})}{2} ]$. The leaves of $\Tilde{\nu }$ are disconnected by $C_g$. So that $\Tilde{\nu}$ still gives a lamination on this new disk, we extend geodesics in $|\Tilde{\nu}|$ that intersect $g$ by attaching the relevant vertical arcs in $C_g$. We carry out the same procedure over the whole $\Gamma$-orbit, choosing the labelings so that $(\gamma g)^{\pm}=\gamma (g^{\pm}).$ We then carry out the procedure over every $\Gamma$-orbit, and then repeat everything for $\Tilde{\nu}$, and we obtain the blown-up closed disk $\overline{\mathbb{D}}^b=\mathbb{D}^b\cup (S^1)^b$ with interior $\mathbb{D}^b$ and circular boundary $(S^1)^b$. Note that no geodesic of $|\tilde{\nu}|$ shares an endpoint with the atomic geodesic of $\tilde{\mu}$ since this would imply that $\nu$ and $\mu$ share a common atomic geodesic which is impossible by the filling property. Therefore, there is no ambiguity in the definition of $(S^1)^b$ and the extension of the laminations $\tilde{\mu}$ and $\tilde{\nu}$ to $\overline{\mathbb{D}}^b=\mathbb{D}^b\cup (S^1)^b$.

\begin{proof}[Proof of Theorem \ref{thm:realization_first_kind}]
    By construction, $\mathbb{D}^b$ comes with two (not necessarily geodesic) measured laminations ($\Tilde{\mu}$ and $\Tilde{\nu}$ in the complement of the strips, and the straight lines in the strips). By abuse of notation, we keep the same notation $\Tilde{\mu}$ and $\Tilde{\nu}$. A simple geodesic that enters the standard collar must cross it. Proposition \ref{prop:component_closures} holds essentially the same for these laminations: the complementary components are contained in the hyperbolic components of $\mathbb{D}^b$ and are convex, and the compact components have finitely many sides which are arcs alternating between $|\Tilde{\mu} |$ and $ |\Tilde{\nu} |$, while the non-compact components accumulate to exactly one point of $S^1$. The action of $\Gamma$ extends to $\overline{\mathbb{D}}^b$: since the set of atoms is $\Gamma$-invariant, the action of $\Gamma$ restricts to their complement, and hence can be transported to the complement of the strips in $\overline{\mathbb{D}}^b$. On the strips, we define, for $\gamma \in \Gamma$, $\gamma\cdot (g,x)=(\gamma g, x)$, and note that this is well-defined since $\Tilde{\mu}$ is $\Gamma$-invariant.

Choose any homeomorphism from $S^1\to S^1$ and use it to glue $\hat{\mathbb{C}}\backslash\mathbb{D}$ and $\overline{\mathbb{D}}^b$ along their boundaries. We abuse notation and just write $\hat{\C}$ for the glued surface.
Similar to the case where $\Gamma$ is trivial, we define a partition of $\hat{\mathbb{C}}$ into non-separating continua using $\Tilde{\mu}$ and $\Tilde{\nu}$.
 The first kind of subsets are the closures of the components of $(\mathbb{D}\cup S^1)\setminus (|\Tilde{\mu} |\cup |\Tilde{\nu} |)$, if any.
 The second kind of subsets are the closures of the components of $|\Tilde{\mu} |\setminus |\Tilde{\nu} |$ and  $|\Tilde{\nu} |\setminus |\Tilde{\mu} |$, excluding those components that lie on the boundary of a component of the first kind.  These second kind subsets are compact subarcs and allowed to have an endpoint on $(S^1)^b$. As before, the third kind of subsets are the point sets of the complements of the union of the first two kinds elements.

 For this partition, the proof of Lemma \ref{lem:UpperC} goes through almost word-for-word. By Moore's theorem, defining $\sim$ by $x\sim y$ if they lie in the same component, the corresponding quotient $\Pi:\hat{\mathbb{C}}\to \hat{\mathbb{C}}/\sim$ produces a space homeomorphic to $\hat{\mathbb{C}}$. As in Theorem \ref{thm:realization_sufficient}, Proposition \ref{prop:endpoints_S^1} shows that the restricted map $\Pi: (S^1)^b\to (S^1)^b/\sim$ is a homeomorphism, the only difference here is that we have to check injectivity on the additional collars $C_g$, but this is obvious since each point on an arc corresponding to $C_g$ is the endpoint of a unique leaf. Since the complement of $\overline{\mathbb{D}}^b$ consists of singletons for the partition, and the boundary is a circle, the quotient of $\overline{\mathbb{D}}^b$ is homeomorphic to $\mathbb{D}\cup S^1$.
 
An additional feature now that $X$ is not a disk is that the partition is $\Gamma$-invariant. Hence, the action of $\Gamma$ on $\overline{\mathbb{D}}^b$ descends to an action on $\overline{\mathbb{D}}^b/\sim$. We argue that the action on $Q:=(\mathbb{D}^b/\sim)\backslash ((S^1)^b/\sim)$ is properly discontinuous and free (note $\mathbb{D}^b/\sim$ may have points on $(S^1)^b/\sim$, so we have to cut it out). To see this, first note that if $E\subset Q$ is compact, then the fiber $\Pi^{-1}(E)\subset \mathbb{D}^b$ (a priori a closed subset) is compact, since it cannot touch the boundary circle (by the upper semicontinuity of the partition). Proper discontinuity then follows immediately from proper discontinuity of the action of $\Gamma$ on $\mathbb{D}^b$ (which follows from proper discontinuity on $\mathbb{D}$). For freeness, if some $\gamma\in \Gamma$ fixes a point $x\in Q$, then for any $x_0\in \Pi^{-1}(x)$, $\{\gamma^n x_0\}_{n=1}^\infty$ belongs to a compact set, and this contradicts the proper discontinuity of the action of $\Gamma$ on $\mathbb{D}^b$.

As in the case of the disk, $\Tilde{\mu}$ and $\Tilde{\nu}$ determine transverse measured foliations $\Tilde{\mu}/\sim$ and $\Tilde{\nu}/\sim$ on $Q$, which are now necessarily invariant under the action of $\Gamma$. We divide $\hat{X}:=Q/\Gamma$ into flow boxes for $\mu /\sim$ and $\nu /\sim$ that are compact and simply connected. The lift of these flow boxes to $Q$ are flow boxes for $\tilde{\mu}/\sim$ and $\tilde{\nu}/\sim$ that are compact and simply connected, and they are invariant under the action of $\Gamma$. We form a new Riemann surface $\hat{Y}$ by taking Euclidean rectangles of the side lengths equal to the transverse measures of the rectangular flow boxes of the foliations $\Tilde{\mu}/\sim$ and $\Tilde{\nu}/\sim$. We choose the rectangles such that the horizontal Euclidean arcs correspond to the leaves of the foliation $\Tilde{\mu}/\sim$ and the vertical Euclidean arcs correspond to the leaves of $\Tilde{\nu}/\sim$. Then we glue the Euclidean rectangles according to the gluing of the rectangular flow boxes on $Q$ to obtain the Riemann surface $\hat{Y}$. The Riemann surface $\hat{Y}$ supports a holomorphic quadratic differential $q_{\hat{Y}}$ given by $dz^2$ in the Euclidean charts $z$ of the Euclidean rectangles that form $\hat{Y}$. The holomorphic quadratic differential $q_{\hat{Y}}$ is well-defined on the whole surface $\hat{Y}$ since the gluings of the Euclidean rectangles are by translations, and, at each vertex, only finitely many rectangles meet. The gluing of rectangles is easily seen to be $\Gamma$-invariant, hence the action of $\Gamma$ becomes a holomorphic action on $\hat{Y}$ that preserves $q_{\hat{Y}}$. The quotient is thus a Riemann surface $Y$ homeomorphic to $\hat{X}$, and $q_{\hat{Y}}$ descends to a holomorphic quadratic differential $q_Y$.

 When the foliations $\Tilde{\mu}/\sim$ and $\Tilde{\nu}/\sim$ descend to $\hat{X}$, their intersection number agrees with $i(\mu,\nu).$ By the integrability assumptions on $(\mu,\nu)$ and Minsky's inequality, the intersection number $i(\mu ,\nu )$ is finite. Since the intersection number is invariant under homeomorphisms, $\int_Y|q_Y |=i(\mu ,\nu )<\infty$.

To complete the proof, we are left to produce a homeomorphism $f:X\to Y$ whose universal covering lift $\Tilde{f}:\Tilde{X}\to \Tilde{Y}=\hat{Y}$ extends to a homeomorphism of the ideal boundaries. We work upstairs on the universal covers, first passing through the intermediate space $\mathbb{D}^b$. There is a one-to-one correspondence between leaves of $(\Tilde{\mu}/\sim,\Tilde{\nu}/\sim)$ and the horizontal and vertical foliations of $q_{\hat{Y}}$. Also note that by integrability, each ray of a trajectory of the lift $q_{\hat{Y}}$ accumulates to a unique point on $\partial_\infty \Tilde{Y}$. From here, following the analogous proof in the case of the disk, almost verbatim, we obtain a map $(S^1)^b\to \partial_\infty \Tilde{Y}$, which takes endpoints of leaves of $(\Tilde{\mu}/\sim,\Tilde{\nu}/\sim)$ to endpoints for $q_{\hat{Y}}$ and is monotone with respect to chosen orientations. The difference now compared to the disk is that the boundary map is equivariant and not quite a homeomorphism. Indeed, it collapses the boundaries of strips corresponding to atoms to points (the endpoints of the corresponding atomic geodesic for the laminations of $q_{\hat{Y}}$) and is otherwise injective and a local homeomorphism. But we recall that if we remove the boundaries of the strips from $(S^1)^b$ and glue the endpoints, we get the original circle $S^1=\partial_\infty X$. Thus, we have an equivariant homeomorphism from $\partial_\infty X\to \partial_\infty Y$. To obtain the map $f$, we choose any equivariant homeomorphic extension of the boundary map (say, the Douady-Earle extension) and descend to a homeomorphism from $X$ to $Y$.
\end{proof}

\begin{remark}\label{rem: equivalence relation and trees}
In \S \ref{sec: trees}, we define pseudo-distances $\Tilde{d}_\mu$ and $\Tilde{d}_\nu$ on $\mathbb{D}$, which induce a map to a product of $\R$-trees. It should be satisfying to observe that two points $x,y\in \mathbb{D}$ lie in the same component of the partition we defined if and only if $(\Tilde{d}_\mu+\Tilde{d}_\nu)(x,y)=0$, or equivalently if and only if they project to the same point in the product of $\R$-trees. It follows that the quotient $\hat{X}$ from the proof of Theorem \ref{thm:realization_first_kind} is the quotient of the image in the product of $\R$-trees. We will not need to use this observation explicitly, so we omit the proof. 
\end{remark}

\section{Quasiconformal realizations}\label{sec: qc realizations}

Here we address the question of when a holomorphic realization is quasiconformal, proving Theorems \ref{thm: non qc}, \ref{thm: sufficient, first kind}, and \ref{thm: sufficient, disk}.
In \S \ref{sec: qc disk}, we focus on elementary Fuchsian groups, proving Theorem \ref{thm: non qc} part (1), and then Theorem \ref{thm: sufficient, disk}. In \S \ref{sec: qc, infinite} we consider infinitely generated Fuchsian groups of the first kind: we prove Theorem \ref{thm: non qc} part (2) and Theorem \ref{thm: sufficient, first kind}.

\subsection{Quasiconformal vs. homeomorphic realizations in the unit disk $\mathbb{D}$}\label{sec: qc disk}
\subsubsection*{Non-quasiconformal realizations}
For Theorem \ref{thm: non qc} part (1), we assume the uniqueness portion Theorem \ref{thm: main}, whose proof is given in \S \ref{sec: trees}.
\begin{proof}[Proof of Theorem \ref{thm: non qc} part (1)]
We construct an integrable quadratic differential $q$ on $\mathbb{D}$ with associated foliations $\mathcal{F}_h(q)$ and $\mathcal{F}_v(q)$ and a homeomorphism $f: S^1\to S^1$ that is not quasisymmetric and such that the pullback partial measured foliations $f^*\mathcal{F}_h(q)$ and $f^*\mathcal{F}_h(q)$ are proper and integrable. If $\mu$ and $\nu$ are the laminations that we get from tightening $f^*\mathcal{F}_h(q)$ and $f^*\mathcal{F}_v(q)$ respectively, then $\mu$ and $\nu$ clearly fill, and $f$ provides a homeomorphic realization of the laminations $\mu$ and $\nu$ that is not homotopic relative to the boundary to a quasiconformal map.  By uniqueness of the realization, this example shows that we cannot realize the pair $\mu$ and $\nu$ under a quasiconformal deformation of $\mathbb{D}$. The other types of Riemann surfaces under consideration here are annuli (either of finite or infinite modulus). To treat these surfaces, we can take our example on the disk, and simply cut away a subdisk in the case of finite modulus or cut away a point in the case of infinite modulus.

Our construction really proceeds on $R=[-1,1]\times [0,1]$, which is conformal to $\mathbb{D}$. Note that the conformal map extends to a piecewise smooth map from $\partial R$ to $S^1$ with non-smoothness at the vertices of $R$. Let $\mathcal{F}$ and $\mathcal{G}$ be the horizontal and vertical foliations respectively of $dz^2$ on $R$ (the pushforward via the Riemann map is the $q$ from above).  We define $f:R \to R$ by
\[
f(x,y) =
\begin{cases}
\big({x}^2,\, y\big),
& 0 \le x \le 1,\; 0 \le y \le 1, \\[6pt]
(x,y),
& -1 \le x < 0,\; 0 \le y \le 1.
\end{cases}
\]
Note that $f|_{\partial R}$ is not quasisymmetric around $(0,0)$ because the right-hand derivative is $0$ and the left-hand derivative is $1$.

We let $f^*(\mathcal{F})=\mathcal{F}^*$ and $f^*(\mathcal{G})=\mathcal{G}^*$ be the pull-back  partial measured foliations on $R$. Integrability of $\mathcal{G}^*$ is immediate. For integrability of $\mathcal{F}^*$, note that, as a partial measured foliation, $\mathcal{F}$ is globally encoded by the level sets of the map $v(x,y)=x$. Since the pull-back $\mathcal{F}^*$ squares the transverse measure on $0\leq x\leq 1$, the foliation $\mathcal{F}^*$ is encoded by the function 
\[
v^*(x,y) =
\begin{cases}
x^2,
& 0 \le x \le 1,\; 0 \le y \le 1, \\[6pt]
x,
& -1 \le x < 0,\; 0 \le y \le 1.
\end{cases}
\]
Then we have
$$
D(\mathcal{F}^*)=\int_0^1\int_{-1}^{1}|\nabla v^*(x,y)|^2dxdy=\frac{7}{3}<\infty .
$$
By the discussion above, $\mathcal{F}^*$ and $\mathcal{G}^*$ give a non-quasiconformally realizable example.
\end{proof}

\subsubsection*{Quasiconformal realizations}\label{subsubsec: qc realizations}
We now set up Theorem \ref{thm: sufficient, disk} about dual regular polygonal laminations $\mu$ and $\nu$.

Let $I_1,\dots, I_n$ and $J_1,\dots, J_n$, $n\geq 3$, be partitions of $S^1$ into equal length intervals, with the endpoints of $J_k$ being the midpoints of $I_{k}$ and $I_{k+1}$ ($I_{n+1}:=I_1$). The geodesic connecting the endpoints of $I_j$ is denoted $g_j$, and the halfplane bounded by $g_j$ and $I_j$ is denoted $H_j$. Similarly, for $J_k$, we have geodesics $g_k'$ and half-planes $H_k'$. The half-planes $H_j$ and $H_k'$ are foliated by geodesics $g$ such that the distance between one endpoint of $g$ and one endpoint of $I_j$ (resp. $J_k$) equals the distance between the other endpoint of $g$ and the other endpoint of $I_j$ (resp. $J_k$). The geodesics foliating the half-planes $H_j$ are the support $|\mu|$ of a tentative lamination, and the geodesics foliating the $H_k$'s are the support $|\nu|$ of another tentative lamination.

\begin{figure}[htb]
    \centering
    \includegraphics[width=8 cm]{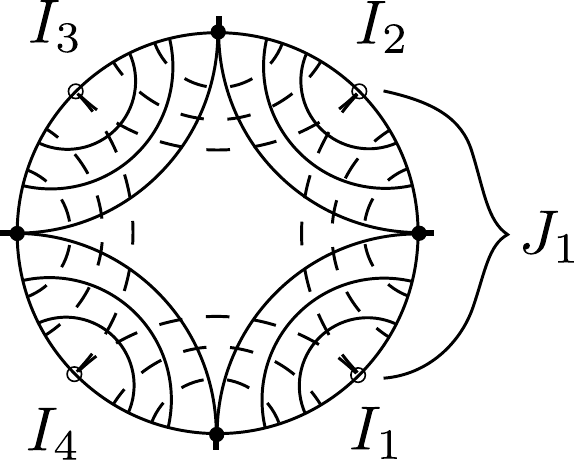}
    \caption{The full geodesics are in $|\mu |$ and the dotted geodesics are in $|\nu |$.}
    \label{fig:ex-qs}
\end{figure}

 Let $h_j:I_j\to [0,\infty)$ be square-integrable and symmetric about the midpoint of $I_j$. We define the \textit{endpoint measure} on $I_j$ by integrating $h_j(t)dt$ on subsets of $I_j$, where $dt$ is the Lebesgue measure on the circle. If $I\subset I_j$ and $G(I)$ is the set of geodesics with endpoints in $I$, then we set $$\mu(G(I))=\int_I h_j(t)dt,$$ which determines a lamination $\mu$. We define a lamination $\nu$ in the analogous fashion.
 \begin{defn}
     The lamination $\mu$ is called a regular polygonal lamination. The pair $(\mu,\nu)$ is referred to as a pair of dual regular polygonal laminations.
 \end{defn}

 \begin{prop}
    A regular polygonal lamination $\mu$ constructed from intervals $I_j$ and functions $h_j(t)$ as above is homotopic to a partial measured foliation with finite Dirichlet integral.
 \end{prop}
 \begin{proof}
      The construction is given by separately defining functions $u_j$ supported in each half-plane $H_j$. We conformally map $\mathbb{D}$ onto the upper-half plane $\mathbb{H}$ so that the central point of $I_j$ is mapped to zero and the support of $|\mu|$ intersected with $H_j$ consists of semi-circles in $\mathbb{H}$ with center $0$. We then use $z\mapsto \log z$ to transform this picture into an infinite strip with boundaries $\mathbb{R}$ and $\mathbb{R}+2\pi i$. The point $0$ is mapped to $-\infty$, the geodesics of the support in $H_j$ are mapped to vertical Euclidean segments between the boundary lines, and this image foliates the part of the strip $S_j$ with $x$-coordinates in $(-\infty ,x_j]$. We define $u_j^*:S_j\to \mathbb{R}$ by
$$
u_j^*(x+iy)=\int_{x}^{x_j}h_j(A(t))|A'(t)|dt,
$$
where $A$ is the inverse of the composition of the M\" obius map that maps $\mathbb{D}$ onto $\mathbb{H}$ and the $\log z$ map. It is immediate that $u_j^*$ depends only on $x$, hence the leaves are the vertical segments. The transverse measure $d[u_j^*(x)]=h_j(A(x))|A'(x)|dx$ gives the same measure as its corresponding set in $H_j$ with respect to $d\mu =h_j|dz|$. Finally, the Dirichlet integral of $u_j^*$ satisfies
$$
D(u_j^*)=\int_{S_j}|\nabla u_j^*(x+iy)|^2dxdy=2\pi \int_{-\infty}^{x_j}|h_j(A(t))|^2|A'(t)|^2dt.
$$
Making the substitution $s=A(t)$, the above integral is bounded above by the product of the $L^\infty$ norm of $A'$ (which is finite) and the $L^2$ norm of $h_j$, hence it is finite.
We can pull back $u_j^*$ to a function on $S_j$ which gives a partial measured foliation in $H_j$. Letting $j$ range over $1,\dots, n$ and summing up the corresponding functions, we obtain a foliation tightening to $\mu$ with finite Dirichlet integral. 
 \end{proof}

We are now set up to prove Theorem \ref{thm: sufficient, disk}. The uniform comparability assumption means that there exists an $L>0$ such that for all $j$ and $I\subset I_j$ as well as $k$ and $J\subset J_k,$
\begin{equation}\label{eq:comparison_measures}
    \frac{1}{L}\leq\frac{\mu (G(I))}{|I|}\leq L, \qquad
\frac{1}{L}\leq\frac{\nu  (G(J))}{|J|}\leq L.
\end{equation}
Since $\mu(G(I)),\nu(G(J))>0$ for all $I,J$ as above, it is immediate that each geodesic of $|\mu |$ is intersected by a transverse interval of geodesics in $|\nu |$, and each geodesic in $|\nu |$ is intersected by a transverse interval of geodesics in $|\mu |$. Further, an arbitrary geodesic in $\mathbb{D}$ is intersected by a transverse interval of geodesics in either $|\mu |$ or $|\nu |$, and most of the time by both. Hence, $\mu$ and $\nu$ are filling. 

\begin{proof}[Proof of Theorem \ref{thm: sufficient, disk}]
Any pair of measured laminations defined as above satisfies the assumptions of Theorem \ref{thm:realization_sufficient} and therefore there exists a homeomorphism $f$ of $S^1$ such that $f_*(\mu )$ and $f_*(\nu )$ are realized by the horizontal and vertical foliation of a finite-area holomorphic quadratic differential $q$ on $\mathbb{D}$. We construct the finite-area holomorphic quadratic differential using charts and study the corresponding homeomorphism. According to the construction, the middle domain not foliated by $|\mu |\cup |\nu |$ is collapsed to a point, and the arcs of the geodesics of $|\mu |$ between adjacent boundary geodesics $g_{j}'$ and $g_{j_1}'$ of $|\nu|$ are collapsed to points, identifying the rays of $g_j'$ and $g_{j_1}'$. Analogously for the arcs of $|\nu |$ connecting adjacent boundary geodesics of $|\mu|$. 

The quotient space has regions corresponding to the transverse intersections of the geodesics of $|\mu |$ and $|\nu |$. These regions are bounded by one ray of some $g_j$ and one ray of some $g_{j_1}'$ (see Figure \ref{fig:block}, rays in bold). The corresponding region 
$\Omega_{j,j_1}$ has an ideal boundary arc $I_j\cap J_{j_1}$. We first construct a Euclidean chart for the region $\Omega_{j,j_1}$, which serves as a building block for the realization. We draw a horizontal arc $a$ starting at the origin and lying on the positive real axis whose Euclidean length equals to the intersection of the ray on $g_j$ with the measured lamination $\nu$, and we draw a vertical arc $b$ starting at the origin and lying on the positive $y$-axis whose Euclidean length equals to the intersection of the ray on $g'_{j_1}$ with the measured lamination $\mu$. Each point $w$ of $b$ corresponds to a geodesic $g_z$ in $|\mu |$ through the point $z$ on the ray on $g_{j_1}'$ that is part of the foliation of $H_j$. It is given by the condition that the Euclidean distance from $w$ to the top $w_0$ of $b$ equals the transverse measure $\mu$ of the geodesics of $|\mu |$ between $g_j$ and $g_z$, where $g_z$ is the geodesic of $|\mu |$ through point $z$. From $w\in b$ we draw a horizontal arc in the first quadrant whose Euclidean length equals the transverse measure $\nu$ of the geodesics of $|\nu |$ that intersect the ray $g_z\cap \Omega_{j,j_1}$. We obtain a domain $E_{j,j_1}$ in $\mathbb{C}$ such that the horizontal (vertical) arcs in $E_{j,j_1}$ correspond to the rays $|\mu |\cap\Omega_{j,j_1}$ ($|\nu |\cap\Omega_{j,j_1}$) and the corresponding transverse measures are equal. Therefore, the construction equips $E_{j,j_1}$ with the quadratic differential $dw^2$. We glue all $E_{j,j_1}$'s in the cyclic order determined by the collapsing; the vertex corresponds to a zero of the quadratic differential of order $n-2$.

\begin{figure}[htb]
    \centering
    \includegraphics[width=9 cm]{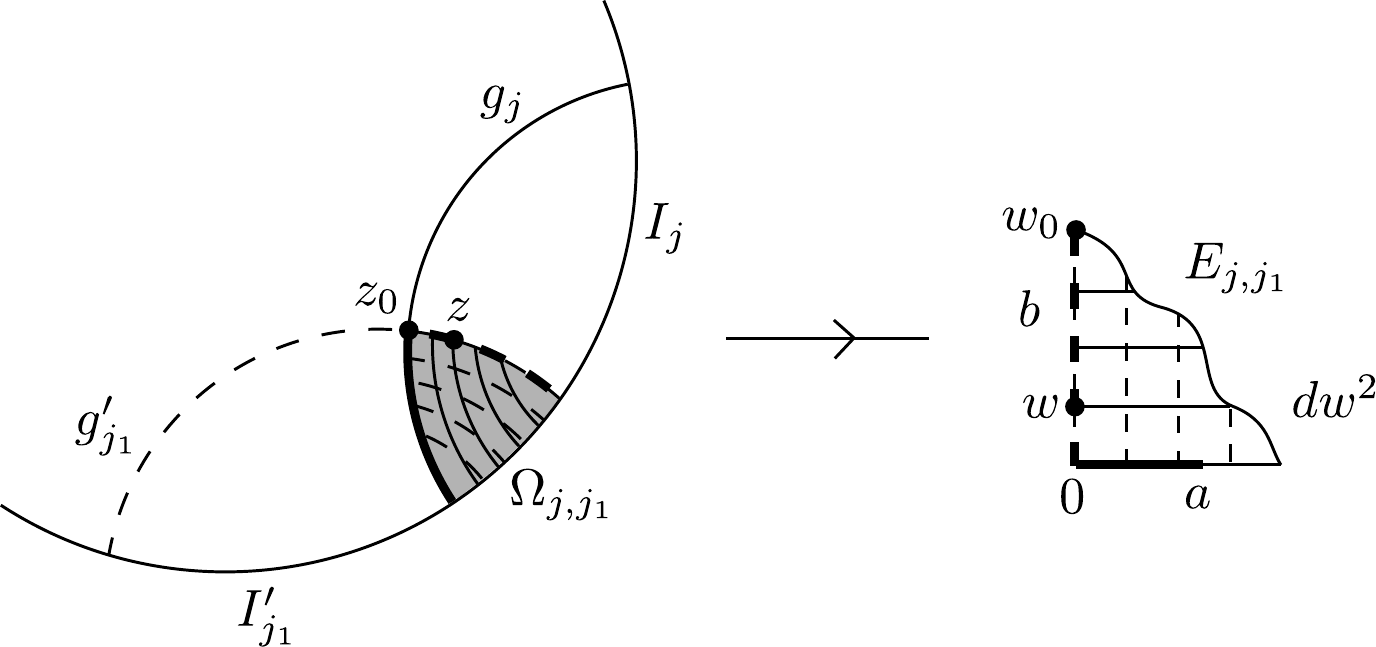}
    \caption{The building block for the realization. The boundary rays are in bold.}
    \label{fig:block}
\end{figure}

The curved part of the boundary of each $E_{j,j_1}$ is the graph of a decreasing function $f_1$ over the interval $a$ as well as the graph of a decreasing function $f_2$ over the interval $b$. This follows because the lengths of the vertical and horizontal arcs in $E_{j,j_1}$ are given by the transverse measures along the corresponding geodesic rays in $\Omega_{j,j_1}$, which are strictly decreasing when the rays approach the endpoints. The lengths along $a$ are given as the transverse measure of $|\nu |$ and the lengths along $b$ as the transverse measure of $|\mu |$. The condition (\ref{eq:comparison_measures}) implies that the function $f_2^{-1}\circ f_1$ defines a bi-Lipschitz homeomorphism of the segment $a$ onto the segment $b$. In fact, the bi-Lipschitz homeomorphism remains a bi-Lipschitz homeomorphism over adjacent building blocks.

Denote by $E$ the Riemann surface obtained by gluing the building block $E_{j,j_1}$. 
The correspondence between the geodesics $|\mu |$ and $|\nu |$ with the horizontal and vertical arcs in $E$ induces a homeomorphism of $S^1$ to the ideal boundary of $E$ (which consists of the curved boundary sides of $E_{j,j_1}$ glued along their endpoints). We show that this homeomorphism, postcomposed by the boundary map of a conformal map from $E$ to the unit disk $\mathbb{D}$, is quasisymmetric. To do so, it is enough (see \cite{GS}) to consider all triples of adjacent arcs of the same Euclidean lengths on $S^1$ (and the lengths bounded from above by some fixed constant $\delta >0$) and show that in the image $E$, the modulus of all curves connecting the image of the first interval to the image of the third interval is between two positive constants (independent of the choice of the intervals on $S^1$). 

We choose $\delta >0$ small enough such that any triple of adjacent intervals of Euclidean lengths at most $\delta$ stays either in a centered subinterval of $I_j$ of size $3/4$ of the size of $I_j$ or in a centered subinterval of $J_j$ of size $3/4$ of the size of $J_j$ Without loss of generality, we assume that the three adjacent intervals are in the centered subinterval of $I_j$. 

Under the boundary homeomorphism $S^1\to\partial E$, the image $m_1,m_2,m_3$ on $\partial E$ of the above triple of equal size intervals is bounded away from the endpoints of the image of $I_j$. Since we consider only finitely many intervals $I_j$, the distance between the images of the endpoints of the centered interval and the endpoints of $I_j$ is bounded uniformly from below for all $j$.  
 Therefore, the images of the three adjacent intervals on $\partial E$ are bounded away from the places where the curved sides meet the horizontal arc corresponding to the geodesic $g_j$, and their endpoints have real and imaginary parts proportional to each other by a uniform constant (since the functions relating boundary parametrizations are bi-Lipschitz). 
 
 We first establish that the modulus of the curves in $E$ connecting $m_1$ and $m_3$ is bounded above by a positive constant. It will be enough to consider a part of $E$ that lies above the horizontal arc corresponding to $g_j$. Let $|I|=h$ and note that the Euclidean sizes of $m_1,m_2,m_3$ are of the order $h$ by (\ref{eq:comparison_measures}). We consider a neighborhood $N(m_3)$ in $E$ of $m_3$ that consists of points at a distance to $m_3$ of the order $h$, and that does not contain points of $m_1$ (see Figure \ref{fig:upper-b}). Then any curve connecting $m_3$ to $m_1$ starts at $m_3$ and exits $N(m_3)$. We define a metric to be the Euclidean metric times $\frac{1}{h}$ in $N(m_3)$ and to be zero elsewhere. This is an allowable metric for the curve family (up to a constant multiple), and the area of $N(m_3)$ is a constant since the Euclidean area is of the order $h^2$ and the square of the metric has $\frac{1}{h^2}$ multiple. Thus there is a positive upper bound for the modulus of the family.
 
 \begin{figure}[htb]
    \centering
    \includegraphics[width=8 cm]{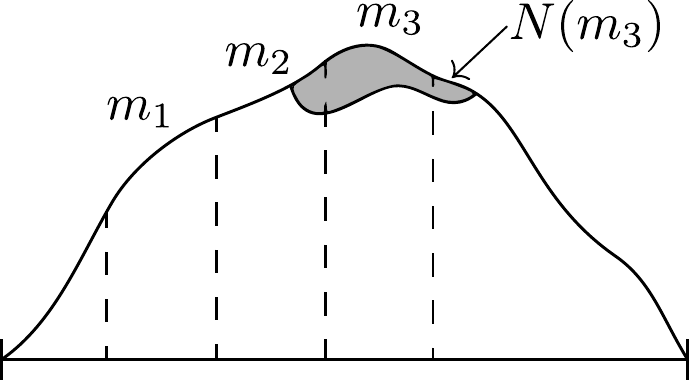}
    \caption{The bi-Lipschitz graph.}
    \label{fig:upper-b}
\end{figure}
 
 We next establish that the modulus of the curves in $E$ connecting $m_1$ and $m_3$ has a positive lower bound. Since we established that the boundary $\partial E$ is the graph of a bi-Lipschitz function and the arcs $m_1,m_2,m_3$ are a positive distance away from its endpoints, and the height of the graph is bounded below by a positive number, there exists a bi-Lipschitz map (with a constant bounded above independently of the positions of $m_1,m_2,m_3$) from the configuration below $m_1\cup m_2\cup m_3$ to a rectangle with length $3h$ and height bounded below with $m_2$ in the middle of the top side of length $h$. The modulus in the rectangle is positive since it is a single configuration. The modulus in $E$ is positive since we restrict to a subfamily of curves that stay in the preimage of the rectangle and the maps have uniformly bounded Lipschitz constants.
 
 Since we established lower and upper bounds on the modulus of the images under the homeomorphism of all symmetric triples of intervals, we conclude that the homeomorphism is a quasisymmetric map. Finally, we can extend this quasisymmetric map to a quasiconformal map of the disk.
 \end{proof}

\subsection{Quasiconformal vs. homeomorphic realizations for infinite surfaces}\label{sec: qc, infinite}
The types of laminations considered in Theorem \ref{thm: non qc} part (2) and in Theorem \ref{thm: sufficient, first kind} are the same. We first recall that any infinite topological surface always has two multicurves $\{\alpha_n\}_{n\in\mathbb{N}}$ and $\{\beta_k\}_{k\in\mathbb{N}}$ that are filling in the sense that the components of the complement of their union consist of finite-sided polygons with consecutive sides on the curve component of $\alpha$ and $\beta$ (and thus are filling in our sense). Moreover, each $\alpha_n$ intersects at most finitely many curves in 
$\{\beta_k\}_{k\in\mathbb{N}}$ and the total number of intersection points is bounded by $C>0$ independently of $n$. An analogous statement with the same bound holds for each $\beta_k$ (see Morales-Valdez \cite[Theorem 1.5]{MV}). Because of this bounded intersection property, we say that the pair of multicurves is \textit{locally finite}.

In this section we consider laminations of the form $$\mu =\sum_{n=1}^{\infty} m(\alpha_n)\delta_{\alpha_n}, \qquad \nu =\sum_{k=1}^{\infty} m(\beta_k)\delta_{\beta_k},$$ where $(m(\alpha_n))_{n=1}^\infty$ and $(m(\beta_k))_{k=1}^\infty$ are sequences of positive numbers, and $\delta_\alpha$ is the Dirac measure of a simple closed curve $\alpha$.

Let $X$ be an infinite Riemann surface with $\{\alpha_n\}_{n\in\mathbb{N}}$, $\{\beta_k\}_{k\in\mathbb{N}}$, $\mu$, and $\nu$ as above. We always use $R_{\alpha_n}$ to denote the standard hyperbolic collar around $\alpha_n$, and $R_{\beta_k}$ for the standard hyperbolic collar around $\beta_k$. The Collar Lemma \cite{buser} states that for different $n$, the $R_{\alpha_n}$'s are pairwise disjoint, and likewise for the $R_{\beta_k}$'s. We recall the integrability criterion contained in \cite[page 134, Theorem 22.1]{Strebel} (see also \cite{Saric24}).
\begin{thm}\label{thm: realization criteria modulus}
    Assume $X=\mathbb{D}/\Gamma$, with $\Gamma$ of the first kind. Then $\mu$ is realized by an integrable holomorphic quadratic differential on $X$ if 
\begin{equation}\label{eq:convergence}
    \sum_{n=1}^{\infty} \frac{m(\alpha_n)^2}{\mathrm{mod}(R_{\alpha_n})}<\infty
\end{equation}
\end{thm}
Of course, the analogous result holds for $\nu$. 
\subsubsection{Non-quasiconformal realizations}
Continue with $X$, $\mu$, and $\nu$ as above. We first demonstrate how quasiconformal realization can fail. Let $\textrm{Ext}_X(\beta)$ denote the extremal length of a curve $\beta$ on $X$.

\begin{thm}\label{thm: not quasiconformal}
        Let $X$ be any Riemann surface and assume that $\mu$ and $\nu$ as above are realized on $X$ by integrable partial measured foliations. If 
        \begin{equation}\label{eq: sup of big quantity}
            \sup_{k\in\mathbb{N}} \frac{m(\beta_k)\textrm{Ext}_X(\beta_k)}{i(\mu,\beta_k)}=\infty,
        \end{equation}
  then the holomorphic realization of $(\mu,\nu)$ is not quasiconformal. 
\end{thm}
\begin{proof}
Assume for contradiction's sake that $f:X\to Y$ is a $K$-quasiconformal map such that the measured laminations $f_*(\mu )$ and $f_*(\nu )$ are realized by the horizontal and vertical foliation of a finite-area holomorphic quadratic differential $q$ on $Y$.

    Fix $k\in\mathbb{N}$ and set $N(k)=\{ n\in\mathbb{N}: i(\alpha_n,\beta_k)>0\}$, which is a finite set because the filling multicurves are locally finite.  Then the vertical foliation of $q$ contains a maximal cylinder $\mathcal{C}(\beta_k)$ swept out by simple closed vertical leaves of $q$ that are homotopic to $f(\beta_k)$. Denote by $L_k$ the length of each closed vertical leaf in $\mathcal{C}(\beta_k)$. Since the horizontal foliation of $q$ is homotopic to $f_*(\mu )$, it follows that
\begin{equation}
\label{eq:length_circumference}
L_k=\sum_{n\in N(k)}i(\alpha_n,\beta_k)m(\alpha_n)=i(\mu,\beta_k).
\end{equation} 
The extremal length $\textrm{Ext}_Y(f(\beta_k))$ of $f(\beta_k)$ on $Y$ is less than or equal to the extremal length of the cylinder $\mathcal{C}(\beta_k)$. Using the natural parameter of $q$, we thus obtain
\begin{equation}
\label{eq:mod-bound}
\textrm{Ext}_Y(f(\beta_k))\leq \frac{L_k}{m(\beta_k)}.
\end{equation} 
Since $f$ is quasiconformal and by (\ref{eq:mod-bound}), $$
\mathrm{Ext}(\beta_k)\leq K\mathrm{Ext}(f(\beta_k))\leq K \frac{L_k}{m(\beta_k)}.
$$
Rearranging and using (\ref{eq:length_circumference}) gives $$\frac{m(\beta_k)\textrm{Ext}_X(\beta_k)}{i(\mu,\beta_k)}\leq K.$$ Since $k$ was arbitrary, this contradicts (\ref{eq: sup of big quantity}).
\end{proof}
Now we prove the second part of Theorem \ref{thm: non qc}. 
\begin{proof}[Proof of Theorem \ref{thm: non qc} part (2)]
    We work with $X$ of the first kind, and $\mu$ and $\nu$ constructed from multicurves $\{\alpha_n\}_{n=1}^\infty$ and $\{\beta_k\}_{k=1}^\infty$ as above. By Theorem \ref{thm: not quasiconformal}, it suffices to specify weights $(m(\alpha_n))_{n=1}^\infty$ and $(m(\beta_k))_{k=1}^\infty$ satisfying the integrability condition (\ref{eq:convergence}) and such that (\ref{eq: sup of big quantity}) holds. 

    As in the proof of Theorem \ref{thm: not quasiconformal}, for each $k\in \mathbb{N}$ we set $N(k)=\{ n\in\mathbb{N}: i(\alpha_n,\beta_k)>0\}$. We also define $M(n)=\{k\in \mathbb{N}: i(\alpha_n,\beta_k)>0\}$. We set $(\epsilon_n)_{n=1}^\infty$ to be any sequence of positive numbers satisfying
\begin{equation}
\label{eq:mu_int}
\sum_{n=1}^{\infty}\epsilon_n =: C <\infty,
\end{equation}
  and define
   $$m(\alpha_n) = \min\Big \{ \epsilon_n\sqrt{\mathrm{mod}(R_{\alpha_n})}, \epsilon_n\cdot \min_{m\in M(n)}\frac{\epsilon_m^2\textrm{Ext}_X(\beta_m)}{i(\alpha_n,\beta_m)\sqrt{\mathrm{mod}(R_{\beta_m})}}\Big \}$$ and 
   $$m(\beta_k) = \epsilon_k \sqrt{\mathrm{mod}(R_{\beta_k})}.$$
The condition (\ref{eq:mu_int}) implies that (\ref{eq:convergence}) holds for both $\mu$ and $\nu$ (for $\mu$, we use the first term in the ``min").

To see that (\ref{eq: sup of big quantity}) holds, we upper bound $i(\mu,\beta_k)$, using the second term in the ``min": 
\begin{align*}
    i(\mu,\beta_k)&=\sum_{n\in N(k)}i(\alpha_n,\beta_k)m(\alpha_n)\leq \sum_{n\in N(k)}i(\alpha_n,\beta_k)\epsilon_n \Big (\min_{m\in M(n)}\frac{\epsilon_m^2\textrm{Ext}_X(\beta_m)}{i(\alpha_n,\beta_m)\sqrt{\mathrm{mod}(R_{\beta_m})}}\Big ) \\
    &\leq \sum_{n\in N(k)} \epsilon_n \epsilon_k^2\frac{\textrm{Ext}_X(\beta_k)}{\sqrt{\mathrm{mod}(R_{\beta_k})}} \leq C\epsilon_k^2\frac{\textrm{Ext}_X(\beta_k)}{\sqrt{\mathrm{mod}(R_{\beta_k})}},
\end{align*}
where in the last inequality we used $(\ref{eq:mu_int})$. Now, inserting the above bound for $i(\mu,\beta_k)$, and then the formula for $m(\beta_k)$, we obtain
$$\frac{m(\beta_k)\textrm{Ext}_X(\beta_k)}{i(\mu,\beta_k)}\geq m(\beta_k) \epsilon_k^{-2}C^{-1}\sqrt{\mathrm{mod}(R_{\beta_k})}=C^{-1}\epsilon_k^{-1}.$$ Since $\epsilon_k\to 0$ as $k\to \infty$, (\ref{eq: sup of big quantity}) indeed holds.
\end{proof}

\subsubsection*{Quasiconformal realization}

We now prove Theorem \ref{thm: sufficient, first kind}. As in the theorem statement, we have a Riemann surface $X=\mathbb{D}/\Gamma$ with laminations $$\mu =\sum_{n=1}^{\infty} m(\alpha_n)\delta_{\alpha_n}, \qquad \nu =\sum_{k=1}^{\infty} m(\beta_k)\delta_{\beta_k}$$ coming from filling and locally finite multicurves constructed as in the beginning of this subsection, but now the $\alpha_n$'s are the cuffs of a pants decomposition and the hyperbolic lengths of the $\alpha_n$'s and $\beta_k$'s are uniformly bounded above and below.

\begin{proof}[Proof of Theorem \ref{thm: sufficient, first kind}]
By \cite{Maskit}, there exist $C_1,C_2>0$ such that for any simple closed geodesic $\alpha$ on a Riemann surface $X$ 
$$
C_1\ell_X(\alpha )\leq \mathrm{Ext}(\alpha )\leq C_2\ell_X(\alpha )e^{\frac{\ell_X(\alpha )}{2}},
$$
where $\ell_X(\alpha )$ is the hyperbolic length of $\alpha$ and $\mathrm{Ext}(\alpha )$ is the extremal length. Therefore, the uniform upper and lower bounds on the hyperbolic lengths of the $\alpha_n$'s and $\beta_k$'s are equivalent to giving upper and lower bounds on the moduli of these curves. Hence, square-summability of $(m(\alpha_n))_{n=1}^\infty$ and $(m(\beta_k))_{k=1}^\infty$ implies that (\ref{eq:convergence}) holds, so that $\mu$ and $\nu$ are realized by straightening integrable partial measured foliations on $X$.

Let $f:X\to Y$ be a homeomorphic realization associated with $\mu$ and $\nu$, obtained via Theorem \ref{thm: main}. To show that $f$ is homotopic to a quasiconformal map, we will use \cite{shiga}. For each $\alpha_n$ we choose one transversely intersecting $\beta_{k_n}$. The hyperbolic lengths of the $\alpha_n$'s together with the twists of the $\beta_{k_n}$'s relative to the $\alpha_{n}$'s allows us to represent $X$ by a Fenchel-Nielsen coordinate.
By the uniform controls on the hyperbolic lengths, the twist coordinates are uniformly bounded. By \cite{shiga}, if $\{ f(\alpha_n)\}_{n=1}^\infty$ is a bounded pants decomposition of $Y$ and $\{ f(\beta_{k_n})\}_{n=1}^\infty$ has bounded relative twists, then $f$ is homotopic to a quasiconformal map. We show that $\{ f(\alpha_n)\}_{n=1}^\infty$ has these properties; to do so, we have to step back into the construction of $f:X\to Y$ (Theorem \ref{thm:realization_first_kind}) and flesh out what happens for our specific laminations.

Recall that the first step in the proof of Theorem \ref{thm:realization_first_kind} is to lift to the universal cover and replace atomic geodesics by foliated rectangles, with the total mass of the rectangles chosen appropriately, obtaining the space $\mathbb{D}^b$. Let $\Pi: \mathbb{D}^b\to Q$ be the quotient map from the proof of Theorem \ref{thm:realization_first_kind}. The complementary regions of the lifted foliations in the lifts of pairs of pants are collapsed to points and the arcs of the foliations which do not intersect the other foliations except at their endpoints are also collapsed into points. The part of $\mathbb{D}^b$ which is not collapsed is exactly the lifts of the thickenings of the points where orthogonal arcs connect the cuffs $\{\alpha_n\}_{n=1}^{\infty}$. In the quotient $\hat{X}=Q/\Gamma$, for each such point we obtain a rectangle and the rest of the surface $\hat{X}$ is obtained by gluing of these rectangles along their sides. 
The rectangles are foliated in the horizontal direction with the foliation corresponding to $\mu$ and in the vertical direction by a foliation corresponding to $\nu$. The rectangles lie inside the collar neighborhoods of cuffs $\{\alpha_n\}_{n=1}^\infty$ and each collar neighborhood contains at most $D$ rectangles, where $D$ is some uniform constant depending only on the combinatorics of the multi-curves, since there are at most a uniformly bounded number of points of orthogonal arcs from inside one pair of pants \cite[Theorem 1.5 part (1)]{MV}. Finally, to obtain $Y$ and $q$, the flow boxes on $\hat{X}$ are replaced with Euclidean rectangles of the same size and glued according to the gluings that produced the surface $\hat{X}$. 

The statement of Theorem \ref{thm: sufficient, first kind} contains the assumption that there exists $M\geq 1$ such that
\begin{equation}
\label{eq:bdd_ratio_weights}
\frac{1}{M}\leq\frac{m(\alpha_n)}{m(\beta_k)}\leq M
\end{equation}
for all $n,k$ with $i(\alpha_n,\beta_k)\neq 0$. Let $f:X\to Y$ be a realizing homeomorphism as in Theorem \ref{thm: main}. Fixing an $n$ and following the construction above, the union of at most $D$ rectangles glued along vertical sides in a cyclic order forms an annular neighborhood of $f(\alpha_n )$ in $Y$. Likewise for every $\beta_k$. By condition (\ref{eq:bdd_ratio_weights}), there is $C=C(M)$, independent of $n,k$, such that
$$
\mathrm{Ext}_Y(f(\alpha_n))\leq C, \qquad
\mathrm{Ext}_Y(f(\beta_k))\leq C.
$$
These two inequalities together with the above inequality between the extremal and hyperbolic lengths imply that
$$
\ell_Y(f(\alpha_n)),\ell_Y(f(\beta_k))\leq C_1.
$$
Since each collar of $\alpha_n$ is intersected by at least one geodesic $\beta_k$ (by the filling property) and vice versa, it follows that there is a $C_2>0$ such that for all $n,k$,
$$
C_2\leq\ell_Y(f(\alpha_n)),\ell_Y(f(\beta_k)). 
$$
Thanks to the upper and lower bounds on $\ell_Y(f(\beta_k))$, every $f(\beta_k)$ must have uniformly bounded twists along $f(\alpha_n)$. We conclude that $\{ f(\alpha_n)\}_{n=1}^\infty$ is indeed a bounded pants decomposition of $Y$ with bounded twists. By the discussion above, we are done.
\end{proof}

\section{Uniqueness and energy}\label{sec: trees}
In this section we introduce dual $\R$-trees and use them to prove the uniqueness and energy minimizing property in Theorem \ref{thm: main}, while also preparing to address the asymptotic Plateau problem. These constructions are well known on a closed surface, but here we need to work in complete generality and hence revisit the foundations.

As discussed in \ref{subsubsec: overview}, it is possible to prove the main results without $\R$-trees. Before beginning with $\R$-trees in \S \ref{sec: dual space}, we sketch the alternative proofs.

\subsection{Sketches of alternative proofs}\label{sec: without R trees}

For uniqueness, just to keep things simpler, we focus on the case $X=\mathbb{D}$. Let $\mu$ and $\nu$ be two filling laminations on $\mathbb{D}$ homotopic to foliations with finite Dirichlet integral. For any holomorphic realization $(f,Y,q)$, after applying a biholomorphism we can always set $Y=\mathbb{D}$.

\begin{proof}[Sketch of uniqueness, on the disk]
Assume we have two holomorphic realizations $(f,\mathbb{D},q)$ and $(g,\mathbb{D},r)$ of $(\mu,\nu)$. Our goal is to prove that $g\circ f^{-1}$ is homotopic to a holomorphic map that pushes $q$ to $r$.

Every point $p\in \mathbb{D}$ is the intersection point of a $q$-horizontal (regular or generalized) trajectory $\ell_1$ and a $q$-vertical (regular or generalized) trajectory $\ell_2$. The trajectories $\ell_1$ and $\ell_2$ tighten to geodesics $g_1$ and $g_2$. The geodesics $g_1$ and $g_2$
 necessarily intersect (the details for this are contained in the proof of Lemma \ref{lem: image contains the core} below). Setting $m = g|_{S^1}\circ f|_{S^1}^{-1}$, applying $m$ on the endpoints of $g_1$ and $g_2$ gives two geodesics $m_*g_1$ and $m_* g_2$ that intersect and are tightenings of intersecting $r$-horizontal and $r$-vertical trajectories $\hat{\ell}_1$ and $\hat{\ell}_2$, which also intersect (as above, see Lemma \ref{lem: image contains the core}). By the Teichm{\"u}ller lemma, $\hat{\ell}_1$ and $\hat{\ell}_2$ intersect at a unique point, say $p(\hat{\ell}_1,\hat{\ell}_2)$. 

Let $\Gamma(q)$ be the union of all critical horizontal and vertical trajectories for $q$. Every $p\in \mathbb{D}\backslash\Gamma(q)$ is the intersection of a \textit{unique} horizontal and vertical trajectory. Therefore, on $\mathbb{D}\backslash\Gamma(q)$, the corresponding trajectories $\hat{\ell}_1$ and $\hat{\ell}_2$ are uniquely determined, and
we can define 
\begin{equation}\label{def: intertwining map}
    F:\mathbb{D}\backslash\Gamma(q)\to \mathbb{D}, \qquad F(p)=p(\hat{\ell}_1,\hat{\ell}_2).
\end{equation}
We will show that $F$ extends to a M{\"o}bius transformation that transports $q$ to $r$.

To see that $F$ extends to all of $\mathbb{D}$, one shows that if $p\in \Gamma(q)$ is the intersection of at least two pairs of horizontal and vertical trajectories, say $(\ell_1,\ell_2)$ and $(\ell_1',\ell_2')$, then the intersection points $p(\ell_1,\ell_2)$ and $p(\ell_1',\ell_2')$ agree. Indeed, let $p\in \Gamma(q)$. Then $p$ is the intersection of at least two pairs of horizontal and vertical trajectories, say $(\ell_1,\ell_2)$ and $(\ell_1',\ell_2')$. By the procedure above, they give pairs of intersecting horizontal and vertical trajectories $(\hat{\ell}_1,\hat{\ell}_2)$ and $(\hat{\ell}_1',\hat{\ell}_2')$ for $r$. Since $\ell_1$ and $\ell_1'$ are both horizontal trajectories passing through $p$, they can be joined with paths of arbitrarily small horizontal transverse measure. Likewise for $\ell_2$ and $\ell_2'$. Transporting via $m$, and using the description of $\mu$ as a current from \S \ref{sec: realizable laminations}, we deduce that $\hat{\ell}_1$ and $\hat{\ell}_1'$ have zero horizontal transverse distance, and $\hat{\ell}_2$ and $\hat{\ell}_2'$ have zero vertical transverse distance. It follows that the intersection points of $(\hat{\ell}_1,\hat{\ell}_2)$ and $(\hat{\ell}_1',\hat{\ell}_2')$ have zero distance in the singular flat metric induced by $r$, and hence they are the same point.

If $\Gamma(r)$ is the set of critical trajectories for $r$, then $F(\Gamma(q))\subset F(\Gamma(r))$. Indeed, the argument from the paragraph above shows that any $p\in \Gamma(q)$ is sent to a point on at least two horizontal-vertical pairs. Another use of the transverse measure shows that that if we have the germs of two leaves at a point $p\in Z(q)$, they are sent to two distinct germs of leaves exiting $F(p)$, hence $F$ takes zeros of $q$ to zeros of $r$. 

If we repeat the construction above but reverse the roles of $q$ and $r$, we obtain a map $G:\mathbb{D}\to \mathbb{D}$ that inverts $F$ and satisfies the analogous properties. Using $G$, one thus deduces that $F(\Gamma(q))=\Gamma(r)$, and similar for the zero sets $Z(q)$ and $Z(r)$. From here we can show that $F$ is holomorphic away from $Z(q)$: $F$ takes $q$-horizontal and $q$-vertical leaves to $r$-horizontal and $r$-vertical leaves. Using the transverse measures, one can show that it has constant speed on leaves; we omit the details. We extend $F$ over $Z(q)$ using the removeable singularities theorem. 

Since we have the inverse $G$, $F$ is a M{\"o}bius transformation. By following up horizontal and vertical trajectories, we see that $F$ agrees with $m$ on the set of endpoints of horizontal and vertical trajectories. Since $\mu$ and $\nu$ are filling, by Proposition \ref{prop:endpoints_S^1} the set of such endpoints is dense in $S^1$, and hence by continuity we have $F=m$ everywhere on $S^1$. We conclude that $g\circ f^{-1}$ is homotopic to $F$, and hence we are done.
\end{proof}
To handle the general case, one has to lift to the universal covers. To define the analogue of the map $F$, one can argue more or less the same as above, except one has to define $F$ appropriately on ring domains corresponding to atomic geodesics, using the height coordinates on ring domains. 

\begin{proof}[Sketch of energy minimization]
Let $(\mu,\nu)$ be a filling pair. For any homeomorphism $h:X\to Z$ extending to a homeomorphism $\partial_\infty h$ of the ideal boundary, Minsky's inequality (see \cite{SaricShima}) gives 
\begin{equation}\label{eq: Minsky}
    i_X(\mu,\nu)^2 = i_Z((\partial_\infty h)_*\mu,(\partial_\infty h)_*\nu))^2\leq \int_{Z}|q_{\mu}|\int_{Z}|q_{\nu}|,
\end{equation}
where $q_\mu$ and $q_\nu$ are the holomorphic quadratic differentials on a fixed surface $Z$ whose foliations straighten to $(\partial_\infty h)_*\mu$ and $(\partial_\infty h)_*\nu$ respectively. We note that if $(f,Y,q)$ is a holomorphic realization of $(\mu,\nu)$, then equality is achieved by $q$ and $-q$, thus, $$\Big (\int_Y |q|\Big )^2\leq \int_{Z}|q_{\mu}|\int_{Z}|q_{\nu}|.$$
Hence, by the Dirichlet principle from \cite[Theorem 3.2]{SaricShima}, holomorphic realizations minimize the product of the Dirichlet integrals. To show minimization for the sum of the energies, since $$\mathcal{D}(\mathcal{F}_h(q))+\mathcal{D}(\mathcal{F}_v(q))=2||q||_{L^1(Y)}=2(\mathcal{D}(\mathcal{F}_h(q))\mathcal{D}(\mathcal{F}_v(q)))^{1/2},$$ we have 
that for any $h:X\to Z$ with its $q_\mu$ and $q_\nu$ as above,
$$\mathcal{D}(\mathcal{F}_h(q))+\mathcal{D}(\mathcal{F}_v(q))\leq 2\Big (\int_{Z}|q_{\mu}|\int_{Z}|q_{\nu}|\Big )^{1/2}\leq \mathcal{D}(\mathcal{F}_h(q_\mu))+\mathcal{D}(\mathcal{F}_v(q_\nu)),$$ and we deduce the minimizing property again from the Dirichlet principle. 

We argue below that any pair realizing the minimum/equality in (\ref{eq: Minsky}) comes from a holomorphic realization, i.e., we must have $q_\mu=-q_\nu$. Once this is accomplished, it follows from the uniqueness in Theorem \ref{thm: main} that minimization is strict (for both sum and product). Note that by ``strict minimization" for the product, this means strict minimization off the Teichm{\"u}ller line generated by $q$.

To see that any minimizer comes from a holomorphic realization, note that any minimizer certainly minimizes in its own quasiconformal Teichm{\"u}ller space.
Take a $1$-parameter family of quasiconformal maps starting at the minimum, take the corresponding holomorphic quadratic differentials obtained by pushing $\mu$ and $\nu$, and then differentiate the product of the $L^1$-norms. By following the computation from \cite[Theorem 3]{LakicMinimalNorm} (carried out on the unit disk, but which goes through in general), one arrives at a variational formula for a single $L^1$ norm (analogous to \cite[Theorem 1.2]{Wen} on closed surfaces). By the product rule, one gets a variational formula for the product, which shows that a minimizer must have $\frac{q_\mu}{||q_\mu||}=-\frac{q_\nu}{||q_\mu||}$, and the result follows.
\end{proof}

Going back to the proof sketch for uniqueness above, while the map $F$ was easily defined on the complement of the critical trajectories, one has to work through some details to show that it extends to a M{\"o}bius transformation. Using $\R$-trees allows one to eschew these details, at the expense of introducing auxiliary objects.
About energy minimization, compared to the sketch above, the complete proof that we will give below is more geometric in nature and in some sense gives more information.

\subsection{The dual space of a measured lamination}\label{sec: dual space}
We recall the notion of the dual space of a geodesic current, following the presentation from \cite{Didac}. This definition is predated by previous definitions of the dual space of a measured lamination (see \cite{MorganShalenFreeActions}, \cite[\S 11.12]{KapovichHyperbolicManifolds}), which are all essentially equivalent (see \cite[\S 5.1]{Didac}). Note that most presentations restrict to closed surfaces, but things go through on general surfaces with minor modifications. 
Let $X=\mathbb{D}/\Gamma$. A geodesic current $\mu$ on $X$ induces a pseudo-distance $\Tilde{d}_\mu$ on $\mathbb{D}$ by $$\tilde{d}_\mu(x,y) = \frac{1}{2}(\mu(G[x,y))+\mu(G(x,y])),$$ where $G[x,y)$ is the set of geodesics transversely intersecting the half-open geodesic from $x$ to $y$ containing the endpoint $x$, and similar for $G(x,y]$. Since $\mu$ is invariant under the action of $\Gamma$ on $\mathcal{G}(\Tilde{X})$, 
$\Tilde{d}_\mu(x,y)=\Tilde{d}_\mu(\gamma\cdot x,\gamma \cdot y)$ for all $\gamma\in \Gamma.$
\begin{defn}
We define an equivalence relation on $\mathbb{D}$ by $x\sim y$ if $\tilde{d}_\mu(x,y)=0.$ 
   The dual space associated with $\mu$ is the quotient metric space $X_\mu = \mathbb{D}/\sim$, with distance function $d_\mu$ inherited from $\tilde{d}_\mu.$
\end{defn}
We always denote the quotient projection by $\pi_\mu: \mathbb{D}\to X_\mu.$ By the $\Gamma$-invariance of $\Tilde{d}_\mu$ given above, the $\Gamma$-action on $\mathbb{D}$ descends to an action by isometries $\Gamma\to \textrm{Isom}(X_\mu,d_\mu)$, with respect to which $\pi_\mu$ is equivariant.

From now on, assume $\mu$ is a measured lamination. We record results gleaned from \cite{Didac}.
\begin{prop}\label{prop: lamination_properties}
    Let $\mu$ be a geodesic current on $X$, with dual space $X_\mu$ and projection $\pi_\mu$. The following properties hold.
    \begin{enumerate}
        \item $X_\mu$ is a $0$-hyperbolic space if and only if $\mu$ is a measured lamination.
        \item $\pi_\mu$ is continuous if and only if $\mu$ has no atoms.
        \item Assume that the set of atoms $A\subset 
        \mathcal{G}(\mathbb{D})$ of $\mu$ is precisely the set of lifts of closed geodesics supported on $\mu$. Then, $X_\mu$ splits into connected components as $$X_\mu = \Big (\bigsqcup_{g\in A}{\pi_\mu(g)}\Big )\sqcup\Big ( \bigsqcup_{V\in \pi_0(\mathbb{D}\backslash |A|)}\pi_\mu(V)\Big ).$$ 
    \end{enumerate}
\end{prop}
\begin{proof}
    For point (1), note that if $\mu$ is a measured lamination, then for any box of geodesics $B$, $\mu(B)\mu(B^\perp)=0$. From here one can follow the proof of Theorem 6.12 in \cite{Didac}, essentially word-for-word. Point (2) is \cite[Proposition 4.13]{Didac}, the proof of which does not use that the underlying surface is closed. For (3), by the half-open convention in the definition of $\Tilde{d}_\mu$, for any $g\in A$ and $x\not \in g$, $$\Tilde{d}_\mu(x,g)\geq \frac{\mu(\{g\})}{2}.$$ It follows that $\pi_\mu(g)$ is equal to the open ball of radius $\frac{\mu(\{g\})}{2}$ around itself, and hence each $\pi_\mu(g)$ defines an isolated point of $X_\mu$. To see that each $\pi_\mu(V)$, $V\in \pi_0(\mathbb{D}\backslash |A|)$, defines its own connected component, first split $\mu$ into its atomic and non-atomic part as $\mu=\mu_{\textrm{at}}+\mu_0$. Observe that any component $V$ is geodesically convex, since its boundary consists of complete, non-intersection geodesics. Thus, if $x,y\in V$ are connected by a geodesic segment $[x,y]$, the atomic part does not see this segment, and hence $$\Tilde{d}_\mu(x,y)=\Tilde{d}_{\mu_0}(x,y).$$ Moreover, the equivalence relation for $\mu_0$ agrees with that of $\mu$ in $V$, and it follows that $\pi_{\mu_0}(V)\subset X_{\mu_0}$ is naturally isometric to $\pi_\mu(V)\subset X_\mu$. By item (2), $\pi_{\mu_0}$ is continuous, hence $\pi_{\mu_0}(V)$ is connected, and thus $\pi_\mu(V)$ is connected. Finally, each $V\in \pi_0(\mathbb{D}\backslash |A|)$ forms its own component, since if $V_1,V_2\in \pi_0(\mathbb{D}\backslash |A|)$ are not equal then they are separated by some $g\in A$, and hence the $\Tilde{d}_\mu$-distance between any two points in $V_1$ and $V_2$ is at least $\mu(\{g\}).$
\end{proof} 
\begin{remark}
 When $\mu$ has no atoms, so for instance when $X=\mathbb{D}$ and $\mu$ comes from an integrable partial measured foliation, we can write $\tilde{d}_\mu(x,y)=\mu(G(x,y)),$ where $G(x,y)$ is the set of geodesics intersecting the open geodesic from $x$ to $y$.
\end{remark}
From now on, we assume further that $\mu$ arises from tightening an integrable proper partial measured foliation. In particular, by Lemma \ref{lem:asymptotic}, the hypothesis of item (3) in Proposition \ref{prop: lamination_properties} applies to $\mu$. By point (1), $X_\mu$ embeds isometrically inside a (connected) $\R$-tree (see \cite[Theorem 3.38]{EvansProbabilityRealTrees}), but the $\R$-tree is not unique. To obtain a canonical $\R$-tree, analogous to \cite[\S 11.12]{KapovichHyperbolicManifolds}, we replace $\mathbb{D}$ with a blown-up disk $\mathbb{D}^b$ as in \S \ref{subsec: realization, general} (this time we're blowing up with respect to just $\mu$, and not two laminations, pardon the abuse of notation). That is, for each atomic geodesic $g\in A$, we cut it out and glue in the strip $C_g = g\times  [-\frac{\mu(\{g\})}{2},\frac{\mu(\{g\})}{2} ]$ in the same fashion as in \S \ref{subsec: realization, general}.
Define a pseudo-distance $\Tilde{d}_{\mu}^b$ on $\mathbb{D}^b$ by setting, in every component of $\mathbb{D}\backslash A$, $\Tilde{d}_\mu^b=d_\mu$,  and in every $C_g$, $\Tilde{d}_\mu^b(x,y)$ equal to the Euclidean distance between the vertical coordinates of $x$ and $y$, and then gluing the pseudo-distances together according to the gluing that made up $\mathbb{D}^b$. 
We define $T_\mu$ to be the quotient of $\mathbb{D}^b$ by the resulting equivalence relation determined by $\Tilde{d}_\mu^b$. By abuse of notation, we denote the resulting distance function by $d_\mu$, and the projection by $\pi_\mu:\mathbb{D}^b\to (T_\mu,d_\mu).$ As in \S \ref{subsec: realization, general}, $\Gamma$ acts on $\mathbb{D}^b$ by the usual action on $\mathbb{D}\backslash A$ and by permuting the strips corresponding to $A$. Hence, for the same reason as with $X_\mu$, $\pi_\mu$ is equivariant for an action of $\Gamma$ by isometries on $T_\mu.$

Each strip $C_g$ maps onto a segment $I_g$ of length $\mu(\{g\})$ in $T_\mu$. There is a natural (discontinuous) map from $T_\mu\to X_\mu$, built by descending the identity map on $\mathbb{D}\backslash A$, and collapsing each segment $I_g$ to the isolated point in $X_\mu$ corresponding to $g$. Outside of the segments, the map is an isometry, and this hopefully excuses the double use of the notation $d_\mu$.
\begin{prop}
    $(T_\mu,d_\mu)$ is an $\R$-tree.
\end{prop}
\begin{proof}
   By item (1) in Proposition \ref{prop: lamination_properties} and \cite[Theorem 3.40]{EvansProbabilityRealTrees}, the connected components of $X_\mu$ are $\R$-trees. Under the map $T_\mu\to X_\mu$, these components lift to unique subsets. When two such components $\pi_\mu(V_1)$ and $\pi_\mu(V_2)$ come from a pair $V_1$ and $V_2$ separated by a geodesic $g$, they are connected in $T_\mu$ by the segment $I_g$.
   It is thus clear that $T_\mu$ is a geodesic metric space. The second defining property of an $\R$-tree is also easily verified: segments in $T_\mu$ are exactly the connected sets obtained by attaching a segment in $X_\mu$ to subintervals of the attached intervals of $T_\mu$, and under this characterization it is immediate that the joining of two segments intersecting at an endpoint is a segment.
\end{proof}

\subsection{Foliations and maps to $\R$-trees}\label{sec: maps from foliations}
We continue with the Riemann surface $X=\mathbb{D}/\Gamma$ and the lamination $\mu$.

\subsubsection{Foliations}
First, we explain how an integrable proper partial measured foliation tightening to a lamination $\mu$ gives rise to a $W_{\textrm{loc}}^{1,2}$ map (in the sense of \cite{KorevaarSchoen}) to $(T_\mu,d_\mu)$. 

Let $\mathcal{F}$ be an integrable proper partial measured foliation on $X$, and let $\mu$ be the associated geodesic lamination. We lift $\mathcal{F}$ to a foliation $\Tilde{\mathcal{F}}$ of $\mathbb{D}$. Let $U_{\mathcal{F}}$ be the subset of $\mathbb{D}$ consisting of the points lying on non-singular trajectories of $\Tilde{\mathcal{F}}$ that accumulate to distinct endpoints of $S^1$. 

Resume the notation from the blow-up construction. If $g$ is an atom of $\mu$, recall the total transverse measure of a path crossing the set of leaves $L_g$ of $\Tilde{\mathcal{F}}$ tightening to $g$ is $\mu(\{g\})$. There is a height map $h_g: L_g\to [-\frac{\mu(\{g\})}{2},\frac{\mu(\{g\})}{2}]$ defined as follows. Choose a complete geodesic $\hat{g}$ crossing $g$, oriented so that it starts on the component of $\mathbb{D}\backslash g$ with boundary component $g_-$. Necessarily, $\hat{g}$ crosses every $\ell \in L_g$. Let $\Tilde{\mathcal{F}}_{L_g}$ be the partial measured foliation obtained by removing every leaf not in $L_g$.  Given $\ell\in L_g$, we erase the component of $\hat{g}$ that goes past $\ell$, denote the resulting curve by $\hat{g}_\ell$, and define $h_g(\ell)$ to be the $\Tilde{\mathcal{F}}_{L_g}$-transverse measure of $\hat{g}_\ell$ minus $\frac{\mu(\{g\})}{2}$. It is clear that the choice of $\hat{g}$ was not important.
\begin{defn}\label{def: foliation map}
    Define $\pi_{\mathcal{F}}:U_{\mathcal{F}}\to (T_\mu,d_\mu)$ as follows. Let $\ell$ be a leaf in $U_{\mathcal{F}}$ and let $g_\ell$ be the geodesic tightening.
    \begin{enumerate}
        \item If $g_\ell$ is not atomic for $\mu$, define $\pi_{\mathcal{F}}$ on $\ell$ by $\pi_{\mathcal{F}}(\{\ell\})=\pi_\mu(\{g_\ell\}).$ 
        \item If $g_\ell$ is atomic for $\mu$, let $C_g=g\times [-\frac{\mu(\{g\}}{2},\frac{\mu(\{g\})}{2}]$ be the associated strip in the blown-up surface $\mathbb{D}^b$, and define, in the coordinates determined by $C_g$, $\pi_{\mathcal{F}}(\{\ell\})=\pi_\mu(\{g\}\times \{h_g(\ell)\}).$
    \end{enumerate}
\end{defn}
See Figure \ref{fig:tree map} for an example on $X=\mathbb{D}$. It should be clear from the definition that $\pi_{\mathcal{F}}$ is equivariant for the same $\Gamma$-action on $(T_\mu,d_\mu)$ as $\pi_\mu$.

For an explicit description of the map $\pi_{\mathcal{F}}$, fix a foliation chart $U$ with defining function $v:U\to \mathbb{R}$. In $U$, the transverse measure between two points $x,y\in U$ is just $|v(x)-v(y)|$. When $x$ and $y$ lie on leaves that tighten to non-atomic geodesics $g_x$ and $g_y$, this is equal by definition to the $\mu$-measure of all geodesics separating $g_x$ and $g_y$ (recall from \S \ref{sec: realizable laminations}). Thus, $\pi_\mathcal{F}(x)$ and $\pi_{\mathcal{F}}(y)$ are connected by a segment in $(T_\mu,d_\mu)$ of length $|v(x)-v(y)|$. If either of $g_x$ and $g_y$ are atomic, the same result holds thanks to item (2) in Definition (\ref{def: foliation map}).
It follows that in $U\cap U_{\mathcal{F}}$, $\pi_{\mathcal{F}}$ maps into a segment of $T_\mu$, and upon identification of that segment with an interval in $\R$, $\pi_{\mathcal{F}}$ is represented by $\pm v +c$ for some constant $c$. After modifying the identification, we can assume $\pi_{\mathcal{F}}$ is represented by $v$. Since every $v:U\to \R$ is $W_{\textrm{loc}}^{1,2}$ (in the usual sense), $\pi_{\mathcal{F}}$ is $W_{\textrm{loc}}^{1,2}$ (in the Korevaar-Schoen sense). Moreover, it is easily seen that the Dirichlet energy $E(\pi_{\mathcal{F}})$ of $\pi_{\mathcal{F}}$ is equal to $1/2$ times the Dirichlet integral $\mathcal{D}(\mathcal{F})$ of $\mathcal{F}$.

When the transverse measure has enough uniform continuity properties, we can hope to extend it past $U_{\mathcal{F}}$. This is the case in the main situation of interest: when $\mathcal{F}$ is the horizontal foliation $\mathcal{F}_h(q)$ of an integrable holomorphic quadratic differential $q$ on $X$. Indeed, for a quadratic differential lifted to the universal cover, recall that all regular leaves accumulate to two distinct endpoints on $S^1$. 
Writing $\pi_{\mathcal{F}_h(q)}=:\pi_q$, we just have to define $\pi_q$ on the critical trajectories, or equivalently on the generalized trajectories. Recall from \S \ref{sec: realizable laminations} that every generalized trajectory $\ell$ accumulates to two distinct endpoints in $S^1$. Thus, we can associate $\ell$ to a unique geodesic $g_\ell$ in the support of $\mu$. In $g_\ell$ is not $\mu$-atomic, then we can define $\pi_q(\ell)=\pi_\mu(g_\ell)$, as above. If $g_\ell$ is $\mu$-atomic, then the trajectory $\ell$ is on the boundary of the lift of a ring domain, and we can assign it to the corresponding endpoint of the interval $I_{g_\ell}$. Note that a point in $\mathbb{D}$ can lie on multiple generalized trajectories, but this is not a problem because such trajectories have zero transverse distance. By the procedure of the previous paragraph, we see that, away from the zeros of $q$, $\pi_q$ is represented in a natural coordinate $z=x+iy$ for $q$ by $z\mapsto y$.
\begin{prop}\label{prop: Dirichlet principle}
    Let $q$ be an $L^1$ holomorphic quadratic differential on $X$ and let $(T_q,d_q)$ be the $\R$-tree associated with the horizontal foliation. Then the map $\pi_q:\mathbb{D}\to (T_q,d_q)$ is harmonic with Hopf differential $-\frac{1}{4}\Tilde{q}$, where $\Tilde{q}$ is the lift to $\mathbb{D}$, and for any other integrable proper partial measured foliation $\mathcal{F}$ tightening to the same lamination, $E(\pi_q)\leq E(\pi_\mathcal{F}|_{U_{\mathcal{F}}}).$
\end{prop}
\begin{proof}
The proof of  \cite[Proposition 2.6]{Sagman} goes through essentially word for word to show that $\pi_q$ is harmonic.
   In a natural coordinate for $q$, $\pi_q$ can be described as $z=x+iy\mapsto y$, which has $z$-derivative $i$. By direct computation, the Hopf differential is $-\frac{1}{4}q$. The global energy minimizing property follows from the identifications with the energies of the foliations, and the Dirichlet principle proved in \cite{SaricShima}.
\end{proof}

\begin{remark}
    Another situation in which $\pi_{\mathcal{F}}$ extends globally to $\mathbb{D}$ is when $\mathcal{F}$ is the pushforward via a diffeomorphism of the horizontal foliation of a quadratic differential on some other Riemann surface. Then $\mathcal{F}$ has the same structure as a union of regular and generalized trajectories. This example will come up in \S \ref{sec: main inequality}.
\end{remark}

\begin{figure}[htb]
    \centering
    \includegraphics[width=9 cm]{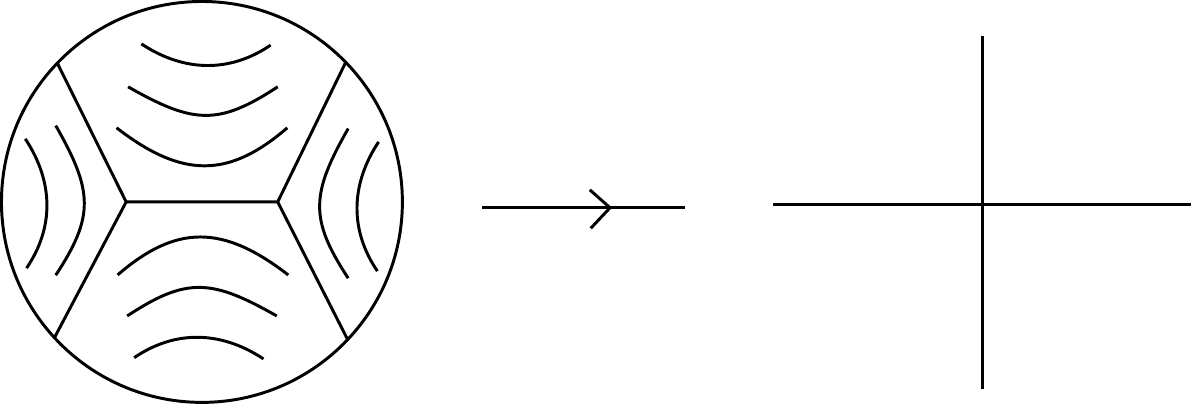}
    \caption{The map to a $\R$-tree associated with the vertical foliation of $q=(z^2-1)dz^2$ on the unit disk.}
    \label{fig:tree map}
\end{figure}

\subsubsection{Homeomorphic pushes}
Let $f:X\to Y$ be a homeomorphism to another Riemann surface that extends to a homeomorphism $\partial_\infty f: S^1\to \partial_\infty \Tilde{Y}$. We explain how $m:=\partial_\infty f$ and an integrable proper partial measured foliation $\mathcal{F}_m$ on $Y$ tightening to $m_*\mu$ determines another (partially defined) map to $(T_\mu,d_\mu)$. Note that $f$ itself is not remembered by the construction, just $m$. In some sense, the trees for $\mu$ and $m_*\mu$ are canonically identified, and we are passing through this identification to obtain a different map. Given $g$ in the support of $\mu$, pushing its endpoints via $m$ gives a geodesic $m_* g$ in the support of $m_*\mu$. Of course, atoms for $\mu$ are sent to atoms for $m_*\mu$. The strip $C_{m_*g}$ in the corresponding blown-up surface $\Tilde{Y}^b$ has the same heights as that of $C_g$ in $\mathbb{D}^b$, and there are two isometries $C_g\to C_{m_*g}$ (they are related by reflection across the midpoint). We choose our isometry $m_*^g$ so that $(m_*g)^{\pm}=m_* (g^\pm).$ We define a map $\pi_\mu^m$ on the union of the set of geodesics $\{m_* g: g\textrm{ non-atomic}\}$ with the $C_{m_*g}$'s, by $\pi_\mu^m(\{m_* g\})=\pi_\mu(\{g\})$ on non-atomic $g$, and $\pi_\mu^m|_{C_{m_*g}}=\pi_\mu|_{C_g} \circ (m_*^g)^{-1}$.

As above, the main case of interest is when $\mathcal{F}_m$ is the horizontal foliation of a holomorphic quadratic differential $q_m$. Then, totally analogous to above, $\pi_{q_m}:=\pi_{\mathcal{F}_m}$ can be globally defined on $\Tilde{Y}$ (by tightening generalized trajectories). Similar to above, it is easily seen that $\pi_{q_m}$ is energy minimizing with Hopf differential $-\frac{1}{4}q_{m}$.

\subsection{The core of a pair of filling laminations}\label{subsec: the core}
Let $X=\mathbb{D}/\Gamma$ be a Riemann surface with $L^1$ holomorphic quadratic differential $q$, and let $\mu$ and $\nu$ be the pair of laminations associated with the horizontal and vertical foliations. Since the vertical foliation of $q$ is the horizontal foliation of $-q$, we write the corresponding leaf spaces and projections as $\pi_q: \mathbb{D}\to (T_\mu,d_\mu)$ and $\pi_{-q}:\mathbb{D}\to (T_\nu,d_\nu)$ respectively. Setting $(P,d)$ to be the product metric space $(P,d)=(T_\mu\times T_\nu,d_\mu\times d_\nu)$, where $d_\mu\times d_\nu$ denotes the $\ell^2$ product metric, $\pi_q$ and $\pi_{-q}$ combine to give a map $$\Pi_q:=(\pi_q,\pi_{-q}):\mathbb{D}\to (P,d).$$ The harmonic map $\Pi_q:\mathbb{D}\to (P,d)$ is minimal, since the Hopf differential is $(-\frac{1}{4}q)+\frac{1}{4}q=0.$
\begin{defn}\label{def: core}
     The \textit{holomorphic core} of the pair $(\mu,\nu)$ is $\mathcal{C}=\Pi_q(\mathbb{D})\subset (P,d)$.
\end{defn}

\begin{remark}
If we equip $X$ with the singular flat metric $|\tilde{q}|$, then the map $\Pi_q: (\mathbb{D},|\tilde{q}|)\to (P,d)$, which is injective by the Teichm{\"u}ller lemma, is an isometry. 
\end{remark}

The lemma below about $\mathcal{C}$ is the key input toward both uniqueness and the energy minimizing property. Let $f:X\to Y$, $m=\partial_\infty f:S^1\to \partial_\infty \Tilde{Y}$ be as in the subsection above, and assume that the push-forward laminations $m_*\mu$ and $m_*\nu$ are realized by a pair of $L^1$ holomorphic quadratic differentials. The maps $\pi_{q_1}: \Tilde{Y}\to (T_\mu,d_\mu)$ and $\pi_{q_2}:\Tilde{Y}\to (T_\nu,d_\nu)$ together yield a product map $\Pi_q^m=(\pi_{q_1},\pi_{q_2}): \Tilde{Y}\to (P,d).$
\begin{lem}\label{lem: image contains the core}
    The image $\Pi_q^m(\Tilde{Y})$ contains the holomorphic core $\mathcal{C}$.
\end{lem}
\begin{proof}
Let $\Tilde{q}$, $\Tilde{q}_1$, and $\Tilde{q}_2$ be the lifts of $q$, $q_1$, and $q_2$ respectively to the universal cover. Any point $x\in\mathcal{C}$ is the image of an intersection point of two regular or generalized trajectories $\ell_1$ and $\ell_2$ of $\Tilde{q}$ and $-\Tilde{q}$ respectively. We recall that, by the Teichm{\"u}ller lemma, $\ell_1$ and $\ell_2$ intersect at exactly one point. The trajectory $\ell_1$ separates $\overline{\mathbb{D}}$ into two components, and because the intersection point of the trajectories is unique and by Lemma \ref{lem:asymptotic}, the two ends of $\ell_2$ live in different components.

    Let $g_1$ and $g_2$ be the geodesic tightenings of $\ell_1$ and $\ell_2$ respectively. Just like $\ell_1$, $g_1$ cuts $\overline{\mathbb{D}}$ into two components. Since the endpoints of $g_2$ are the endpoints of $\ell_2$, these endpoints live in different components of the cutting by $g_1$. Hence, $g_1$ and $g_2$ intersect. 

    Let $a_1$ and $b_1$ be the ideal endpoints of $g_1$, and let $a_2$ and $b_2$ be the endpoints of $g_2$. Let $m_* g_1$ be the geodesic on $\Tilde{Y}$ with ideal endpoints $m(a_1)$ and $m(b_1)$, and let $m_* g_2$ be the geodesic on $\Tilde{Y}$ with endpoints $m(a_2)$ and $m(b_2)$. Since $m(a_1),m(b_1),m(a_2),$ and $m(b_2)$ have the same ordering on $\partial \Tilde{Y}\simeq S^1$ as $a_1,b_1,a_2,$ and $b_2$, the geodesics $m_*g_1$ and $m_* g_2$ intersect. Because $m_*\mu$ and $m_*\nu$ come from quadratic differentials, $m_* g_1$ and $m_* g_2$ arise from straightening unique regular or generalized trajectories $\hat{\ell}_1$ and $\hat{\ell}_2$ for $\Tilde{q}_1$ and $\Tilde{q}_2$ respectively. Using the boundary pattern again, $\hat{\ell}_1$ is forced to separate the endpoints of $\hat{\ell}_2$, and hence $\hat{\ell}_1$ and $\hat{\ell}_2$ intersect at least once (but maybe more than once). To visualize all this, see Figure \ref{fig: leaf map}.
    
    If the geodesic $g_1$ and $g_2$ are not atomic, then $\hat{\ell}_1$ is unique.
    If $g_1$ is $\mu$-atomic, then $\ell_1$ is contained in the lift of a closed ring domain of leaves (and bounded by two critical leaves) that maps onto the segment $I_{g_1}$ in $(T_\mu,d_\mu)$, and $\ell_1$ is mapped onto a unique point. The geodesic $m_*g_1$ is $m_*\mu$-atomic, and by definition of the pushforward measure, there is a lifted closed ring domain for $\Tilde{q}_1$, with leaves tightening to $m_* g_1$, mapping onto $I_{g_1}$ via $\pi_{q_1}.$ We select $\hat{\ell}_1$ to be the unique leaf mapping onto the same point as $\ell_1$. For $g_2$, we carry out the same procedure: if it's not atomic, then $\hat{\ell}_2$ is unique, and otherwise we select $\hat{\ell}_2$ in the same fashion. Now, we've chosen $\hat{\ell}_1$ and $\hat{\ell}_2$ so that $\pi_{q_1}(\{\hat{\ell}_1\})=\pi_q(\{\ell_1\})$ and $\pi_{q_2}(\{\hat{\ell}_2\})=\pi_{-q}(\{\ell_2\})$. Consequently $\Pi_q^m(\{\hat{\ell}_1\cap\hat{\ell}_2\})=x$. Since $x\in\mathcal{C}$ was arbitrary, we conclude that $\Pi_q^m(\Tilde{Y})\supset\mathcal{C}$.
\end{proof}
\begin{figure}[htb]
    \centering
    \includegraphics[width=7 cm]{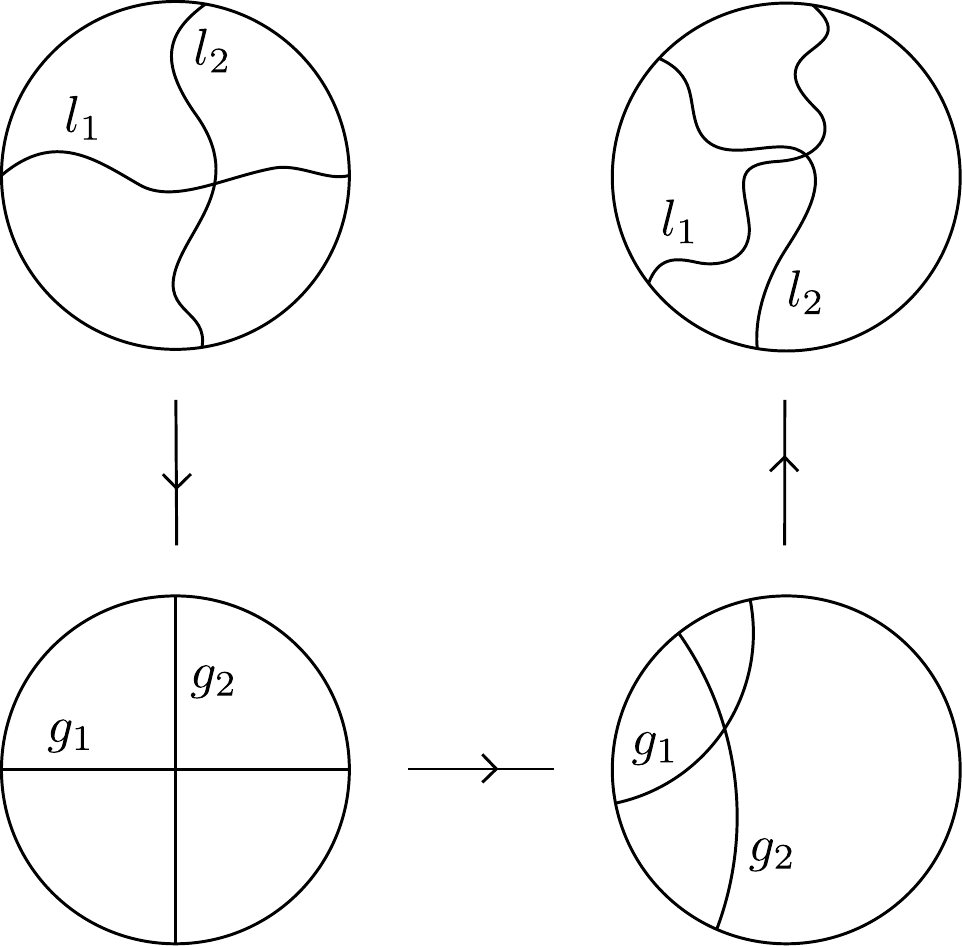}
    \caption{}
    \label{fig: leaf map}
\end{figure}

To see an example with $\Pi_q^m(\tilde{Y})\supsetneq \mathcal{C}$, consider a homeomorphism $m$ such that a trajectory $\ell_1$ of $q_1$ and a trajectory $\ell_2$ of $q_2$ intersect in $\mathbb{D}$ but the endpoints of $\ell_1$ do not separate the endpoints of $\ell_2$ on $S^1$. The image of each intersection point of $\ell_1$ and $\ell_2$ is $(\pi_{q_1}(\ell_1),\pi_{q_2}(\ell_2))\notin\mathcal{C}$ (see Figure \ref{fig:extra_intersections})
\begin{figure}[htb]
    \centering
    \includegraphics[width=4.5 cm]{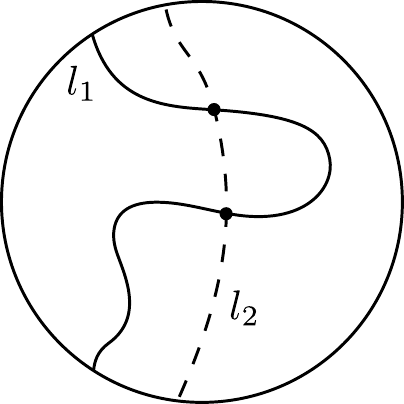}
    \caption{The two intersection points of $\ell_1$ and $\ell_2$ have their image outside $\mathcal{C}$.}
    \label{fig:extra_intersections}
\end{figure}

The assumption that $\pi_{q_1}$ and $\pi_{q_2}$ arose from foliations of quadratic differentials was not really necessary; we just needed that every geodesic in the support of $m_*\mu$ comes from tightening a leaf of the foliation, and likewise for $m_*\nu$.

Finally, the additional lemma below will be used for uniqueness.
\begin{lem}\label{lem: homeo onto the core}
    If $q_1=-q_2$, then $\Pi_q^m$ defines a bijection from $\Tilde{Y}$ onto $\mathcal{C}$.
\end{lem}
\begin{proof}
    If $q_1=-q_2$, then $\Pi_q^m$ is injective by the Teichm{\"u}ller lemma. By Lemma \ref{lem: image contains the core}, it suffices to show $\Pi_q^m(\Tilde{Y})\subset \mathcal{C}.$ The argument is the same as that of Lemma \ref{lem: image contains the core}. Any point $x=(p,q)\in \Pi_q^m(\Tilde{Y})$ is the unique intersection of two trajectories $\hat{\ell}_1$ and $\hat{\ell}_2$ of $F_h(\Tilde{q}_1)$ and $F_v(\Tilde{q}_1)$ respectively. By reversing the procedure from the proof of Lemma \ref{lem: image contains the core}, we produce leaves $\ell_1$ and $\ell_2$ of $F_h(\Tilde{q})$ and $F_v(\Tilde{q})$ respectively such that $\pi_{q}(\ell_1)=\pi_{q_1}(\hat{\ell}_1)$ and $\pi_{-q}(\ell_2)=\pi_{q_2}(\hat{\ell}_2)$. Thus, $\Pi_q$ takes the intersection point of $\ell_1$ and $\ell_2$ to $x$, and hence $x\in \mathcal{C}$.
\end{proof}

\subsection{Proof of uniqueness and energy minimization}
Let $\mu$ and $\nu$ be a pair of filling laminations arising from tightening integrable proper partial measured foliations on a Riemann surface $X$, and let $(f,Y,q)$ be a holomorphic realization.
\begin{thm}\label{thm: uniquness}
    Let $(h,Z,r)$ be another holomorphic realization. Then there exists a biholomorphism $F:Y\to Z$ extending to $\partial_\infty h\circ (\partial_\infty f)^{-1}$ on $\partial\Tilde{Y}$ such that $F_*q=r$.
\end{thm}
\begin{proof}
    By applying $f$, we can relabel everything so that $X=Y$ and $f$ is the identity map. Let $\Pi_q:\Tilde{Y}\to P$ be the conformal harmonic map into the product of $\R$-trees, and let $\mathcal{C}$ be the holomorphic core. Setting $m:=\partial_\infty h$, we also have the harmonic map $\Pi_q^m:\Tilde{Z}\to P$, which is also conformal. By Lemma \ref{lem: homeo onto the core}, $\Pi_q^m$ is a homeomorphism onto $\mathcal{C}$. Thus, we can define the map $$F=(\Pi_q^m)^{-1}\circ \Pi_q: \Tilde{Y}\to \Tilde{Z},$$ which is equivariant for the deck group actions on $\Tilde{Y}$ and $\Tilde{Z}$, and hence descends to a map from $Y\to Z$. Since both maps $\Pi_q$ and $\Pi_q^m$ are isometries from singular flat surfaces onto $\mathcal{C}$, the map $F$ is weakly conformal. Since $Y$ and $Z$ have the same orientation, $F$ is holomorphic. Being injective and holomorphic, $F$ is in fact a diffeomorphism onto its image. The same reasoning can be applied to $$G=\Pi_q^{-1}\circ \Pi_q^m:\Tilde{Z}\to \Tilde{Y},$$ which inverts $F$, and hence $F$ is a biholomorphism.

    Finally, being a biholomorphism, $F:\Tilde{Y}\to \Tilde{Z}$ extends to a map $\partial_\infty F: \partial_\infty\Tilde{Y}\to \partial_\infty\Tilde{Z}$. By construction, the boundary extension agrees with $\partial_\infty h$ on the set of endpoints of leaves of $\mu$ and $\nu$. Since $\mu$ and $\nu$ are filling, this set is dense in $\partial \Tilde{Y}$ (Proposition \ref{prop:endpoints_S^1}). By continuity, $\partial_\infty F=\partial_\infty h$.
\end{proof}

Toward the proof of energy minimization, we take note of one basic fact.
Let $F$ be any $W_{\textrm{loc}}^{1,2}$ equivariant map from $\mathbb{D}$ to a non-positively curved geodesic metric space with measurable pullback metric $g(F)$ and finite energy. The area of $F$ is $$A(F) = \int_X \sqrt{\det g(F)}.$$ The definition is motivated by the fact that when $F$ is an embedding into a Riemann manifold, $A(F)$ returns the area of the image. In general, it's essentially a weighted Hausdorff $2$-measure (see \cite{Kirchheim}). Analogous to the case of manifolds, an application of Cauchy-Schwarz on the components of the measurable pullback metric yields 
$$A(F)\leq E(F),$$ with equality if and only if $F$ has vanishing Hopf differential.
\begin{thm}\label{thm: energy minimizing}
    Let $h:X\to Z$ be a homeomorphism to another Riemann surface $Z$ that extends to a homeomorphism $h:\partial_\infty X\to \partial_\infty Z$. Let $\mathcal{F}'$ and $\mathcal{G}'$ be any pair of integrable proper partial measured foliations tightening to $(\partial_\infty h)_*\mu$ and $(\partial_\infty h)_*\nu$ respectively. Then 
    \begin{equation}\label{eq: energy min}
        \mathcal{D}(\mathcal{F}_h(q))+\mathcal{D}(\mathcal{F}_v(q))\leq \mathcal{D}(\mathcal{F}')+\mathcal{D}(\mathcal{G}'),
    \end{equation}
 with equality if and only if $h:X\to Z$ is part of a holomorphic realization.
\end{thm}
\begin{proof}
As above, assume that $X=Y$ and $f$ is the identity map.
    Let $m:=\partial_\infty h$ and let $q_1$ and $q_2$ be the integrable holomorphic quadratic differentials associated with $m_*\mu$ and $m_*\nu$ respectively, and let $\Pi_q^m:\Tilde{Z}\to P$ be the corresponding harmonic map into a product of $\R$-trees. By Lemma \ref{lem: image contains the core}, $\Pi_q^m(\Tilde{Z})\supset \mathcal{C}$. It follows from \cite[Corollary 8]{Kirchheim} that $$A(\Pi_q)\leq A(\Pi_q^m).$$ Since $\Pi_q$ is conformal and energy dominates area, $$E(\Pi_q)=A(\Pi_q)\leq A(\Pi_q^m) \leq E(\Pi_q^m).$$ The left hand side of the equation above is equal to $\frac{1}{2}(\mathcal{D}(\mathcal{F}_h(q))+\mathcal{D}(\mathcal{F}_v(q)))$. By the Dirichlet principle (Proposition \ref{prop: Dirichlet principle}), $E(\Pi_q^m)\leq \frac{1}{2}(\mathcal{D}(\mathcal{F}')+\mathcal{D}(\mathcal{G}'))$, and hence (\ref{eq: energy min}) is proved. 

    If $h:X\to Z$ is another holomorphic realization, then equality holds by uniqueness. Conversely, if equality holds in (\ref{eq: energy min}), then $A(\Pi_q^m) = E(\Pi_q^m)$, and this occurs if and only if $\Pi_q^m$ is weakly conformal. Since the $(2,0)$ component of the pullback metric for $\Pi_q^m$ is $-q_1-q_2$, this forces $q_1=-q_2.$
\end{proof}

\subsection{A question}\label{sec: a question} The proof of Theorem \ref{thm: main} is now complete. We follow it up with a question. Let $(\mu,\nu)$ be a pair of laminations on a Riemann surface $X$ homotopic to foliations with finite Dirichlet integral. Pushing the laminations to other Riemann surfaces via quasiconformal maps, the finite Dirichlet integral property is preserved. Thus, we can define an energy functional $$\mathcal{E}: T_{qc}(X)\to (0,\infty)$$ as follows. For $f:X\to Y$ representing a point in $T_{qc}(X)$, $\mathcal{E}([f])$ is the sum of the $L^1$ norms of the two quadratic differentials on $Y$ realizing the laminations $f_*\mu$ and $f_*\nu$. Equivalently, it is the sum of the Dirichlet integrals of the foliations (called the extremal length in \cite{SaricShima}), or the sum of the energies of the two harmonic maps.

By \cite[Theorem 3]{LakicMinimalNorm}, $\mathcal{E}$ has a critical point at $[f]$ if and only if the quadratic differentials are opposite (\cite[Theorem 3]{LakicMinimalNorm} concerns the unit disk, but the proof goes through in general). That is, calling them $q$ and $-q$, $(f,Y,q)$ will be a holomorphic realization. By the energy minimizing property from Theorem \ref{thm: main}, any critical point of $\mathcal{E}$ is a minimum and any critical point is unique. If $(\mu,\nu)$ is not quasiconformally holomorphically realizable, then $\mathcal{E}$ cannot be minimized in $T_{qc}(X)$ (but we can pass to another quasiconformal Teichm{\"u}ller space, and define the analogous functional).

\begin{que}\label{que: converge}
    Assume that $(\mu,\nu)$ is quasiconformally holomorphically realizable and let $h: S^1\to  S^1$ be the quasisymmetric boundary map associated with the holomorphic realization. Let $h_n: S^1\to S^1 $ be the quasisymmetric boundary maps along a minimizing sequence in $T_{qc}(X)$ for $\mathcal{E}$. Are the $h_n$'s uniformly quasisymmetric, and do they converge to $h$?
\end{que}
On closed surfaces the answer is yes, and this is Wentworth's proof of Gardiner-Masur's theorem from \cite{Wen}. We believe that the answer to Question \ref{que: converge} is yes, but we leave its pursuit for another day.

\section{Asymptotic Plateau problems}\label{sec: plateau}
Here we study the boundary behaviour of maps to $\R$-trees, and reinterpret Theorem \ref{thm: main} as an existence and uniqueness result for an asymptotic Plateau problem. Throughout this section, let $X=\mathbb{D}/\Gamma$ be a fixed Riemann surface. Everything up to \S \ref{subsec: asymptotic plateau} is vacuous if $X$ is of the first kind and the most important case to have in mind is $X=\mathbb{D}$. Throughout, let $\Lambda_\Gamma$ be the limit set of $\Gamma$, with complement $\Lambda_\Gamma^c$ in $S^1$ (for $X=\mathbb{D}$, $\Lambda_\Gamma^c=S^1$).

In the work below, we're going to take Gromov boundaries of $\R$-trees (in this context, also called the visual boundary or the space of ends). Recall that the Gromov boundary of a $\delta$-hyperbolic metric space $(M,d)$ can be defined as follows. For a basepoint $o\in M$, the Gromov product of two points $x,y\in M$ is $$(x|y)_o = \frac{1}{2}(d(x,o)+d(y,o)-d(x,y)).$$ A sequence is a Gromov sequence if $$\lim_{n,m\to\infty} (x_n|x_m)_o=\infty.$$ Two Gromov sequences are equivalent if $$\lim_{n\to\infty} (x_n|y_n)_o=\infty.$$ The Gromov boundary $\partial_\infty M$ is the set of equivalence classes of Gromov sequences, with the topology built out of basic open sets $$U([x_n],R)=\{[y_n]\in \partial_\infty M: \liminf_{n\to\infty} (x_n|y_n)_o>R\}.$$

\subsection{Boundary maps}
Let $\mu$ be a measured lamination on $X$ arising from tightening a partial measured foliation with finite Dirichlet integral. Let $\mathbb{D}^b$ be the blown-up disk specified by atoms of $\mu$ (if any). The atoms won't play much of a role--we work on $\mathbb{D}^b$ just so that the leaf-space projection $\pi_\mu: \mathbb{D}^b\to (T_\mu,d_\mu)$ is nicely defined. Since $\mathbb{D}^b$ is obtained by adding strips to $\mathbb{D}$ along geodesics (with every strip having a unique geodesic), we can attach $S^1$ to $\mathbb{D}^b$ and compactify to $\overline{\mathbb{D}}^b:=\mathbb{D}^b\cup S^1$ so that the boundary map $\overline{\mathbb{D}}^b\to \overline{\mathbb{D}}$ is the identity. Henceforth, we always identify the boundaries, and when working with $S^1$ we don't specify $\mathbb{D}$ or $\mathbb{D}^b$.

Let $(\overline{T}_\mu,d_\mu)$ be the metric completion of the leaf space $(T_\mu,d_\mu)$. We construct the canonical boundary map $$\beta:\Lambda_\Gamma^c\to \overline{T}_\mu\cup \partial_\infty \overline{T}_\mu$$ associated with $\mu$. Fixing a point $o\in \mathbb{D}^b$, for $p\in \Lambda_\Gamma^c$ let $$r_p:[0,\infty)\to \mathbb{D}^b$$ be the geodesic ray from $o$ to $p$; here, by geodesic ray, we mean the corresponding ray in $\mathbb{D}$ with added (shortest path) segments corresponding to the (flat) strips of $\mathbb{D}^b$. We will consider the maps $\pi_\mu\circ r_p: [0,\infty)\to (\overline{T}_\mu,d_\mu)$, which are essentially rays.

\begin{prop}\label{prop: canonical boundary map}
    For every $p\in \Lambda_\Gamma^c$, after deleting constant subintervals and reparametrizing by arc-length, $\pi_\mu\circ r_p$ is either
    \begin{enumerate}
        \item a finite geodesic segment, converging to a point $\beta_\mu(p)$ in $\overline{T}_\mu$, or 
        \item a geodesic ray, and sequences along the ray determine a point $\beta_\mu(p)$ in $\partial_\infty \overline{T}_\mu$. 
    \end{enumerate}
    The resulting point $\beta_\mu(p)\in \overline{T}_\mu\cup \partial_\infty \overline{T}_\mu$ is independent of the initial point $o$.
\end{prop}
Granting the proposition, the map $\beta_\mu: \Lambda_\Gamma^c\to \overline{T}_\mu\cup \partial_\infty \overline{T}_\mu$ is $p\mapsto \beta_\mu(p)$.
\begin{proof}
    Since $p\in \Lambda_\Gamma^c$, by Lemma \ref{lem:asymptotic}, $\mu(E_p)=0$. Thus, up to a null-set, every geodesic in the support of $\mu$ splits $\mathbb{D}$ into two components, one which contains $p$ and one which does not. Let $A(p)$ be the set of geodesics $g$ in $\mathbb{D}$ separating $o$ from $p$. Since $r_p$ has to eventually cross any $g\in A(p)$, and every $g\in A(p)$ has to hit some $r_p(t)$, 
$$d_\mu(\pi_\mu(o),\pi_\mu(r_p(t))) \underset{t\to\infty}{\nearrow} \mu(A(p)).$$
    If $\mu(A(p))<\infty$, then $\pi_\mu\circ r_p(t)$ is Cauchy in $t$ and hence converges to a point $\beta_\mu(p)$ in $\overline{T}_\mu$. If $\mu(A(p))=\infty$, then since the distance along $\pi_\mu(r_p(t))$ is additive, any sequence that goes along the image of $\pi_\mu\circ r_p$ determines an (equivalent) Gromov sequence, and hence a point $\beta_\mu(p)$ in $\partial_\infty \overline{T}_\mu$.

    If $o'$ is another basepoint with set of geodesics $B(p)$ separating $o'$ from $p$, then up to a null set, the symmetric difference $A(p)\Delta B(p)$ is the set of geodesics separating $o$ and $o$'. Since geodesics separating the basepoints all intersect a compact segment (the geodesics separating $o$ and $o'$), the symmetric difference forms a compact set of geodesics. It follows that a change of basepoint gives two different paths in the tree that eventually coincide, and thus determine the same endpoints.
\end{proof}

We relate $\beta_\mu$ to distance functions. Previously we defined $G[x,y)$, $G(y,x]$, etc. for points $x,y\in \mathbb{D}$. Here we extend the definition and allow for $x$ and $y$ to be ideal boundary points: if $x$ is a boundary point, $G(x,y)$ will be the set of geodesics transversely intersecting the open geodesic connecting $x$ to $y$, and similar for the other notations. Using these extended definitions, we extend the original pseudo-distance $\Tilde{d}_\mu^b$ on $\mathbb{D}^b$ to $\mathbb{D}^b\cup S^1$ by allowing $\Tilde{d}_\mu^b(\cdot,\cdot)$ to take boundary points as input and to output $+\infty$.  It is easily verified that this extension is still symmetric and satisfies the triangle inequality.
\begin{lem}\label{lem: injectivity lemma}
    Let $p,q\in \Lambda_\Gamma^c$. 
    \begin{enumerate}
        \item     If $\Tilde{d}_\mu^b(p,q)<\infty$, then $\beta_\mu(p)$ and $\beta_\mu(q)$ lie in $\overline{T}_\mu$, and $d_\mu(\beta_\mu(p),\beta_\mu(q))=\Tilde{d}_\mu^b(p,q)$.
        \item If $\Tilde{d}_\mu^b(p,q)=\infty$, then $\beta_\mu(p)\neq \beta_\mu(q)$.
    \end{enumerate}
  In particular, $\beta_\mu(p)=\beta_\mu(q)$ if and only if $\Tilde{d}_\mu^b(p,q)=0.$
\end{lem}
\begin{proof}
    Assume $p\neq q$, otherwise there is nothing to do. We are working with a basepoint $o\in \mathbb{D}^b$. We observe that 
    \begin{equation}\label{eq: symmetric difference}
        A(p)\Delta A(q) = G(p,q)\cup N,
    \end{equation}
where $N$ is the null set of geodesics with an endpoint at $p$ or $q$. The set $A(p)\cap A(q)$ consists of geodesics separating $o$ from both $p$ and $q$. Since $p\neq q$, this is a relatively compact subset of the space of geodesics (its endpoint pairs remain inside two distinct closed boundary arcs). Thus, since $\mu$ is a Radon measure, $\mu(A(p)\cap A(q))<\infty$. If $\mu(G(p,q))<\infty$, then $\mu(A(p)\cap A(q))<\infty$ together with (\ref{eq: symmetric difference}) imply that both $A(p)$ and $A(q)$ have finite measure, hence both $\beta_\mu(p)$ and $\beta_\mu(q)$ live in $\overline{T}_\mu$. By (\ref{eq: symmetric difference}), $d_\mu(\beta_\mu(p),\beta_\mu(q))=\Tilde{d}_\mu^b(p,q)$.

 If $\beta_\mu(p)= \beta_\mu(q)$, then their geodesic rays in $\overline{T}_\mu$ starting at the common point $o$ would coincide, because $T_\mu$ is an $\R$-tree (note uniqueness of rays with the same starting and end points holds even if the endpoint is in the Gromov boundary). By (\ref{eq: symmetric difference}), this would imply that $\Tilde{d}_\mu^b(p,q)=0$, and hence we cannot have $\Tilde{d}_\mu^b(p,q)=\infty$.
\end{proof}
\begin{remark}
    We do not make the claim that $\pi_\mu$ extends continuously to $\beta_\mu$ on $\mathbb{D}^b\cup \Lambda_\Gamma^c$. If a sequence $(p_n)_{n=1}^\infty\subset \mathbb{D}^b$ converges tangentially to $p$, then we generally do not expect $\pi_\mu(p_n)$ to converge to $\beta_\mu(p)$.
\end{remark}

\subsection{Korevaar-Schoen trace maps}\label{subsec: sobolev trace maps}
Now, keep $\mu$ and all the notations from above, and let $\mathcal{F}$ be an integrable partial measured foliation tightening to $\mu$. We assume that the map $\pi_{\mathcal{F}}$ can be globally defined on $\mathbb{D}$. The map $\pi_{\mathcal{F}}$ extends to $\Lambda_\Gamma^c$ in its own way. Fix $o\in \mathbb{D}$ (note we now don't need $\mathbb{D}^b$) and for $p\in \Lambda_\Gamma^c$, let $r_p$ be as above.
\begin{prop}\label{prop: non-tangntial limit}
For almost every $p\in \Lambda_\Gamma^c$, the limit $$\textrm{tr}\pi_{\mathcal{F}}(p):=\lim_{t\to\infty}\pi_{\mathcal{F}}\circ r_p(t)$$ exists in $(\overline{T}_\mu,d_\mu).$ The assignment $p\to \textrm{tr}\pi_{\mathcal{F}}(p)$ determines a measurable and locally $L^2$ map $$\textrm{tr}\pi_{\mathcal{F}}: \Lambda_\Gamma^c\to (\overline{T}_\mu,d_\mu),$$ which is independent of the chosen point $o$. 
\end{prop}
Here, $L^2$ is in the sense of \cite{KorevaarSchoen}: for a function $u$ from a Riemannian manifold to a metric space $(M,d)$, it means that for any point $Q\in M$, the real-valued function $x\mapsto d(u(x),Q)$ is $L^2$. Using $S^1$ as the source, the measure is just the Lebesgue measure. In the proof below, we will apply the metric space version of the classical trace theorem for Sobolev spaces, \cite[Theorem 1.12.1]{KorevaarSchoen}. This result applies to functions that are $W^{1,2}$ in the sense of \cite{KorevaarSchoen} (not just $W_{\textrm{loc}}^{1,2}$), which means they are $L^2$ and have finite total energy.

\begin{proof}[Proof of Proposition \ref{prop: non-tangntial limit}]
The complement of the convex hull of $\Lambda_\Gamma$ in $\mathbb{D}$ decomposes into a union of half-planes corresponding to geodesic half-planes on $X$ and fundamental domains of geodesic funnels in $X$. Let $H$ be one such half-plane or fundamental domain, with ideal boundary $\partial_\infty H\subset S^1$. Equip $H\cup \partial_\infty H$ with the flat metric inherited from $\overline{\mathbb{D}}$. We claim that $\pi_{\mathcal{F}}$ restricts on $H$ to a $W^{1,2}$ function. Since $\pi_{\mathcal{F}}$ has finite (equivariant) energy, we just need to show that for any $Q\in\pi_{\mathcal{F}}(H)$, $h(x)= d(\pi_{\mathcal{F}}(x),Q)$ is $L^2$ on $H$ (with respect to the flat metric). Corollary 16.3 from \cite{KorevaarSchoen} shows that the Dirichlet energy of $h$ is, up to a uniform constant, pointwise bounded above by that of $\pi_{\mathcal{F}}$. Thus, since $\pi_{\mathcal{F}}$ has finite energy on $H$, the same is true for $h$. The result now follows from a standard extension of the Poincar{\'e} inequality: for any relatively compact subset $B\subset H$, $$||h||_{L^2(H)}\leq C ( E(h|_{H})^{1/2}+ ||h||_{L^2(B)})<\infty.$$
With the $W^{1,2}$ property established, we are prepared to set up \cite[Theorem 1.12.1]{KorevaarSchoen}. In $H\cup \partial_\infty H$, let
$t\mapsto \chi(\cdot,t)$ be a flow satisfying $\chi(x,0)=x\in \partial_\infty H$, preserving every truncated ray $r_p([0,\infty))\cap H$, and with infinitesimal generator transverse to $\partial_\infty H$. By \cite[Theorem 1.12.1]{KorevaarSchoen}, the limit of $\pi_{\mathcal{F}}$ along the flow exists at almost every boundary point $p\in \partial_\infty H$ and determines an $L^2$ measurable map $\textrm{tr}\pi_{\mathcal{F}}$ on $\partial_\infty H$. It follows that, at almost every $p\in \partial_\infty H$, $\lim_{t\to\infty}\pi_{\mathcal{F}}\circ r_p(t)= \textrm{tr}\pi_{\mathcal{F}}(p)$. By uniqueness in the trace theorem, this measurable map is the same if we change our original point $o$.
To obtain the global map, we patch extensions together over every component of the complement of the convex hull. 
\end{proof}
   \begin{remark}\label{rem: non-tangential limits}
   When $\pi_{\mathcal{F}}$ is harmonic, the best type of convergence one can expect is for $\pi_\mathcal{F}$ to converge to $\textrm{tr}\pi_{\mathcal{F}}$ non-tangentially almost everywhere. The assertion that $\pi_\mathcal{F}\to \textrm{tr}\pi_{\mathcal{F}}$ non-tangentially a.e. is not contained in \cite{KorevaarSchoen}, but it is expected. This is completely analogous to the situation for finite energy harmonic functions from $\mathbb{D}$ to $\mathbb{R}$, where the non-tangential limits are determined by the boundary behaviour of Poisson integrals.
   
   For an example, the height function on a comb domain $\Omega$ with infinitely many vertical teeth gives a harmonic function that does not extend continuously to its trace map. Pulling back $dz^2$ to the unit disk $\mathbb{D}$ via the Riemann map gives a quadratic differential with a foliation whose leaf-space projection does not extend continuously to the boundary.
\end{remark}

We now have two boundary maps, one associated with $\mu$ and one associated with $\mathcal{F}$. Since $\mu$ and $\mathcal{F}$ give the same endpoints on $S^1$, we expect the boundary maps to be the same in some sense. This is indeed the case.

\begin{prop}\label{prop: trace map agrees}
    For almost every $p\in \Lambda_\Gamma^c$, $\textrm{tr}\pi_{\mathcal{F}}(p)=\beta_\mu(p).$
\end{prop}
Implicitly, for the points $p$ under consideration, $\beta_\mu(p)$ is in $\overline{T}_\mu$, and not the Gromov boundary.
\begin{proof}
    Fix a basepoint $o\in \mathbb{D}$ and consider rays $r_p$ based at $o$ as above. Take a full measure subset such that
$$\lim_{t\to\infty}\pi_{\mathcal{F}}(r_p(t))=\textrm{tr}\pi_{\mathcal{F}}(p).$$ Assume for the sake of contradiction that, for some $p$ inside this subset, $\textrm{tr}\pi_{\mathcal{F}}(p)\neq \beta_\mu(p)$. Let $J\subset \overline{T}_\mu$ be the geodesic segment or ray joining $\textrm{tr}\pi_{\mathcal{F}}(p)$ and $\beta_\mu(p).$ Choose a point $y\in J$, represented by a leaf $\ell$ in $\mathbb{D}$, such that $y$ separates $\overline{T}_\mu\cup \partial_\infty \overline{T}_\mu$ into two components $T_1$ and $T_2$ and such that $p$ is not an endpoint of $\ell$; for the first condition, we just need $p$ to lie on a regular leaf, and the second condition is possible by Lemma \ref{lem:asymptotic} and because $J$ corresponds to a positive measure subset of leaves. Let $g_\ell$ be the geodesic tightening of $\ell$. $\ell$ cuts $\mathbb{D}$ into two subsets $C_1$ and $C_2$ and $g_\ell$ cuts $\mathbb{D}$ into two components $D_1$ and $D_2$, with the relations $\partial_\infty C_1=\partial_\infty D_1$, and $\partial_\infty C_2=\partial_\infty D_2$. It follows, up to relabeling $T_1$ and $T_2$, that $\pi_{\mathcal{F}}$ maps $C_1$ (resp. $C_2$) to $T_1\cup \{y\}$ (resp. $T_2 \cup \{y\}$) and $\pi_{\mu}$ maps $D_1$ (resp. $D_2$) to $T_1\cup \{y\}$ (resp. $T_2 \cup \{y\}$). Assume the labellings are such that $p\in \partial_\infty C_1=\partial_\infty D_1$. Then, for $t$ large, $r_p(t)\in C_1\cap D_1$. But then both $\pi_\mu(r_p(t))$ and $\pi_{\mathcal{F}}(r_p(t))$ lie in the same component of $T_{\mu}\backslash \{y\}$ for $t$ large. Taking $t\to \infty$ shows that $\textrm{tr}\pi_{\mathcal{F}}(p)$ and $\beta_\mu(p)$ lie in the same component of $\overline{T}_\mu\cup \partial_\infty \overline{T}_\mu\backslash \{y\},$ which gives a contradiction.
\end{proof}

Even though the boundary map $\beta_\mu$ is defined on all of $\Lambda_\Gamma^c$ as a map into $\overline{T}_\mu\cup \partial_\infty \overline{T}_\mu$, we will say that $\textrm{tr}\pi_{\mathcal{F}}$, a measurable map to $\overline{T}_\mu$, is \textit{represented by $\beta_\mu$}.

\subsection{Asymptotic Plateau problem}\label{subsec: asymptotic plateau}
Let $\mu$ and $\nu$ be two measured geodesic laminations obtained by tightening foliations on $X$ with finite Dirichlet integral. As in \S \ref{subsec: the core}, we denote the corresponding product of $\R$-trees $(T_\mu,d_\mu)\times (T_\nu,d_\nu)$ by $(P,d)$. The completion is denoted $(\overline{P},d)$. 
We identify $\overline{P}$ with its image inside the topological space $(\overline{T}_\mu\cup \partial_\infty \overline{T}_\mu)\times (\overline{T}_\nu\cup \partial_\infty \overline{T}_\nu)$, which we write as $\overline{P}\cup \partial_\infty \overline{P}$, so that $\partial_\infty \overline{P}$ denotes the complement of the image of $\overline{P}$. The space $(P,d)$ comes with a product representation $\rho=(\rho_\mu,\rho_\nu):\pi_1(X)\to \textrm{Isom}(P,d)$, which extends to $(\overline{P},d)$ and to $\overline{P}\cup \partial_\infty \overline{P}$. On $\Lambda_\Gamma^c$, we have the product of canonical maps $\beta=(\beta_\mu,\beta_\nu)$. This map is $\rho$-equivariant, where the $\pi_1(X)$-action on $S^1$ is the action induced by the deck group action on $\mathbb{D}$. For a homeomorphism between Riemann surfaces $f:X\to Y$, we write $f_*$ for the induced isomorphism of the fundamental groups. 

\begin{defn}\label{def: asymptotic Plateau}
    The asymptotic Plateau problem for $(\mu,\nu)$ asks about the existence and uniqueness of a homeomorphism of Riemann surfaces $f:X\to Y$, extending to a homeomorphism $\partial_\infty f:S^1\to \partial_\infty \Tilde{Y}$, together with a $\rho\circ f_*^{-1}$-equivariant conformal harmonic map $\Pi: \Tilde{Y}\to P$ with trace map represented by $\beta\circ (\partial_\infty f)^{-1}$. We say that $(f,Y,\Pi)$ is a solution to the asymptotic Plateau problem.
\end{defn}
Above, uniqueness means that if $(g,Z,\Theta)$ is any other solution, then there are homeomorphisms $f:X\to Y$ and $g:X\to Z$ such that $g\circ f^{-1}:Y\to Z$ is homotopic to a biholomorphism whose lift $\mathbb{D}\to\mathbb{D}$ intertwines the two maps $\Pi$ and $\Theta$. Theorem \ref{thm: main} shows that filling is equivalent to the existence of a solution. 

The asymptotic Plateau problem is actually more delicate than Theorem \ref{thm: main}. The reason is that if we have a harmonic map to a tree with given equivariance and trace map $\beta\circ (\partial_\infty f)^{-1}$, it is not clear a priori that this harmonic map is isometric to the leaf-space projection of its Hopf differential (there may be a ``folding"). Nevertheless, we're able to establish results with relative ease for a large class of Riemann surfaces.
\begin{defn}
 We say that $X$ is of Dirichlet-Liouville type if every non-negative subharmonic function with finite Dirichlet energy and constant trace map on the ideal boundary is constant.
\end{defn}
Surfaces of Dirichlet-Liouville type include: the unit disk, all parabolic Riemann surfaces, and most generally Riemann surfaces of the second kind such that the ends of the convex core resemble that of a parabolic Riemann surface.
\begin{lem}\label{lem: uniqueness for harmonic maps}
    Let $X=\mathbb{D}/\Gamma$ be of Dirichlet-Liouville type, $(T,d)$ an $\R$-tree with action $\rho:\pi_1(X)\to \textrm{Isom}(T,d)$, and let $f,g$ be two finite energy $\rho$-equivariant harmonic maps from $\mathbb{D}$ to $(M,d)$ with the same trace map. Then $f=g$.
\end{lem}
Here, the trace map is interpreted in the same sense as Proposition \ref{prop: non-tangntial limit}. Note that the lemma goes through just fine if we replace $(T,d)$ with any complete NPC metric space, one just has to define the trace map analogously to Proposition \ref{prop: non-tangntial limit}.
\begin{proof}
   As in the proof of Proposition \ref{prop: non-tangntial limit}, if we fix a point $Q\in T$, both functions $x\mapsto d(f(x),Q)$ and $x\mapsto d(g(x),Q)$ have finite Dirichlet energy on a fundamental domain for $\Gamma$. By equivariance, the function $\Tilde{F}(x)=d(f(x),g(x))$ descends to a function $F$ on $X$. By the triangle inequality, $F$ has finite Dirichlet energy. Since harmonic maps pull back germs of convex functions to subharmonic functions, $F$ is subharmonic. Since $f$ and $g$ have the same trace map, the trace map of $F$ is zero, hence $F$ is identically zero.
\end{proof}

\begin{thm}\label{thm: asymptotic plateau existence}
 Let $\mu$ and $\nu$ be measured geodesic laminations realized by tightening foliations on the Riemann surface $X=\mathbb{D}/\Gamma$ with finite Dirichlet integral. 
 
If $\mu$ and $\nu$ are filling, the map $\beta:\Lambda_\Gamma^c\to \overline{P}\cup \partial_\infty \overline{P}$ is injective and the asymptotic Plateau problem admits a solution. If the asymptotic Plateau problem admits a solution $(f,Y,\Pi)$ such that $Y$ is of Dirichlet-Liouville type, then $\mu$ and $\nu$ are filling.
\end{thm}
We note that in the case $X=\mathbb{D}$, the proof below shows that filling is equivalent to the injectivity of $\beta$. In general, $\beta$ being injective implies that $i(\mu,\eta)+i(\nu,\eta)>0$ for every cross-cut $\eta$.
\begin{proof}[Proof of Theorem \ref{thm: asymptotic plateau existence}]
    Let $x,y\in \Lambda_\Gamma^c$ and let $\gamma$ be a cross-cut on $X$ whose lift to $\mathbb{D}$ has endpoints $x$ and $y$. Then, the distance between points $\beta(x)$ and $\beta(y)$ in $\beta(\Lambda_\Gamma^c)$ (which can be infinite) is exactly $$\sqrt{i(\mu,\gamma)^2+i(\nu,\gamma)^2}=\sqrt{\mu(G(x,y))^2+\nu(G(x,y))^2}.$$ Thus, by Lemma \ref{lem: injectivity lemma}, filling implies that $\beta$ is injective.

    If $(\mu,\nu)$ is filling, then by Theorem \ref{thm: main} there exists a holomorphic realization $(f,Y,q)$.
    As in \S \ref{subsec: the core}, $q$ induces an equivariant conformal harmonic map $\Pi_q:\Tilde{Y}\to P$. By Proposition \ref{prop: trace map agrees}, the trace map is represented by $\beta \circ (\partial_\infty f)^{-1}$. Conversely, if we have a Plateau problem solution $(f,Y,\Pi)$ and $Y$ is of Dirichlet-Liouville type, then by Proposition \ref{prop: Dirichlet principle}, Proposition \ref{prop: trace map agrees}, and Lemma \ref{lem: uniqueness for harmonic maps}, the Hopf differentials of the components of $\Pi$ produce a holomorphic realization of $(\partial_\infty f)_*\mu$ and $(\partial_\infty f)_*\nu$. Hence, by the necessity part of Theorem \ref{thm: main}, $(\partial_\infty f)_*\mu$ and $(\partial_\infty f)_*\nu$ are filling. Pulling back to $X$, we deduce that $\mu$ and $\nu$ are filling. 
\end{proof}

\begin{thm}\label{thm: asymptotic plateau uniqueness}
 Let $X$, $\mu$, and $\nu$ be as in Theorem \ref{thm: asymptotic plateau existence}. Then, there exists at most one asymptotic Plateau solution $(f,Y,\Pi)$ such that $Y$ is of Dirichlet-Liouville type. Moreover, $\Pi$ is energy minimizing in the sense that for any other surface $Z$ of Dirichlet-Liouville type and equivariant map $\Theta:\Tilde{Z}\to P$ with trace map of the form $\beta\circ \eta^{-1}$, where $\eta$ is the extension of a homeomorphism as above, $E(Y,\Pi)\leq E(Z,\Theta).$
\end{thm}
\begin{proof}
    Let $(g,Z,\Theta)$ be another asymptotic Plateau solution with $Z$ of Dirichlet-Liouville type. Then, by Proposition \ref{prop: trace map agrees} and Lemma \ref{lem: uniqueness for harmonic maps}, $\Theta$ is the product of leaf-space projections for the holomorphic quadratic differentials on $Z$ realizing the filling pair $(g_*\mu,g_*\nu)$. Since $\Theta$ is conformal, the Hopf differentials of its projections on $T_\mu$ and $T_\nu$ are of the form $r$ and $-r$, for some $L^1$ quadratic differential $r$. Thus, $(g,Z,r)$ is another holomorphic realization of $(\mu,\nu)$, and the uniqueness follows by the uniqueness statement from Theorem \ref{thm: main}. 
\end{proof}

\begin{remark}
   For Riemann surfaces of Dirichlet-Liouville type, Lemma \ref{lem: uniqueness for harmonic maps} gives a new proof of the uniqueness of the heights theorem for quadratic differentials (see \cite{Saric22}).
\end{remark}

\section{Main inequalities, harmonic maps, and minimal surfaces}\label{sec: main inequality}
For a map $f:X\to (M,d)$, in order to properly distinguish the domain, we write energy as $E(X,f):=E(f)$. If $f:\Tilde{X}\to (M,d)$ is equivariant, we write $E(X,f)$ (since integration is done over $X$).
\subsection{Main inequalities}
Theorem \ref{thm: new main} follows swiftly from the energy minimizing property in Theorem \ref{thm: main} and the Reich-Strebel energy formula, which we now introduce. Recall that the Beltrami form of a locally quasiconformal map between Riemann surfaces $f:X\to Y$ is the measurable tensor $$\mu_f(z)=\frac{f_{\overline{z}}(z)d\overline{z}}{f_z(z)dz},$$ where derivatives are taken in the weak sense. 
Let $(M,d)$ be a complete and non-positively curved geodesic metric space, and $h:\tilde{X}\to (M,d)$ a $W_{\textrm{loc}}^{1,2}$ equivariant map such that $E(X,h)<\infty$. Let $J_{f^{-1}}$ be the Jacobian of $f^{-1}$ and $q(h)$ the Hopf differential of $h$, which need not be holomorphic. Lifting $f$ to $\Tilde{f}:\Tilde{X}\to \Tilde{Y}$, one can verify the identity of functions on $X$,
\begin{align*}
     e(h\circ \Tilde{f}^{-1})&=(e(h)\circ f^{-1})J_{f^{-1}}+2(e(h)\circ f^{-1})J_{f^{-1}} \frac{(|\mu_f|^2\circ f^{-1})}{1-(|\mu_f|^2\circ f^{-1})} \\
     &-4\textrm{Re}\Big ( (q(h)\circ f^{-1})J_{f^{-1}}\frac{(\mu_f\circ f^{-1})}{1-(|\mu_f|^2\circ f^{-1})}\Big )
\end{align*}
from \cite{ReichStrebelGerstenhaberRauch} (note we've descended $e(h)$, $q(h)$, etc. to $X$) . Integrating, we obtain the proposition below.
\begin{prop}\label{prop:RS} We have the formula
\begin{equation}\label{RSorig}
    E(Y,h\circ \Tilde{f}^{-1}) -E(X,h) =  \int_X -4\textrm{Re} \Big (q(h)\cdot \frac{ \mu_f}{1-|\mu_f|^2} \Big )+ 2 e(h)\cdot \frac{|\mu_f|^2}{1-|\mu_f|^2}.
\end{equation}
\end{prop}

We now prove the new main inequality. The analogous result is proved for closed surfaces in \cite[\S 3.1]{MarkovicSagmanMainInequality}.
\begin{proof}[Proof of Theorem \ref{thm: new main}]
       Let $\Pi_q=(\pi_q,\pi_{-q})$ be the product of the harmonic leaf-space projections associated with $q$ and let $\mu$ and $\nu$ be the laminations arising from tightening the foliations $\mathcal{F}_h(q)$ and $\mathcal{F}_v(q)$ respectively. Let $f_1,f_2:X\to Y$ be as in the statement of the theorem, with common boundary extension $m$. The maps $f_1$ and $f_2$ determine pushforward partial measured foliations
       $(f_1)_*\mathcal{F}_h(q)$ and $(f_2)_*\mathcal{F}_v(q)$ that tighten to $m_*\mu$ and $m_*\nu$ respectively. By the energy minimizing property from Theorem \ref{thm: main}, $$\mathcal{D}(\mathcal{F}_h(q))+\mathcal{D}(\mathcal{F}_v(q))\leq \mathcal{D}((f_1)_*\mathcal{F}_h(q))+\mathcal{D}((f_2)_*\mathcal{F}_v(q)).$$
       Since the Dirichlet energy of a proper partial measured foliation is equal to half the Dirichlet energy of the corresponding map to the dual $\R$-tree, the inequality above is equivalent to $$E(X,\Pi_q)\leq E(Y,\pi_q\circ f_1^{-1})+E(Y,\pi_{-q}\circ f_2^{-1}).$$ Splitting the energy of the product maps as sums of energies, substituting in the formula (\ref{RSorig}), and then rearranging, we obtain the new main inequality.
\end{proof}

\subsection{Harmonic maps and minimal immersions}

\subsubsection{Harmonic diffeomorphism}
Following Theorem \ref{thm: new main}, the proof of Corollary \ref{cor: uniqueness_harmonic} is verbatim the same as that of Theorem 3 from \cite{MarkovicMateljevic}, once we replace the disk with an arbitrary Riemann surface. We don't provide the complete proof, but just give a sketch, referring the reader to \cite{MarkovicMateljevic} for more details.

\begin{proof}[Proof sketch of Corollary \ref{cor: uniqueness_harmonic}]
    Let $f,g:X\to (Y,\sigma)$ be the harmonic diffeomorphisms as in the statement of the corollary and let $q$ be the Hopf differential of $f$. We denote the Beltrami forms by $\mu_f$ and $\mu_g$ respectively. The lifts to the universal covers $\Tilde{f}$ and $\Tilde{g}$ extend to the same map on the ideal boundary (they agree on the limit set from the argument in \cite[Theorem 3.6]{Saric}, and they agree on the complement of the limit set by our assumption). Thus, $\Tilde{g}^{-1}\circ \Tilde{f}$ extends to the identity map on $\partial_\infty \Tilde{X}$. Since $g^{-1}\circ f$ is also a locally quasiconformal homeomorphism, Corollary \ref{cor: generalized main} can be applied to it.  Corollary \ref{cor: generalized main} and Remark \ref{rem: main other form} imply
\begin{equation}\label{eq: lemma 1}
    \int_X |q| \leq \int_X |q|\cdot\frac{1-|\mu_f|}{1+|\mu_f|}\cdot \frac{1+|\mu_g|}{1-|\mu_g|}
\end{equation}
and 
$$\int_X \sigma \cdot \frac{|\mu_f|}{1-|\mu_f|^2}\leq\int_X \sigma \cdot \frac{|\mu_f|}{1-|\mu_f|^2} \cdot \frac{1-|\mu_f|}{1+|\mu_f|}\cdot \frac{1+|\mu_g|}{1-|\mu|_g}.$$
Indeed, see the pointwise calculations from the proofs of Lemmas 1 and 2 from \S E4 of \cite{MarkovicMateljevic}. Reversing the roles of $f$ and $g$, and then summing the corresponding inequalities, one obtains 
\begin{equation}\label{eq: E5}
    \int_X \alpha \sigma \leq \int_X \beta \sigma,
\end{equation}
where $$\alpha = \frac{|\mu_f|}{1-|\mu_f|^2}+\frac{|\mu_g|}{1-|\mu_g|^2}, \hspace{1mm} \beta = \frac{|\mu_f|}{(1+|\mu_f|)^2}\frac{1+|\mu_g|}{1-|\mu_g|}+ \frac{|\mu_g|}{(1+|\mu_g|)^2}\frac{1+|\mu_f|}{1-|\mu_f|}.$$
After lifting to the universal covers, Lemma 3 from \cite{MarkovicMateljevic} directly shows that $\alpha\geq \beta$. Combining this with (\ref{eq: E5}) yields $\alpha=\beta$, which implies $|\mu_f|=|\mu_g|$. This means we have equality in (\ref{eq: lemma 1}). By stepping into the calculations used to derive (\ref{eq: lemma 1}) (see the proof of Lemma 1 from \cite{MarkovicMateljevic}), one arrives at $\mu_f=\mu_g$. Thus, $g^{-1}\circ f$ is conformal, and since $\Tilde{g}^{-1}\circ \Tilde{f}$ extends to the identity on $\partial_\infty \Tilde{X}$, we conclude that $f=g$.
\end{proof}

\subsubsection{Minimal immersions}
Similar to above, the proof of Corollary \ref{cor: minimal} is nearly identical to the proof of Theorem A from \cite{Sagman} (see also \cite{MarkovicMinimalDiffeomorphisms}). We include a proof so that we can point out where the quasiconformal assumption is used and can be weakened.

\begin{proof}[Proof of Corollary \ref{cor: minimal}]
    Let $(h_1,h_2)$ and $(g_1,g_2)$ be two minimal immersions to $(Y\times Z,\sigma_1\oplus\sigma_2)$ as in the statement of the corollary, with $\partial_\infty h_2\circ (\partial_\infty h_1)^{-1}=\partial_\infty g_2\circ (\partial_\infty g_1)^{-1}=m$. 
For $i=1,2$, set $f_i = g_i^{-1}\circ h_i.$ By our relation on the boundary maps, the extensions of $f_1$ and $f_2$ to $\partial_\infty Y$ agree. Similar to the proof above, when we take the universal lifts, we get the same map on the ideal boundary. Therefore, the new main inequality can be applied on the Beltrami forms $\mu_i$ of $f_i$. Suppose that for at least one $i$, $f_i$ is not the identity. Let $q=q_1$ be the Hopf differential of $h_1$. The Hopf differential of $h_2$ is $-q=q_2$.

Let $(X_n)_{n=1}^\infty\subset X$ be an exhaustion of $X$ by compact subsurfaces with boundary. Applying the Reich-Strebel formula (\ref{RSorig}) over the subsurface $X_n,$
$$\sum_{i=1}^2 E(f_i(X_n),g_i)-\sum_{i=1}^2E(X_n,h_i) = \sum_{i=1}^2 -4\textrm{Re}\int_{X_n}q_i \frac{\mu_i}{1-|\mu_i|^2}+ \sum_{i=1}^2 2\int_{X_n}e(h_i) \frac{|\mu_i|^2}{1-|\mu_i|^2}.$$
Since $e(h)>|q(h)|$ (a trivial application of Cauchy-Schwarz), the rightmost side strictly dominates 
$$\sum_{i=1}^2 -4\textrm{Re}\int_{X_n}q_i \frac{\mu_i}{1-|\mu_i|^2}+ \sum_{i=1}^2 4\int_{X_n}|q_i| \frac{|\mu_i|^2}{1-|\mu_i|^2}.$$
The difference between the last two expressions is the integration of a positive function over $X_n$, hence increasing with $n.$ Consequently, there exists $N>0$ such that once $n>N$, there is an $\epsilon>0$ such that $$\sum_{i=1}^2 \mathcal{E}(f_i(X_n),g_i)-\sum_{i=1}^2\mathcal{E}(X_n,h_i)\geq \sum_{i=1}^2 -4\textrm{Re}\int_{X_n}q_i \frac{\mu_i}{1-|\mu_i|^2}+ \sum_{i=1}^2 4\int_{X_n}|q_i| \frac{|\mu_i|^2}{1-|\mu_i|^2}+\epsilon.$$ Since each $q_i$ is integrable and $\mu_i$ has $L^\infty$ norm strictly less than $1$ (since $f_i$ is quasiconformal), an application of dominated convergence yields $$\sum_{i=1}^2 4\textrm{Re}\int_{X_n}q_i \frac{\mu_i}{1-|\mu_i|^2}\to \sum_{i=1}^2 4\textrm{Re}\int_{X}q_i \frac{\mu_i}{1-|\mu_i|^2}.$$ Similarly, $$\sum_{i=1}^2 4\int_{X_n}|q_i| \frac{|\mu_i|^2}{1-|\mu_i|^2}\to \sum_{i=1}^2 4\int_{X}|q_i| \frac{|\mu_i|^2}{1-|\mu_i|^2}$$ as $n\to \infty$. By the new main inequality, i.e., Theorem \ref{thm: new main}, we obtain that $$\sum_{i=1}^2 E(f_i(X_n),g_i)-\sum_{i=1}^2E(X_n,h_i)\geq \epsilon$$ for all $n>N$.  

Reversing the roles of $h_i$ and $g_i$, we repeat the argument to get a contradictory inequality. If $r_i$ is the Hopf differential of $g_i$ and $\nu_i$ is the Beltrami form of $f_i^{-1}$, the reasoning from above gives that there exist $M,\delta>0$ such that for $n>M$,
$$\sum_{i=1}^2 E(X_n,h_i)-\sum_{i=1}^2E(f_i(X_n),g_i)\geq \sum_{i=1}^2 -4\textrm{Re}\int_{f_i(X_n)}r_i \frac{\nu_i}{1-|\nu_i|^2}+ \sum_{i=1}^2 4\int_{f_i(X_n)}|r_i| \frac{|\nu_i|^2}{1-|\nu_i|^2} +\delta.$$
Using that $f_i$ is quasiconformal once again, dominated convergence yields $$\sum_{i=1}^2 4\textrm{Re}\int_{f_i(X_n)}r_i \frac{\nu_i}{1-|\nu_i|^2}\to \sum_{i=1}^2 4\textrm{Re}\int_{X}r_i \frac{\nu_i}{1-|\nu_i|^2}$$ as $n\to \infty,$ and likewise for the other term. We deduce that
$$\sum_{i=1}^2 E(X_n,h_i)-\sum_{i=1}^2E(f_i(X_n),g_i)\geq \delta$$ for $n>M$, which gives the contradiction. We conclude that both $f_i$ are the identity map, and hence $(h_1,h_2)=(g_1,g_2).$
\end{proof}

\begin{remark}\label{rem: weaken quasiconformal}
    Of course, the quasiconformal assumption was not needed for invoking the new main inequality. The assumption was used to justify that $$\sum_{i=1}^2 4\textrm{Re}\int_{X_n}q_i \frac{\mu_i}{1-|\mu_i|^2}\to \sum_{i=1}^2 4\textrm{Re}\int_{X}q_i \frac{\mu_i}{1-|\mu_i|^2}$$ and $$\sum_{i=1}^2 4\textrm{Re}\int_{f_i(X_n)}r_i \frac{\nu_i}{1-|\nu_i|^2}\to \sum_{i=1}^2 4\textrm{Re}\int_{X}r_i \frac{\nu_i}{1-|\nu_i|^2}$$ as $n\to \infty$. If the maps are not quasiconformal, but $|q_i| \frac{|\mu_i|}{1-|\mu_i|^2}$ and $|r_i| \frac{|\nu_i|}{1-|\nu_i|^2}$ are integrable (note $|\mu_i|=|\nu_i|$), then the proofs go through. 
\end{remark}

\bibliographystyle{plain}

\begin{thebibliography}{Thua}

\vskip .5cm




 \bibitem{AndersonCompleteMinimalHypersurfaces}
Michael Anderson, \emph{Complete minimal hypersurfaces in hyperbolic $n$-manifolds},
Comment. Math. Helv. \textbf{58} (1983), 264--290.
 

 
 \bibitem{BS} Ara Basmajian and Dragomir \v Sari\' c,  {\it Geodesically complete hyperbolic structures.} Math. Proc. Cambridge Philos. Soc. 166 (2019), no. 2, 219-242. 


 \bibitem{BonsanteSchlenker}
Francesco Bonsante and Jean-Marc Schlenker, \emph{Maximal surfaces and the universal Teichmüller space}, Invent. Math. \textbf{182} (2010), no.~2, 279--333.

\bibitem{buser} Peter Buser, {\it Geometry and spectra of compact Riemann surfaces.} Reprint of the 1992 edition. Modern Birkh\"auser Classics. Birkh\"auser Boston, Ltd., Boston, MA, 2010. 

\bibitem{CantwellConlonFenley}
John Cantwell, Lawrence Conlon, and Sergio R. Fenley,
\emph{Endperiodic automorphisms of surfaces and foliations},
Ergodic Theory Dynam. Systems \textbf{41} (2021), no.~1, 66--212.

\bibitem{Didac}
Luca De Rosa and Didac Martinez-Granado,
\emph{Dual spaces of geodesic currents},
J. Topol. \textbf{18} (2025), no.~4, e70045.

\bibitem{EvansProbabilityRealTrees}
S.N. Evans, \emph{Probability and real trees},
Lecture Notes in Mathematics, vol.~1920, Springer, Berlin, 2008.



\bibitem{GLPartial} Frederick P. Gardiner and Nikola Lakic, {\it A synopsis of the Dirichlet principle for measured foliations.} Infinite dimensional Teichm\" uller spaces and moduli spaces, 55–64, RIMS K\^{o}ky\^{u}roku Bessatsu, B17, Res. Inst. Math. Sci. (RIMS), Kyoto, 2010.

\bibitem{GarMas} Frederick P. Gardiner and Howard Masur,
{\it Extremal length geometry of Teichm\" uller space.}
Complex Variables Theory Appl. 16 (1991), no. 2-3, 209–237. 

\bibitem{GS} Frederick P. Gardiner and Dennis P. Sullivan, {\it Symmetric structures on a closed curve.} Amer. J. Math. 114 (1992), no. 4, 683-736.

\bibitem{GuirardelCore}
Vincent Guirardel, \emph{C\oe ur et nombre d'intersection pour les actions de groupes sur les arbres},
Ann. Sci. 'Ec. Norm. Sup'er. (4) \textbf{38} (2005), no.~6, 847--888.

\bibitem{HakSar}  Hrant Hakobyan and Dragomir \v Sari\' c {\it Limits of Teichm\"uller geodesics in the universal Teichm\"uller space.} Proc. Lond. Math. Soc. (3) 116 (2018), no. 6, 1599–1628. 

\bibitem{HMmain} John Hubbard and Howard Masur, {\it Quadratic differentials and foliations.} Acta Math. 142 (1979), no. 3-4, 221-274. 

\bibitem{HooperGridGraphs}
W. Patrick Hooper,
\emph{Grid graphs and lattice surfaces},
Int. Math. Res. Not. IMRN \textbf{2013} (2013), no.~12, 2657--2698.




\bibitem{KapovichHyperbolicManifolds}
Michael Kapovich, \emph{Hyperbolic manifolds and discrete groups},
Progress in Mathematics, vol.~183, Birkh"auser Boston, Inc., Boston, MA, 2001.

\bibitem{Ker}  Steven P. Kerckhoff, {\it The asymptotic geometry of Teichm\" uller space.} Topology 19 (1980), no. 1, 23–41. 

\bibitem{Kirchheim}
Bernd Kirchheim, \emph{Rectifiable metric spaces: local structure and regularity of the Hausdorff measure},
Proc. Amer. Math. Soc. \textbf{121} (1994), no.~1, 113--123.

\bibitem{KorevaarSchoen}
Nicholas J. Korevaar and Richard M. Schoen, \emph{Sobolev spaces and harmonic maps for metric space targets}, Comm. Anal. Geom. \textbf{1} (1993), no.~3--4, 561--659.

\bibitem{LabourieCyclicSurfaces}
Francois Labourie,
\emph{Cyclic surfaces and Hitchin components in rank $2$},
Ann. of Math. (2) \textbf{185} (2017), no.~1, 1--58.

\bibitem{LakicMinimalNorm}
Nikola Lakic, \emph{The minimal norm property for quadratic differentials in the disk},
Michigan Math. J. \textbf{44} (1997), no.~2, 299--316.

\bibitem{LandryMinskyTaylorEndperiodic}
Michael P. Landry, Yair N. Minsky, and Samuel J. Taylor,
\emph{Endperiodic maps via pseudo-Anosov flows},
Geom. Topol. \textbf{30} (2026), no.~5, 1987--2042.



\bibitem{MardenStrebel1}  Albert Marden and Kyrt Strebel, {\it On the ends of trajectories.} Differential geometry and complex analysis, 195-204, Springer, Berlin, 1985.

\bibitem{MarkovicMinimalDiffeomorphisms}
Vladimir Markovi{\'c}, \emph{Uniqueness of minimal diffeomorphisms between surfaces}, Bull. Lond. Math. Soc. \textbf{53} (2021), no.~4, 1196--1204.

\bibitem{MarkovicMateljevic}
Vladimir Markovi{\'c} and Miodrag Mateljevi{\'c}, \emph{A new version of the main inequality and the uniqueness of harmonic maps}, J. Anal. Math. \textbf{79} (1999), 315--334.

\bibitem{MarkovicSagmanMainInequality}
Vladimir Markovi{\'c} and Nathaniel Sagman, \emph{Minimal surfaces and the new main inequality},
Ann. Fenn. Math. \textbf{49} (2024), no.~1, 99--117.

\bibitem{MarkovicSagmanSmillie}
Vladimir Markovi{\'c}, Nathaniel Sagman, and Peter Smillie,
\emph{Unstable minimal surfaces in $\mathbb{R}^n$ and in products of hyperbolic surfaces}, Comment. Math. Helv. \textbf{100} (2025), no.~1, 93--121.

\bibitem{Maskit} Bernard Maskit, {\it
Comparison of hyperbolic and extremal lengths.} 
Ann. Acad. Sci. Fenn. Ser. A I Math. 10 (1985), 381–386. 

\bibitem{MeierWenger}
Damaris Meier and Stefan Wenger,
\emph{Quasiconformal almost parametrizations of metric surfaces},
J. Eur. Math. Soc. \textbf{27} (2025), no.~12, 5133--5154.

\bibitem{Minsky}  Yair N. Minsky, {\it Teichm\"uller geodesics and ends of hyperbolic 3-manifolds.} Topology 32 (1993), no. 3, 625-647.

\bibitem{moore} Robert L. Moore, {\it Concerning upper semi-continuous collections of continua,} Trans. Amer. Math. Soc. 27 (1925), no. 4, 416-428.

\bibitem{MV}  Israel Morales, and Ferr\' an Valdez, {\it Loxodromic elements in big mapping class groups via the Hooper-Thurston-Veech construction.} Algebr. Geom. Topol. 22 (2022), no. 8, 3809-3854.



\bibitem{MorganShalenDegenerationsII}
John W. Morgan and Peter B. Shalen,
\emph{Degenerations of hyperbolic structures, II: Measured laminations in $3$-manifolds}, Ann. of Math. (2) \textbf{127} (1988), no.~2, 403--456.

\bibitem{MorganShalenFreeActions}
John W. Morgan and Peter B. Shalen,
\emph{Free actions of surface groups on $\mathbb{R}$-trees},
Topology \textbf{30} (1991), no.~2, 143--154.


\bibitem{ReichStrebelGerstenhaberRauch}
Edgar Reich and Kurt Strebel, \emph{On the Gerstenhaber--Rauch principle},
Israel J. Math. \textbf{57} (1987), no.~1, 89--100.



\bibitem{Sagman} Nathaniel Sagman,  {\it Minimal diffeomorphisms with L1 Hopf differentials.} Int. Math. Res. Not. IMRN 2024, no. 13, 10088-10103.

\bibitem{SagmanSmillie}
Nathaniel Sagman and Peter Smillie, \emph{Unstable minimal surfaces in symmetric spaces of non-compact type}, to appear in Ann. of Math. (2), arXiv:2208.04885.

\bibitem{Saric} Dragomir \v Sari\' c, {\it Train tracks and measured laminations on infinite surfaces}, Trans. Amer. Math. Soc. 374 (2021), no. 12, 8903-8947. 

\bibitem{Saric22}  Dragomir \v Sari\'c, {\it The heights theorem for infinite Riemann surfaces.} Geom. Dedicata 216 (2022), no. 3, Paper No. 33, 23 pp.

\bibitem{Saric24} Dragomir \v Sari\'c, {\it Quadratic differentials and foliations on infinite Riemann surfaces.} Duke Math. J. 173 (2024), no. 10, 1883-1930.

\bibitem{Saric25} Dragomir \v Sari\'c, {\it Quadratic differentials and function theory on Riemann surfaces}, arXiv:2407.16333.

\bibitem{SaricShima}  Dragomir \v Sari\'c and Taro Shima, {\it Intersection numbers between horizontal foliations of quadratic differentials}, to appear in Trans. Amer. Math. Soc.

\bibitem{SeppiSmithToulisse}
Andrea Seppi, Graham Smith, and Jeremy Toulisse, \emph{On complete maximal submanifolds in pseudo-hyperbolic space}, arXiv:2305.15103, 2023.

\bibitem{shiga} H. Shiga, {\it On a distance defined by the length spectrum of Teichm\"uller space}, Ann. Acad. Sci. Fenn. Math.  28  (2003),  no. 2, 315-326.

 \bibitem{Strebel} Kurt Strebel {\it Quadratic differentials}, Ergebnisse der Mathematik und ihrer Grenzgebiete (3) [Results in Mathematics and Related Areas (3)], 5. Springer-Verlag, Berlin, 1984. 



\bibitem{TrebeschiCMCHypersurfaces}
Enrico Trebeschi, \emph{Constant mean curvature hypersurfaces in anti-de Sitter space},
Int. Math. Res. Not. IMRN \textbf{2024} (2024), no.~9, 8026--8066.






\bibitem{WanCMC}
Tom Y.H. Wan,
\emph{Constant mean curvature surface, harmonic maps, and universal Teichmüller space}, J. Differential Geom. \textbf{35} (1992), no.~3, 643--657.

\bibitem{Wen} Richard A. Wentworth {\it Energy of harmonic maps and Gardiner's formula.} In the tradition of Ahlfors-Bers. IV, 221–229, Contemp. Math., 432, Amer. Math. Soc., Providence, RI, 2007. 

\bibitem{Wolf} Michael Wolf {\it On realizing measured foliations via quadratic differentials of harmonic maps to $\mathbb{R}$-trees.} J. Anal. Math. (1996), 68, 107--120


\end{thebibliography}

\end{document}